\documentclass[10pt,reqno]{amsart}
\usepackage{a4wide}

\numberwithin{equation}{section}
\usepackage{mathrsfs}
\usepackage{amsfonts}
\usepackage{amsmath}
\usepackage{stmaryrd}
\usepackage{amssymb}
\usepackage{amsthm}
\usepackage{mathrsfs}
\usepackage{url}
\usepackage{amsfonts}
\usepackage{amscd}
\usepackage{indentfirst}
\usepackage{enumerate}
\usepackage{amsmath,amsfonts,amssymb,amsthm}
\usepackage{amsmath,amssymb,amsthm,amscd}
\usepackage{graphicx,mathrsfs}
\usepackage{appendix}
\usepackage{color}
 \usepackage[colorlinks, linkcolor=blue, citecolor=blue]{hyperref}
\usepackage{exscale}
\usepackage{relsize}

\newcommand{\R}{\mathbb{R}}

\newcommand{\dis}{\displaystyle}

\renewcommand{\theequation}{\arabic{section}.\arabic{equation}}

\newtheorem{Thm}{Theorem}[section]
\newtheorem{Lem}[Thm]{Lemma}

\newtheorem{Prop}[Thm]{Proposition}

\newtheorem{Rem}[Thm]{Remark}

\allowdisplaybreaks

\begin{document}

\title[Neumann problem for the Lane-Emden System]{Sign-changing  solutions to the slightly supercritical Lane-Emden system with Neumann boundary conditions}

\author{Qing Guo, \,\, Shuangjie Peng}
 \address[Qing Guo]{College of Science, Minzu University of China, Beijing 100081, China} \email{guoqing0117@163.com}


\address[Shuangjie Peng]{ School of Mathematics and  Statistics, Central China Normal University, Wuhan, P.R. China}\email{ sjpeng@ccnu.edu.cn}


\date{\today}

\begin{abstract}

We consider the following slightly supercritical problem for the Lane-Emden system with Neumann boundary conditions:
\begin{equation*}
\begin{cases}
-\Delta u_1=|u_2|^{p_\epsilon-1}u_2,\  &in\ \Omega,\\
-\Delta u_2=|u_1|^{q_\epsilon-1}u_1, \  &in\ \Omega,\\
\partial_\nu u_1=\partial_\nu u_2=0,\ &on\ \partial\Omega
\end{cases}
\end{equation*}
where $\Omega=B_1(0)$ is the unit ball in $\R^n$ ($n\geq4$) centered at the origin, $p_\epsilon=p+\alpha\epsilon, q_\epsilon=q+\beta\epsilon$ with $\alpha,\beta>0$ and $\frac1{p+1}+\frac1{q+1}=\frac{n-2}n$. 
We show the existence and multiplicity of concentrated solutions based on the Lyapunov-Schmidt reduction argument incorporating the zero-average condition by certain symmetries.

It is worth noting that we simultaneously consider two cases: $p>\frac n{n-2}$ and $p<\frac n{n-2}$. The coupling mechanisms of the system are completely different in these different cases, leading to significant changes in the behavior of the solutions. The research challenges also vary. Currently, there are very few papers that take both ranges into account when considering solution construction. Therefore, this is also the main feature and  new ingredient of our work.

\end{abstract}

\maketitle

\section{Introduction and main results}
\setcounter{equation}{0}

In this paper, we are concerned with  the following Neumann problem
\begin{equation}\label{eq1}
\begin{cases}
-\Delta u_1=|u_2|^{p_\epsilon-1}u_2,\  &in\ \Omega,\\
-\Delta u_2=|u_1|^{q_\epsilon-1}u_1, \  &in\ \Omega,\\
\partial_\nu u_1=\partial_\nu u_2=0,\ &on\ \partial\Omega
\end{cases}
\end{equation}
where $\frac1{p+1}+\frac1{q+1}=\frac{n-2}n$, $p_\epsilon=p+\alpha\epsilon, q_\epsilon=q+\beta\epsilon$ with $\alpha,\beta>0$,  $\Omega=B_1(0)$ is the unit ball in $\R^n$ ($n\geq4$) centered at the origin and $\nu$ is the outward pointing normal on $\partial\Omega$.

\medskip
The standard Lane-Emden system in \eqref{eq1}
is a typical Hamiltonian-type strongly coupled elliptic systems, which have been a subject of intense interest and have a rich structure.
Due to the fact that  tools  for analyzing a single equation cannot be used in a direct way to treat the systems, there have been very few results on the existence of solutions for strongly indefinite systems and their qualitative properties.
 One of the first result about positive solutions  appeared in \cite{c-f-m} based on topological arguments. In \cite{f-f}, a variational argument relying on a linking theorem was used to show an existence result.
 In \cite{b-s-r}, the existence, positivity and uniqueness of ground state solutions  was studied.  One may also refer to  \cite{s} and the surveys in \cite{f}.
 It is well known that the system is strongly affected by the values of the couple $(p,q)$.
 The existence theory is associated with the critical hyperbola
\begin{align}\label{pq}
\frac1{p+1}+\frac1{q+1}=\frac{n-2}n,
\end{align}
 which was introduced by  \cite{clement-figueiredo-mitidieri}
and \cite{vorst}.
\medskip

An important observation is that solutions to the Neumann problem \eqref{eq1} satisfy the following compatibility condition necessary:
\begin{align}
\int_\Omega|u_1|^{q_\epsilon-1}u_1=\int_\Omega|u_2|^{p_\epsilon-1}u_2=0,
\end{align}
from which we know that if $(u_1,u_2)\neq0$ is a strong solution, then  it is necessary that both $u_1$ and $u_2$ are  sign-changing.
In contrast to the Dirichlet problem, it is believed that the Neumann problem   allows the existence of smooth solutions on the critical hyperbola, but {\bf there have been very few results to \eqref{eq1} with Neumann conditions up to now.}
The only known existence of least energy nodal solutions is in the subcritical case  \cite{s-t} and in the critical case \cite{arxiv2}: $\frac n{p+1}+\frac n{q+1}\geq n-2$.
It is worth noting that {\bf the application of Lyapunov-Schmidt reduction strategy has several differences with respect to its implementation in the study of Dirichlet problems} (see \cite{KP,kim-moon}). For example, the maxima and minima of the solutions to the Neumann problems can be located at specific points on the boundary, while, in the Dirichlet case, located in the interior of the domain.
This difference implies vital changes in the method, since the curvature of the boundary plays an important role and the blow-up analysis in the Neumann case leads to a limiting problem in the half-space up to a rescaling.
\medskip

It is natural to believe that
 the supercritical case $\frac1{p+1}+\frac1{q+1}<\frac{n-2}n$ would be more complex, and the existence of a nontrivial homology class in $\Omega$ does not guarantee the existence of a nontrivial solution to \eqref{eq1}.  This can be seen  from the single Lane-Emden-Fowler problem with Dirichlet condition \cite{passaseo,c-f-p,wei-yan-jmpa,ackermann-clapp-pistoia}.
 Motivated by \cite{arxiv}, which is related to a single equation, we are aimed to deal with \eqref{eq1}, with small $\epsilon>0$, and investigate the existence of multiple concentrated solutions in the unit ball.


\medskip
In order to introduce our results, we first consider the limit problem.
A positive ground state $(U,V)$ to the following system was found in \cite{lions},
\begin{align}\label{eqUV}
\begin{cases}
&\displaystyle-\Delta U=|V|^{p -1}V,\ \ in\ \ \mathbb R^n,\\
&\displaystyle-\Delta V=|U|^{q -1}U,\ \ in\ \ \mathbb R^n,\\
&\displaystyle(U,V)\in \dot{W}^{2,\frac{p +1}{p }}(\mathbb R^n)\times\dot{W}^{2,\frac{q +1}{q }}(\mathbb R^n),
\end{cases}
\end{align}
where $n\geq3$ and $(p ,q )$ satisfy \eqref{pq}.
By Sobolev embeddings, there holds that
\begin{align}\label{emb}
\begin{split}
\displaystyle \dot{W}^{2,\frac{p +1}{p }}(\mathbb R^n)\hookrightarrow \dot{W}^{1,p^*}(\mathbb R^n)\hookrightarrow L^{q +1}(\R^n),\\
\displaystyle  \dot{W}^{2,\frac{q +1}{q }}(\mathbb R^n)\hookrightarrow \dot{W}^{1,q^*}(\mathbb R^n)\hookrightarrow L^{p +1}(\R^n),
\end{split}\end{align}
with
$$\frac1{p^*}=\frac {p }{p +1}-\frac1n=\frac1{q +1}+\frac1n,\ \ \ \ \frac1{q^*}=\frac {q }{q +1}-\frac1n=\frac1{p +1}+\frac1n,$$
and so the following energy functional is well-defined in $\dot W^{2,\frac{p +1}{p }}(\R^n)\times\dot W^{2,\frac{q +1}{q }}(\R^n)$:
\begin{align*}
 I_{00}(u,v):=\int_{\R^n}\nabla u\cdot\nabla v
-\frac1{p +1}\int_{\R^n}| v|^{p +1}-\frac1{q +1}\int_{\R^n}|u|^{q +1}.
\end{align*}
According to \cite{alvino-lions-trombetti},  the ground state is radially symmetric and decreasing up to a suitable translation.
Thanks to  \cite{hulshof-vorst} and \cite{wang}, \eqref{eqUV} admits a family of solutions
\begin{align*}
(U_{\xi,\delta}(y),V_{\xi,\delta}(y))=(\delta^{-\frac n{q+1}}U_{0,1}(\delta^{-1}(y-\xi)),\delta^{-\frac n{p+1}}V_{0,1}(\delta^{-1}(y-\xi))),
\end{align*}
which are positive and radially decreasing with respect to $\xi\in\R^n$ and $\delta$ is a concentration parameter related to the scalings of the system.
Denote the positive ground state $(U,V):=(U_{0,1},V_{0,1})$ of \eqref{eqUV} with $U_{0,1}(0)=1$.
It is worth noting that we need prove the compatibility condition by performing delicate asymptotic estimates, which are not at all straightforward mainly because {\bf $(U_{\xi,\delta}(y),V_{\xi,\delta}(y))$ do not have an explicit expression and one only has the access to the decay at infinity.}
Sharp asymptotic behavior  of the ground states to \eqref{eqUV}  (see \cite{hulshof-vorst}) and the non-degeneracy for \eqref{eqUV}  at each ground state (see \cite{frank-kim-pistoia}) play an important role to construct bubbling solutions especially using Lyapunov-Schmidt reduction methods.

\medskip

Denoting the Banach space
\begin{align*}
&X_{p ,q }:=\Big(W^{2,\frac{p+1}p}(\Omega)\times W^{2,\frac{q+1}q}(\Omega)\Big)\cap\Big( L^{q_\epsilon +1}(\Omega)\times L^{p_\epsilon +1}(\Omega)\Big),
\end{align*}
we introduce our workspace as follows:
\begin{align*}
 H_\epsilon:=\{&(u_1,u_2)\in X_{p,q}:u_i(x_1,\ldots,x_{n-1},-x_n)=-u_i(x_1,\ldots,x_{n-1},x_n),\\&
u_i(x_1,\ldots,-x_j,\ldots,x_n)=u_i(x_1,\ldots,x_j,\ldots,x_n),j=1,\ldots,n-1,i=1,2\},
\end{align*}
which is equipped with the norm
\begin{align}\label{norm}
\begin{split}
\|(v_1,v_2)\|_{\epsilon}&=\|(v_1,v_2)\|+\|(v_1,v_2)\|_{L^{q_\epsilon +1}\times L^{p_\epsilon +1}},
\end{split}\end{align}
where
\begin{align}\label{norm'}
\begin{split}
\|(v_1,v_2)\|
:=\|\Delta v_1\|_{L^{\frac{p+1}p}(\Omega)}+\|\Delta v_2\|_{L^{\frac{q+1}q}(\Omega)}.
\end{split}\end{align}
Note that if $(u_1,u_2)\in H_\epsilon$, then
\begin{align*}
\int_\Omega u_i=\int_\Omega|u_i|^tu_i=0,\ \ for\ any\ t>0,\ i=1,2.
\end{align*}

\medskip
For any $s>1$ and  $h\in L^s(\Omega)$ satisfying $\int_\Omega h=0$, we denote that
$u=K(h)\in W^{2,s}(\Omega)$ is the unique solution to the pure Neumann problem
\begin{align}\label{1.8}
-\Delta u=h\ \ in\ \ \Omega,\ \ \ \ \  \partial_\nu u=0\ \ on\ \ \partial\Omega,\ \ \ \ \ \int_\Omega u=0.
\end{align}
On the other hand, for $v\in W^{2,s}_{loc}(\R^n)$ such that $\int_\Omega\Delta v=0$, we set
$$Pv:=K(-\Delta v).$$
\medskip

Let $$W_{1,\delta}=U_{e_n,\delta}-U_{-e_n,\delta},\ \ \ W_{2,\delta}=V_{e_n,\delta}-V_{-e_n,\delta},$$
where $e_n=(0,\ldots,0,1)\in\R^n$. We know that $(PW_{1,\delta},PW_{2,\delta})$ is the solution to
\begin{align*}
\begin{cases}
\dis -\Delta PW_{1,\delta}=-\Delta W_{1,\delta}=V_{e_n,\delta}^p-V_{-e_n,\delta}^p \ \ in\ \ \Omega,\\
\dis -\Delta PW_{2,\delta}=-\Delta W_{2,\delta}=U_{e_n,\delta}^q-U_{-e_n,\delta}^q \ \ in\ \ \Omega,\\
\dis \partial_\nu PW_{1,\delta}=\partial_\nu PW_{1,\delta}=0\ \ on\ \ \partial\Omega,\ \ \ \ \int_\Omega PW_{1,\delta}=\int_\Omega PW_{2,\delta}=0.
\end{cases}
\end{align*}
\medskip

Throughout this paper, we assume\\
\medskip
{\bf (P)} $(p,q)$ satisfies \eqref{pq} with $p<q$ and one of the following condition holds:

(i) $\frac n{n-2}<p<\frac{n+2}{n-2}$    \ \ or
\smallskip

(ii) $p_n<p<\frac n{n-2}$ with
$$p_n=\frac{2n+1+\sqrt{(2n+1)^2-24(n-2)}}{4(n-2)}>1.$$

\medskip
\begin{Rem}
The condition (ii) of {\bf (P)} ensures that
$\frac{p-1}{q+1}n>\frac12$, which and other necessary conditions are needed in our discussion later.
\end{Rem}

Our main result in this paper can be stated as follows.
\begin{Thm}\label{th1}
Assume  {\bf (P)} with $n\geq4$ and let $\Omega\subset\R^n$ be the unite ball centred at the origin. Then, there exists $\epsilon_0>0$ such that, for each $\epsilon\in(0,\epsilon_0)$, problem \eqref{eq1} has a solution $(u_{1,\epsilon},u_{2,\epsilon})\in H_\epsilon$
of the form
\begin{align}\label{constructv}
u_{i,\epsilon}=PW_{i,\delta_\epsilon}+\phi_{i,\epsilon},\ i=1,2,
\end{align}
where $\delta_\epsilon=d(\epsilon)\epsilon$ with $d(\epsilon)\rightarrow d_*$ as $\epsilon\rightarrow0$ for certain $d_*>0$.
Moreover, $(\phi_{1,\epsilon},\phi_{2,\epsilon})\in H_\epsilon$ satisfies that
 $\|(\phi_{1,\epsilon},\phi_{2,\epsilon})\|_\epsilon\rightarrow0$.

\end{Thm}
\medskip
\begin{Rem}
It is worth noting that {\bf in the case of $p<\frac n{n-2}$, the system \eqref{eq1} exhibits   stronger nonlinear feature that the single equation does not have.}
The essential reason  lies in the fact that the decay order of the ground state solution $(U,V)$ of the limit problem \eqref{eqUV} at infinity differs from (strictly smaller than) that of the fundamental  solution of $-\Delta$ in $\R^n$ (see Lemma \ref{lemasym}), resulting in significant differences in the properties of the solution compared to the single equations.

From the counterpart problem involving the Dirichlet boundary condition, it is much easier to understand: unlike $p>\frac n{n-2}$, in the case of $p<\frac n{n-2}$, the approximate solution of the Lane-Emden system cannot mainly be expressed using the ground state solution of the limiting problem and the regular part $H$ of the Green's function. More precisely, the characterization of the boundary behavior of  a similar regular function $\tilde H$ corresponding to  $p<\frac n{n-2}$ becomes rougher and more difficult to control than that of $H$ related to the case of $p>\frac n{n-2}$.
\medskip

{\bf It is worth noting that  we consider both two ranges  contained in $p>\frac n{n-2}$ and $p<\frac n{n-2}$ respectively.} The coupling mechanism of the strongly indefinite problem in these two cases is totally different.  Even in the case of $p>\frac{n}{n-2}$, the blow-up scenario is not the same as that of the single Lane-Emden equation, and the standard approach does not work well, which forces us to adopt some new approach and analysis.

Moreover, {\bf the case $p=\frac n{n-2}$ can be treated by slightly modifying the proof of the case (i)} in {\bf (P)}, and we omit the details.
\end{Rem}

\medskip

\medskip

This paper is organized as follows. In section 2, we introduce investigate the projection of the bubbles and construct the basic Ansatz for our proofs.
In section 3, we perform the linear analysis and solve the auxiliary nonlinear problem, reducing the problem to finding a critical point of some function $J_\epsilon$,
which is called the reduced energy on a finite-dimensional set. Finally, the reduced problem is solved in section 4.  Some useful tools and basic estimates 
are put in the appendix.

\medskip

\section{Projections and Ansatz}

\subsection{Properties of the bubbles}

Recall the bubbles satisfy the following properties.

\begin{Lem}\label{lemasym}\cite{hulshof-vorst}
Assume that $p \leq\frac{n+2}{n-2}$. There exist some positive constants $a=a_{n,p }$ and $b=b_{n,p }$ depending only on $n$ and $p $ such that
\begin{align}\label{asymV}
&\lim_{r\rightarrow\infty}r^{n-2}V_{0,1}(r)=b_{n,p };
\end{align}
while
\begin{align}\label{asymU}
\begin{cases}
\dis \lim_{r\rightarrow\infty}r^{n-2}U_{0,1}(r)=a_{n,p },\ \ \ \ \ \ \ &\text{if}\  p >\frac n{n-2};\vspace{0.12cm}\\
\dis \lim_{r\rightarrow\infty}\dis\frac{r^{n-2}}{\log r}U_{0,1}(r)=a_{n,p },\ \ \ \ \ \ \ &\text{if}\ p =\frac n{n-2};\vspace{0.12cm}\\
\dis \lim_{r\rightarrow\infty}r^{(n-2)p -2}U_{0,1}(r)=a_{n,p },\ \ &\text{if}\ p <\frac n{n-2}.
\end{cases}
\end{align}
Furthermore, in the last case, we have $b_{n,p }^{p }=a_{n,p }((n-2)p -2)(n-(n-2)p )$.
\end{Lem}

\begin{Lem}\label{lemasym'}\cite{kim-moon}
For $r=|y|\geq1$, there exists some $C>0$ such that
\begin{align}\label{asymV'}
&\Big|V_{0,1}(r)-\frac{b_{n,p}}{r^{n-2}}\Big|\leq\frac C{r^n},\ \ \Big|V'_{0,1}(r)+\frac{b_{n,p}(n-2)}{r^{n-1}}\Big|\leq\frac C{r^{n+1}}
\end{align}
while for $\kappa_0=(n-2)p-n$, $\kappa_1\in(0,(\min\{n-(n-2)p,((n-2)p-2)q-n\})^6)$,
\begin{align}\label{asymU}
\begin{cases}
\dis \Big|U_{0,1}(r)-\frac{a_{n,p}}{r^{n-2}}\Big|\leq\frac C{r^{n-2+\kappa_0}},\ \ \Big|U'_{0,1}(r)+\frac{a_{n,p}(n-2)}{r^{n-1}}\Big|\leq\frac C{r^{n-1+\kappa_0}},\ &\text{if}\  p >\frac n{n-2};\vspace{0.12cm}\\
\dis\Big|U_{0,1}(r)-\frac{a_{n,p}\log r}{r^{n-2}}\Big|\leq\frac C{r^{n-2}},\ \ \Big|U'_{0,1}(r)+\frac{a_{n,p}(n-2)\log r}{r^{n-1}}\Big|\leq\frac C{r^{n-1}},\ \ \ &\text{if}\ p =\frac n{n-2};\vspace{0.12cm}\\
\dis\Big|U_{0,1}(r)-\frac{a_{n,p}}{r^{(n-2)p-2}}\Big|\leq\frac C{r^{(n-2)p-2+\kappa_1}},\ \ \Big|U'_{0,1}(r)+\frac{a_{n,p}((n-2)p-2)}{r^{(n-2)p-1}}\Big|\leq\frac C{r^{(n-2)p-1+\kappa_1}},\ \ &\text{if}\ p <\frac n{n-2}.
\end{cases}
\end{align}

\end{Lem}
\begin{Rem}\label{rem2.3}
By use of the equation \eqref{eqUV} of $(U,V)$, we obtain
\begin{align}\label{eqrUV}
\begin{cases}
&\displaystyle-U''(r)-\frac{n-1}rU'(r)=V^p(r),\\
&\displaystyle-V''(r)-\frac{n-1}rV'(r)=U^q(r).
\end{cases}
\end{align}
By Lemma \ref{lemasym} and Lemma \ref{lemasym'}, we obain
\begin{align}\label{UV''}
\begin{split}
 |U''(r)|&\leq C\Big(\frac1r|U'(r)|+V^p(r)\Big)\leq C\begin{cases}
\dis \frac1{r^n}+\frac1{r^{(n-2)p}}=O(\frac1{r^n}) \ &if\ p\in(\frac n{n-2},\frac{n+2}{n-2}),\\
\dis \frac{\log r}{r^n}+\frac1{r^{(n-2)p}}=O(\frac{\log r}{r^n}) \ &if\ p=\frac n{n-2},\\
\dis O(\frac1{r^{(n-2)p}}) \ &if\ p\in(\frac 2{n-2},\frac{n}{n-2}),
\end{cases}\\
 |V''(r)|&\leq C\Big(\frac1r|V'(r)|+U^q(r)\Big)\leq C\begin{cases}
\dis \frac1{r^n}+\frac1{r^{(n-2)q}}=O(\frac1{r^n}) \ &if\ p\in(\frac n{n-2},\frac{n+2}{n-2}),\\
\dis \frac{\log r}{r^n}+\frac{(\log r)^q}{r^{(n-2)q}}=O(\frac{1}{r^n}) \ &if\ p=\frac n{n-2},\\
\dis \frac1{r^n}+\frac1{r^{((n-2)p-2)q}}=O(\frac1{r^n})  \ &if\ p\in(\frac 2{n-2},\frac{n}{n-2})
\end{cases}\\
&=O(\frac1{r^n}).
\end{split}
\end{align}

\end{Rem}

\begin{Lem}\label{lemnonde}\cite{frank-kim-pistoia}
Set
$$
(\Psi_{0,1}^0,\Phi_{0,1}^0)=\Big(y\cdot\nabla U_{0,1}+\frac{n U_{0,1}}{q +1},y\cdot\nabla V_{0,1}+\frac{nV_{0,1}}{p +1}\Big)
$$
 and
$$(\Psi_{0,1}^l,\Phi_{0,1}^l)=(\partial_l U_{0,1},\partial_l V_{0,1}),\ \ for\ \ l=1,\ldots,n.$$
Then the space of solutions to the linear system
\begin{align}\label{eqspan}
\begin{cases}
&\displaystyle -\Delta\Psi=p V_{0,1}^{p -1}\Phi,\ \ \text{in}\ \mathbb R^n,\vspace{0.12cm}\\
&\displaystyle -\Delta\Phi=q U_{0,1}^{q -1}\Psi,\ \ \text{in}\ \mathbb R^n,\vspace{0.12cm}\\
&\displaystyle (\Psi,\Phi)\in  \dot{W}^{2,\frac{p +1}{p }}(\mathbb R^n)\times\dot{W}^{2,\frac{q +1}{q }}(\mathbb R^n)
\end{cases}
\end{align}
is spanned by
$$\Big\{(\Psi_{0,1}^0,\Phi_{0,1}^0),(\Psi_{0,1}^1,\Phi_{0,1}^1),\ldots,(\Psi_{0,1}^n,\Phi_{0,1}^n)\Big\}.$$
\end{Lem}

Therefore, the space of solutions of \eqref{eqspan} which are even in $x_1,\ldots,x_n$, is spanned by
$$(\partial_\delta U_{0,\delta}|_{\delta=1},\partial_\delta V_{0,\delta}|_{\delta=1}).$$
Observe that
\begin{align}\label{deltaU}
\delta|\partial_\delta U_{\xi,\delta}(x)|\leq CU_{\xi,\delta}(x),\ \ \delta|\partial_\delta V_{\xi,\delta}(x)|\leq CV_{\xi,\delta}(x).
\end{align}
\medskip
For convenience, in the sequel, we denote $$(U,V):=(U_{0,1},V_{0,1}).$$
\medskip

\subsection{Expansions of the projections}

In this subsection, we consider both the case  $p>\frac n{n-2}$ and  $p<\frac n{n-2}$ to give the expansion of the projections.

For $\Omega=B_1(0)$, we denote
\begin{align}\label{omegasplit}
\Omega_+=\Omega\cap B_{\frac12}(e_n),\ \ \ \Omega_-=\Omega\cap B_{\frac12}(-e_n).
\end{align}

Let $\varphi_{1,0}$ and $\varphi_{2,0}$ be the solution of
\begin{align}\label{phi0}
\begin{split}
&-\Delta\varphi_{1,0}=0,\ in\ \R^n_+,\ \ \frac{\partial\varphi_{1,0}}{\partial x_n}=-\frac{r}2 U'_{0,1}(r)|_{x_n=0},\ on\ \partial\R^n_+,\ \ \varphi_{1,0}\rightarrow0\ as\ r\rightarrow\infty;\\
&-\Delta\varphi_{2,0}=0,\ in\ \R^n_+,\ \ \frac{\partial\varphi_{2,0}}{\partial x_n}=-\frac{r}2V'_{0,1}(r)|_{x_n=0},\ on\ \partial\R^n_+,\ \ \varphi_{2,0}\rightarrow0\ as\ r\rightarrow\infty,
\end{split}
\end{align}
where $r=|x|,x=(x_1,\ldots,x_n)$.

We are aimed to establish the following expansions.
\begin{Prop}\label{lemexpansion}
It holds that
\begin{align*}
&PW_{1,\delta}(x)=W_{1,\delta}(x)-\delta^{-\frac n{q+1}+1}\Big(\varphi_{1,0}(\frac{e_n-x}{\delta})-\varphi_{1,0}(\frac{e_n+x}{\delta})\Big)+\zeta_{1,\delta}(x),\\
&PW_{2,\delta}(x)=W_{2,\delta}(x)-\delta^{-\frac n{p+1}+1}\Big(\varphi_{2,0}(\frac{e_n-x}{\delta})-\varphi_{2,0}(\frac{e_n+x}{\delta})\Big)+\zeta_{2,\delta}(x),
\end{align*}
where $\zeta_{1,\delta}=O(\delta^{2-\frac n{q+1}})$, $\zeta_{2,\delta}=O(\delta^{2-\frac n{p+1}})$,
$\partial_\delta\zeta_{1,\delta}=O(\delta^{1-\frac n{q+1}})$, $\partial_\delta\zeta_{2,\delta}=O(\delta^{1-\frac n{p+1}})$ as $\delta\rightarrow0$ uniformly in $\Omega$.
Moreover, for some constant $C>0$, we have
\begin{align}\label{zeta1}
\begin{split}
&|\zeta_{1,\delta}(x)|,|PW_{1,\delta}(x)-W_{1,\delta}(x)|\leq \begin{cases}\dis C\Big(\frac{\delta^{1-\frac n{q+1}}}{\big(1+\frac{|x-e_n|}{\delta}\big)^{n-3}}+\frac{\delta^{1-\frac n{q+1}}}{\big(1+\frac{|x+e_n|}{\delta}\big)^{n-3}}\Big) &if\ p>\frac n{n-2}\\
\dis C\Big(\frac{\delta^{1-\frac n{q+1}}}{\big(1+\frac{|x-e_n|}{\delta}\big)^{(n-2)p-3}}+\frac{\delta^{1-\frac n{q+1}}}{\big(1+\frac{|x+e_n|}{\delta}\big)^{(n-2)p-3}}\Big)
 &if\ p<\frac n{n-2}
\end{cases},\\
&|\zeta_{2,\delta}(x)|,|PW_{2,\delta}(x)-W_{2,\delta}(x)|\leq C\Big(\frac{\delta^{1-\frac n{p+1}}}{\big(1+\frac{|x-e_n|}{\delta}\big)^{n-3}}+\frac{\delta^{1-\frac n{p+1}}}{\big(1+\frac{|x+e_n|}{\delta}\big)^{n-3}}\Big),
\end{split}\end{align}
\begin{align}\label{zeta2}
\begin{split}
&|\partial_\delta\zeta_{1,\delta}(x)|,|\partial_\delta(PW_{1,\delta}(x)-W_{1,\delta}(x))|\leq
\begin{cases}\dis C\Big(\frac{\delta^{-\frac n{q+1}}}{\big(1+\frac{|x-e_n|}{\delta}\big)^{n-3}}+\frac{\delta^{-\frac n{q+1}}}{\big(1+\frac{|x+e_n|}{\delta}\big)^{n-3}}\Big)
 &if\ p>\frac n{n-2} \\
\dis C\Big(\frac{\delta^{-\frac n{q+1}}}{\big(1+\frac{|x-e_n|}{\delta}\big)^{(n-2)p-3}}+\frac{\delta^{-\frac n{q+1}}}{\big(1+\frac{|x+e_n|}{\delta}\big)^{(n-2)p-3}}\Big)
 &if\ p<\frac n{n-2}
 \end{cases},\\
&|\partial_\delta\zeta_{2,\delta}(x)|,|\partial_\delta(PW_{2,\delta}(x)-W_{2,\delta}(x))|\leq C\Big(\frac{\delta^{-\frac n{p+1}}}{\big(1+\frac{|x-e_n|}{\delta}\big)^{n-3}}+\frac{\delta^{-\frac n{p+1}}}{\big(1+\frac{|x+e_n|}{\delta}\big)^{n-3}}\Big).
\end{split}\end{align}

\end{Prop}

\medskip

The proof of Proposition \ref{lemexpansion} required the potential theory associated with the Neumann condition and the boundary analysis below.
\medskip

\subsubsection{{\bf Estimates for $\varphi_{i,0}$}}
Applying the Poisson kernel for the halfspace (see \cite{c-p}) and the fact that $n\geq4$, we have that
\begin{align*}
&\varphi_{1,0}(x)=\frac2{\omega_n(n-2)}\int_{\R^{n-1}}\frac{1}{|x'-y'|^{n-2}}\frac{|y|}2U'_{0,1}(|y|)\Big|_{y_n=0}dy',\\
&\varphi_{2,0}(x)=\frac2{\omega_n(n-2)}\int_{\R^{n-1}}\frac{1}{|x'-y'|^{n-2}}\frac{|y|}2V'_{0,1}(|y|)\Big|_{y_n=0}dy',
\end{align*}
where $\omega_n$ is the measure of the unit sphere in $\R^n$, $x=(x',x_n)=(x_1,\ldots,x_{n-1},x_n)$.

\medskip
(i) For $p>\frac n{n-2}$, by Lemma \ref{lemasym'}, for $r=|y|$,
$$\frac{|y|}2U'_{0,1}(|y|)\Big|_{y_n=0}
\leq \frac{C|y'|^2}{(1+|y'|^2)^{\frac n2}},\ \ \ \frac{|y|}2V'_{0,1}(|y|)\Big|_{y_n=0}
\leq \frac{C|y'|^2}{(1+|y'|^2)^{\frac n2}}.$$
Similar as in \cite{arxiv}, we can prove that for $i=1,2$ and $x\in\R^n_+$,
\begin{align*}
|\varphi_{i,0}(x)|&\leq C\int_{\R^{n-1}}\frac1{|x-y|^{n-2}}\frac{|y|}{(1+|y|)^{n-1}}dy\\
&\leq C\frac1{|x|^{n-3}}\int_{\R^{n-1}}\frac1{|\frac{x}{|x|}-z|^{n-2}}\frac{|z|}{(\frac{1}{|x|}+|z|)^{n-1}}dz.
\end{align*}
Actually, there holds that
\begin{align}\label{estphi0}
|\varphi_{i,0}(x)|\leq\frac C{(1+|x|)^{n-3}},\ \ |\nabla\varphi_{i,0}(x)|\leq\frac C{(1+|x|)^{n-2}},\ \
|D^2\varphi_{i,0}(x)|\leq\frac C{(1+|x|)^{n-1}}.
\end{align}
\medskip

(ii) For $p<\frac n{n-2}$,
$$\frac{|y|}2U'_{0,1}(|y|)\Big|_{y_n=0}
\leq \frac{C|y'|^2}{(1+|y'|^2)^{\frac {p(n-2)}2}},\ \ \ \frac{|y|}2V'_{0,1}(|y|)\Big|_{y_n=0}
\leq \frac{C|y'|^2}{(1+|y'|^2)^{\frac n2}}.$$
Then, for $x\in\R^n_+$, it is similar to get
\begin{align*}
|\varphi_{2,0}(x)|&\leq\frac C{(1+|x|)^{n-3}},\ \ |\nabla\varphi_{i,0}(x)|\leq\frac C{(1+|x|)^{n-2}},\ \
|D^2\varphi_{2,0}(x)|\leq\frac C{(1+|x|)^{n-1}}.
\end{align*}
While for $i=1$,
\begin{align*}
|\varphi_{1,0}(x)|&\leq C\int_{\R^{n-1}}\frac1{|x-y|^{n-2}}\frac{|y|}{(1+|y|)^{p(n-2)-1}}dy\\
&\leq C\frac1{|x|^{p(n-2)-3}}\int_{\R^{n-1}}\frac1{|\frac{x}{|x|}-z|^{n-2}}\frac{|z|}{(\frac{1}{|x|}+|z|)^{p(n-2)-1}}dz.
\end{align*}
Moreover,
\begin{align}\label{estphi0'}
|\varphi_{1,0}(x)|\leq\frac C{(1+|x|)^{p(n-2)-3}},\ \ |\nabla\varphi_{i,0}(x)|\leq\frac C{(1+|x|)^{p(n-2)-2}},\ \
|D^2\varphi_{1,0}(x)|\leq\frac C{(1+|x|)^{p(n-2)-1}}.
\end{align}
\medskip

\subsubsection{{\bf Boundary estimates }}
For $x\in\partial B_1(e_n)$, let $$z_{1,\delta}=\partial_\nu\Big(U_{0,\delta}(x)-\delta^{1-\frac n{q+1}}\varphi_{1,0}(\frac x\delta)\Big),\ \
z_{2,\delta}=\partial_\nu\Big(V_{0,\delta}(x)-\delta^{1-\frac n{p+1}}\varphi_{2,0}(\frac x\delta)\Big).$$
\begin{Lem}\label{lema1}
As $\delta\rightarrow0$, on $\partial B_1(e_n)\cap B^c_1(0)$, there holds that
\begin{align}\label{zbc}
\begin{split}
 z_{1,\delta}&=\begin{cases}\dis O (\delta^{\frac n{p+1}} ) &if\ p>\frac n{n-2}\\
\dis O (\delta^{\frac {pn}{q+1}} ) &if\ p<\frac n{n-2}
\end{cases},\ \ \
z_{2,\delta}=O (\delta^{\frac n{q+1}} ),\\
 \partial_\delta z_{1,\delta}&=\begin{cases}\dis O (\delta^{\frac n{p+1}-1} ) &if\ p>\frac n{n-2}\\
\dis O (\delta^{\frac {pn}{q+1}-1} ) &if\ p<\frac n{n-2}
\end{cases},\ \
\partial_\delta z_{2,\delta}=O (\delta^{\frac n{q+1}-1} );
\end{split}
\end{align}
While on $\partial B_1(e_n)\cap B_1(0)$
\begin{align}\label{a4}
\begin{split}
 z_{1,\delta}&=\begin{cases}\dis O\Big(\frac{\delta^{1-\frac n{q+1}}}{(1+\frac{|x'|}\delta)^{n-3}}\Big) &if\ p>\frac n{n-2}\\
\dis O\Big(\frac{\delta^{1-\frac n{q+1}}}{(1+\frac{|x'|}\delta)^{(n-2)p-3}}\Big) &if\ p<\frac n{n-2}
\end{cases},\ \ \ \ \
z_{2,\delta}=O\Big(\frac{\delta^{1-\frac n{p+1}}}{(1+\frac{|x'|}\delta)^{n-3}}\Big),\\
 \partial_\delta z_{1,\delta}&=\begin{cases}\dis O\Big(\frac{\delta^{-\frac n{q+1}}}{(1+\frac{|x'|}\delta)^{n-3}}\Big) &if\ p>\frac n{n-2}\\
\dis O\Big(\frac{\delta^{-\frac n{q+1}}}{(1+\frac{|x'|}\delta)^{(n-2)p-3}}\Big) &if\ p<\frac n{n-2}
\end{cases},\ \ \ \
\partial_\delta z_{2,\delta}=O\Big(\frac{\delta^{-\frac n{p+1}}}{(1+\frac{|x'|}\delta)^{n-3}}\Big).
\end{split}\end{align}

\end{Lem}

\begin{proof}
On $\partial B_1(e_n)\cap B^c_1(0)$, from the asymptotic behavior of $U_{0,1}$ and $V_{0,1}$ as in Lemma \ref{lemasym}, Lemma \ref{lemasym'}, Remark \ref{rem2.3} and using \eqref{estphi0},
we get that when $p>\frac n{n-2}$,
\begin{align*}
&|\partial_\nu U_{0,\delta}(x)|=O(\delta^{n-2-\frac{n}{q+1}}),\ \  |\partial_\nu V_{0,\delta}(x)|=O(\delta^{n-2-\frac{n}{p+1}}),\\
&|\partial_\delta\partial_\nu U_{0,\delta}(x)|=O(\delta^{\frac{n}{p+1}-1}),\ \  |\partial_\delta\partial_\nu V_{0,\delta}(x)|=O(\delta^{\frac{n}{q+1}-1}),\\
&\Big|\delta^{1-\frac n{q+1}}\partial_\nu\Big(\varphi_{1,0}(\frac x\delta)\Big)\Big|=O(\delta^{n-2-\frac{n}{q+1}}),\ \
\Big|\partial_\delta\Big[\delta^{1-\frac n{q+1}}\partial_\nu\Big(\varphi_{1,0}(\frac x\delta)\Big)\Big]\Big|=O(\delta^{\frac{n}{p+1}-1}),\\
&\Big|\delta^{1-\frac n{p+1}}\partial_\nu\Big(\varphi_{2,0}(\frac x\delta)\Big)\Big|=O(\delta^{n-2-\frac{n}{p+1}}),\ \
\Big|\partial_\delta\Big[\delta^{1-\frac n{p+1}}\partial_\nu\Big(\varphi_{2,0}(\frac x\delta)\Big)\Big]\Big|=O(\delta^{\frac{n}{q+1}-1}).
\end{align*}
While when $p<\frac n{n-2}$,
\begin{align*}
&|\partial_\nu U_{0,\delta}(x)|=O(\delta^{\frac{pn}{q+1}}),\ \  |\partial_\nu V_{0,\delta}(x)|=O(\delta^{\frac{n}{q+1}}),\\
&|\partial_\delta\partial_\nu U_{0,\delta}(x)|=O(\delta^{\frac{pn}{q+1}-1}),\ \  |\partial_\delta\partial_\nu V_{0,\delta}(x)|=O(\delta^{\frac{n}{q+1}-1}),\\
&\Big|\delta^{1-\frac n{q+1}}\partial_\nu\Big(\varphi_{1,0}(\frac x\delta)\Big)\Big|=O(\delta^{\frac{pn}{q+1}}),\ \
\Big|\partial_\delta\Big[\delta^{1-\frac n{q+1}}\partial_\nu\Big(\varphi_{1,0}(\frac x\delta)\Big)\Big]\Big|=O(\delta^{\frac{pn}{q+1}-1}),\\
&\Big|\delta^{1-\frac n{p+1}}\partial_\nu\Big(\varphi_{2,0}(\frac x\delta)\Big)\Big|=O(\delta^{\frac{n}{q+1}}),\ \
\Big|\partial_\delta\Big[\delta^{1-\frac n{p+1}}\partial_\nu\Big(\varphi_{2,0}(\frac x\delta)\Big)\Big]\Big|=O(\delta^{\frac{n}{q+1}-1}).
\end{align*}
Hence on $\partial B_1(e_n)\cap B^c_1(0)$, we obtain \eqref{zbc}.
\medskip

Next, we estimate on  the set $\partial B_1(e_n)\cap B_1(0)$.
Note that $$x_n=\rho(x')=1-\sqrt{1-|x'|^2}=\frac{|x'|^2}2+O(|x'|^3)\ \ as\ |x'|\rightarrow0.$$
Moreover, the exterior normal on the boundary is
\begin{align*}
\nu(x)=\frac{(\nabla\rho,-1)}{\sqrt{1+|\nabla\rho|^2}}=(x',-\sqrt{1-|x'|^2}).
\end{align*}
Thus, for $x\in\partial B_1(e_n)\cap B_1(0)$, when $p>\frac n{n-2}$,

\begin{align}\label{nuU}
\begin{split}
\partial_\nu U_{0,\delta}&=\nabla U_{0,\delta}\cdot\nu(x)
=\delta^{-\frac n{q+1}-1}U'(\frac r\delta)\frac{x\cdot\nu(x)}r\\
&=\delta^{-\frac n{q+1}-1}U'(\frac r\delta)\frac{1-\sqrt{1-|x'|^2}}r\\
&=\frac12\delta^{-\frac n{q+1}-1}U'(\frac r\delta)\frac{|x'|^2}{|x|}+O\Big(\delta^{-\frac n{q+1}-1}U'(\frac r\delta)\frac{|x'|^3}{|x|}\Big)\\
&=\frac{|x'|}2\delta^{-\frac n{q+1}-1}U'(\frac r\delta)|_{x_n=0}+O\Big(\frac{\delta^{-\frac n{q+1}+1}}{(1+\frac{|x'|}\delta)^{n-3}}\Big),
\end{split}
\end{align}
where the last equality is because Lemma \ref{lemasym'} and \eqref{UV''},
\begin{align*}
&\Big|\frac12\delta^{-\frac n{q+1}-1}U'(\frac r\delta)\frac{|x'|^2}{|x|}-\frac{|x'|}2\delta^{-\frac n{q+1}-1}U'(\frac r\delta)|_{x_n=0}\Big|\\
&\leq \frac{C\delta^{-\frac n{q+1}-1}|x'|^2x_n^2}2\Big|U''(\frac{|x'|}\delta)\frac{\delta^{-3}}{\frac{|x'|^2}{\delta^2}}-U'(\frac{|x'|}\delta)\frac{\delta^{-3}}{\frac{|x'|^3}{\delta^3}}\Big|\\
&\leq C\delta^{n-2-\frac n{q+1}}\frac{|x'|^4}{(\delta+|x'|)^n}
=o\Big(\frac{\delta^{-\frac n{q+1}+1}}{(1+\frac{|x'|}\delta)^{n-3}}\Big).
\end{align*}
While  when $p<\frac n{n-2}$,  since now
\begin{align*}
&\Big|\frac12\delta^{-\frac n{q+1}-1}U'(\frac r\delta)\frac{|x'|^2}{|x|}-\frac{|x'|}2\delta^{-\frac n{q+1}-1}U'(\frac r\delta)|_{x_n=0}\Big|\\
&\leq C\delta^{(n-2)p-2-\frac n{q+1}}\frac{|x'|^4}{(\delta+|x'|)^{(n-2)p}}
=o\Big(\frac{\delta^{-\frac n{q+1}+1}}{(1+\frac{|x'|}\delta)^{(n-2)p-3}}\Big),
\end{align*}
then
\begin{align}\label{nuU'}
\begin{split}
\partial_\nu U_{0,\delta}&=\nabla U_{0,\delta}\cdot\nu(x)
=\frac{|x'|}2\delta^{-\frac n{q+1}-1}U'(\frac r\delta)|_{x_n=0}+O\Big(\frac{\delta^{-\frac n{q+1}+1}}{(1+\frac{|x'|}\delta)^{(p(n-2)-3}}\Big).
\end{split}
\end{align}

Similarly
\begin{align}\label{nuV}
\partial_\nu V_{0,\delta}=\frac{|x'|}2\delta^{-\frac n{p+1}-1}V'(\frac r\delta)|_{x_n=0}+O\Big(\frac{\delta^{-\frac n{p+1}+1}}{(1+\frac{|x'|}\delta)^{n-3}}\Big).
\end{align}

Now let $\eta_{i,\delta}=\delta^{-\frac n{q+1}+1}\varphi_{i,0}(\frac x\delta)$ for $i=1,2$.
Using the definition of $\varphi_{i,0}$, for $x\in\partial B_1(e_n)\cap B_1(0)$, it holds that in the case of $p>\frac n{n-2}$,
\begin{align}
\begin{split}
&\partial_\nu\eta_{1,\delta}(x)=\nabla\eta_{1,\delta}(x)\cdot\nu(x)\\&
=\partial_{-e_n}\eta_{1,\delta}(x',0)+(\nabla\eta_{1,\delta}(x)-\nabla\eta_{1,\delta}(x',0))\cdot\nu(x)+\nabla\eta_{1,\delta}(x',0)\cdot(\nu(x)+e_n)\\
&=-\delta^{-\frac n{q+1}}\frac{\varphi_{1,0}}{\partial x_n}(\frac x\delta)\Big|_{x_n=0}+O\Big(\frac{\delta^{-\frac n{q+1}}|x'|}{(1+\frac{|x'|}\delta)^{n-2}}\Big)
\\&=\frac{|x'|}2\delta^{-\frac n{q+1}-1}U'(\frac r\delta)|_{x_n=0}+O\Big( \frac{\delta}{(1+\frac{|x'|}\delta)^{n-3}}\Big),
\end{split}
\end{align}
where we estimate by \eqref{phi0} that
\begin{align*}
&\delta^{\frac n{q+1}}|(\nabla\eta_{1,\delta}(x)-\nabla\eta_{1,\delta}(x',0))\cdot\nu(x)|
\leq|\nabla\varphi_{1,0}(\frac{x}\delta)-\nabla\varphi_{1,0}(\frac{x'}\delta,0)|\\
&\leq\Big|\Big(\frac{\partial^2\varphi_{1,0}}{\partial x_1\partial x_n}(\frac{x'}{\delta},\theta_1\frac{x_1}{\delta}),\ldots,\frac{\partial^2\varphi_{1,0}}{\partial x_n^2}(\frac{x'}{\delta},\theta_n\frac{x_1}{\delta})\Big)
\Big|\frac{|x_n|}\delta\\
&\leq\frac{C|x_n|}{\delta(1+\frac{|x'|}\delta)^{n-1}}=O\Big(\frac{|x'|^2}{\delta(1+\frac{|x'|}\delta)^{n-1}}\Big)
=O\Big(\frac{\delta}{(1+\frac{|x'|}\delta)^{n-3}}\Big)
\end{align*}
and
\begin{align*}
&\delta^{\frac n{q+1}}|\nabla\eta_{1,\delta}(x',0)\cdot(\nu(x)+e_n)|
\leq|\nabla\varphi_{1,0}(\frac{x'}\delta,0)||\nu(x)+e_n|\\
&\leq C\frac{|(x',1-\sqrt{1-|x'|^2})|}{(1+\frac{|x'|}\delta)^{n-2}}=O\Big(\frac{|x'|}{(1+\frac{|x'|}\delta)^{n-2}}\Big)
=O\Big( \frac{\delta}{(1+\frac{|x'|}\delta)^{n-3}}\Big).
\end{align*}
While in the case of $p<\frac n{n-2}$,
\begin{align}
\begin{split}
&\partial_\nu\eta_{1,\delta}(x)=\nabla\eta_{1,\delta}(x)\cdot\nu(x)\\
&=-\delta^{-\frac n{q+1}}\frac{\varphi_{1,0}}{\partial x_n}(\frac x\delta)\Big|_{x_n=0}+O\Big(\frac{\delta^{-\frac n{q+1}}|x'|}{(1+\frac{|x'|}\delta)^{p(n-2)-2}}\Big)
\\&=\frac{|x'|}2\delta^{-\frac n{q+1}-1}U'(\frac r\delta)|_{x_n=0}+O\Big( \frac{\delta}{(1+\frac{|x'|}\delta)^{p(n-2)-3}}\Big),
\end{split}
\end{align}
since in this case,
\begin{align*}
&\delta^{\frac n{q+1}}|(\nabla\eta_{1,\delta}(x)-\nabla\eta_{1,\delta}(x',0))\cdot\nu(x)|
\leq|\nabla\varphi_{1,0}(\frac{x}\delta)-\nabla\varphi_{1,0}(\frac{x'}\delta,0)|\\
&\leq\frac{C|x_n|}{\delta(1+\frac{|x'|}\delta)^{p(n-2)-1}}=O\Big(\frac{|x'|^2}{\delta(1+\frac{|x'|}\delta)^{p(n-2)-1}}\Big)
=O\Big(\frac{\delta}{(1+\frac{|x'|}\delta)^{p(n-2)-3}}\Big)
\end{align*}
and
\begin{align*}
&\delta^{\frac n{q+1}}|\nabla\eta_{1,\delta}(x',0)\cdot(\nu(x)+e_n)|
\leq|\nabla\varphi_{1,0}(\frac{x'}\delta,0)||\nu(x)+e_n|\\
&\leq C\frac{|(x',1-\sqrt{1-|x'|^2})|}{(1+\frac{|x'|}\delta)^{p(n-2)-2}}=O\Big(\frac{|x'|}{(1+\frac{|x'|}\delta)^{p(n-2)-2}}\Big)
=O\Big( \frac{\delta}{(1+\frac{|x'|}\delta)^{p(n-2)-3}}\Big).
\end{align*}

\end{proof}
\medskip

\subsubsection{{\bf The expansion}}
\begin{proof}[\textbf{Proof of Proposition \ref{lemexpansion}}]

For $x\in\Omega=B_1(0)$, set
\begin{align*}
&\zeta_{1,\delta}=PW_{1,\delta}-W_{1,\delta}-\delta^{1-\frac n{q+1}}\Big(\varphi_{1,0}(\frac{e_n-x}\delta)-\varphi_{1,0}(\frac{e_n+x}\delta)\Big)\\
&\zeta_{2,\delta}=PW_{2,\delta}-W_{2,\delta}-\delta^{1-\frac n{p+1}}\Big(\varphi_{2,0}(\frac{e_n-x}\delta)-\varphi_{2,0}(\frac{e_n+x}\delta)\Big).
\end{align*}
Then we have
\begin{align*}
-\Delta\zeta_{1,\delta}=0\ \  in\ \Omega,\ \ \ \partial_\nu\zeta_{1,\delta}=z_{1,\delta}'+z_{1,\delta}''\ \ on\ \partial\Omega,\ \ \ \int_\Omega\zeta_{1,\delta}=0,\\
-\Delta\zeta_{2,\delta}=0\ \  in\ \Omega,\ \ \ \partial_\nu\zeta_{2,\delta}=z_{2,\delta}'+z_{2,\delta}''\ \ on\ \partial\Omega,\ \ \ \int_\Omega\zeta_{2,\delta}=0,
\end{align*}
where
\begin{align*}
z_{1,\delta}'=\partial_\nu\Big[\delta^{1-\frac n{q+1}}\varphi_{1,0}(\frac{e_n-x}\delta)-U_{e_n,\delta}\Big],\ \
z_{2,\delta}'=\partial_\nu\Big[\delta^{1-\frac n{p+1}}\varphi_{2,0}(\frac{e_n-x}\delta)-V_{e_n,\delta}\Big],\\
z_{1,\delta}''=\partial_\nu\Big[U_{-e_n,\delta}-\delta^{1-\frac n{q+1}}\varphi_{1,0}(\frac{e_n+x}\delta)\Big],\ \
z_{2,\delta}''=\partial_\nu\Big[V_{e_n,\delta}-\delta^{1-\frac n{p+1}}\varphi_{2,0}(\frac{e_n+x}\delta)\Big].
\end{align*}

Using the transformation $\tau_1(x)=e_n-x$, we have $\tau_1(B_1(0))=B_1(e_n)$. For $x\in\partial B_1(0)$, denoting $\nu(x)$ as the exterior unitary normal on $\partial B_1(0)$
at $x$, then $\nu(\tau_1(x))$ is the exterior unitary normal on $\partial B_1(e_n)$ at $\tau_1(x)$ and it is easy that $\nu(x)=-\nu(\tau_1(x))$.
For $y\in B_1(e_n)$, set $$\chi_{1,\delta}(y)=\delta^{1-\frac n{q+1}}\varphi_{1,0}(\frac y\delta)-U_{0,\delta}(y),\ \
\chi_{2,\delta}(y)=\delta^{1-\frac n{p+1}}\varphi_{2,0}(\frac y\delta)-V_{0,\delta}(y).$$
Then
$$z_{i,\delta}'=\partial_\nu[\chi_{i,\delta}(\tau_1(x))]=\nabla[\chi_{i,\delta}(\tau_1(x))]\cdot\nu(x)=\nabla\chi_{i,\delta}(\tau_1(x))\cdot\nu(\tau_1(x)),\ \ i=1,2.$$
Therefore, \eqref{phi0}, \eqref{estphi0} and \eqref{estphi0'} imply that
on $\partial B_1(0)\cap B^c_1(e_n)$
\begin{align*}
\begin{split}
&z_{1,\delta}'=\begin{cases}\dis O (\delta^{\frac n{p+1}} ) &if\ p>\frac n{n-2}\\
\dis O (\delta^{\frac {pn}{q+1}} ) &if\ p<\frac n{n-2}
\end{cases},\ \ \ \ \
z_{2,\delta}'=O (\delta^{\frac n{q+1}} ).
\end{split}
\end{align*}
While on $\partial B_1(0)\cap B_1(e_n)$
\begin{align*}
&z_{1,\delta}'=\begin{cases}\dis O\Big(\frac{\delta^{1-\frac n{q+1}}}{(1+\frac{|x'|}\delta)^{n-3}}\Big) &if\ p>\frac n{n-2}\\
\dis O\Big(\frac{\delta^{1-\frac n{q+1}}}{(1+\frac{|x'|}\delta)^{(n-2)p-3}}\Big)  &if\ p<\frac n{n-2}
\end{cases},\ \ \
z_{2,\delta}'=O\Big(\frac{\delta^{1-\frac n{p+1}}}{(1+\frac{|x'|}\delta)^{n-3}}\Big).
\end{align*}

Analogously, we use the isometry $\tau_2(x)=x+e_n=x-(-e_n)$ and obtain the same estimates for $z_{i,\delta}''(i=1,2)$ on
$\partial B_1(0)\cap B^c_1(-e_n)$ and $\partial B_1(0)\cap B_1(-e_n)$ respectively.

Combining the above estimates, there hold that on $\partial B_1(0)\cap B_1(\pm e_n)$ and using the expansion of $x_n$ on $\partial B_1(0)\cap B_1(\pm e_n)$
\begin{align}\label{a14}
\begin{split}
&\partial_\nu \zeta_{1,\delta}=\begin{cases}\dis O\Big(\frac{\delta^{1-\frac n{q+1}}}{(1+\frac{|x'|}\delta)^{(n-2)p-3}}\Big)
=O\Big(\frac{\delta^{1-\frac n{q+1}}}{(1+\frac{|x\mp e_n|}\delta)^{(n-2)p-3}}\Big) &if\ p<\frac n{n-2}\\
\dis O\Big(\frac{\delta^{1-\frac n{q+1}}}{(1+\frac{|x'|}\delta)^{n-3}}\Big)
=O\Big(\frac{\delta^{1-\frac n{q+1}}}{(1+\frac{|x\mp e_n|}\delta)^{n-3}}\Big) &if\ p>\frac n{n-2}
\end{cases},\\
&\partial_\nu \zeta_{2,\delta}=O\Big(\frac{\delta^{1-\frac n{p+1}}}{(1+\frac{|x'|}\delta)^{n-3}}\Big)
=O\Big(\frac{\delta^{1-\frac n{p+1}}}{(1+\frac{|x\mp e_n|}\delta)^{n-3}}\Big),
\end{split}\end{align}
while on $\partial B_1(0)\cap B^c_1(e_n)$
\begin{align*}
\begin{split}
&\partial_\nu \zeta_{1,\delta}=\begin{cases}O (\delta^{\frac {pn}{q+1}} ) &if\ p<\frac n{n-2}\\
O (\delta^{\frac n{q+1}} ) &if\ p>\frac n{n-2}
\end{cases},\ \ \ \
\partial_\nu \zeta_{2,\delta}=O (\delta^{\frac n{p+1}} ).
\end{split}
\end{align*}
In particular, on $\partial B_1(0)$, it holds that
\begin{align*}
&\partial_\nu \zeta_{1,\delta}
=\begin{cases}\dis O\Big(\frac{\delta^{1-\frac n{q+1}}}{(1+\frac{|x-e_n|}\delta)^{(n-2)p-3}}+\frac{\delta^{1-\frac n{q+1}}}{(1+\frac{|x+e_n|}\delta)^{(n-2)p-3}}\Big)  &if\ p<\frac n{n-2}\\
\dis O\Big(\frac{\delta^{1-\frac n{q+1}}}{(1+\frac{|x-e_n|}\delta)^{n-3}}+\frac{\delta^{1-\frac n{q+1}}}{(1+\frac{|x+e_n|}\delta)^{n-3}}\Big)  &if\ p>\frac n{n-2}
\end{cases},\\
&\partial_\nu \zeta_{2,\delta}
=O\Big(\frac{\delta^{1-\frac n{p+1}}}{(1+\frac{|x-e_n|}\delta)^{n-3}}+\frac{\delta^{1-\frac n{p+1}}}{(1+\frac{|x+ e_n|}\delta)^{n-3}}\Big).
\end{align*}

Using the integral representation formulas \cite{f-j-r}, we know that
\begin{align*}
& \zeta_{1,\delta}
= \begin{cases}\dis O\Big(\frac{\delta^{1-\frac n{q+1}}}{(1+\frac{|x-e_n|}\delta)^{(n-2)p-3}}+\frac{\delta^{1-\frac n{q+1}}}{(1+\frac{|x+e_n|}\delta)^{(n-2)p-3}}\Big)  &if\ p<\frac n{n-2}\\
\dis O\Big(\frac{\delta^{1-\frac n{q+1}}}{(1+\frac{|x-e_n|}\delta)^{n-3}}+\frac{\delta^{1-\frac n{q+1}}}{(1+\frac{|x+e_n|}\delta)^{n-3}}\Big)  &if\ p>\frac n{n-2}
\end{cases},\\
&\zeta_{2,\delta}
=O\Big(\frac{\delta^{1-\frac n{p+1}}}{(1+\frac{|x-e_n|}\delta)^{n-3}}+\frac{\delta^{1-\frac n{p+1}}}{(1+\frac{|x+ e_n|}\delta)^{n-3}}\Big),
\end{align*}
which combined with \eqref{estphi0} and \eqref{estphi0'} leads to the same estimate for $|PW_{i,\delta}-W_{i,\delta}|$ for $i=1,2$, concluding \eqref{zeta1} in Proposition \ref{lemexpansion}.

Finally, set
$$\hat\zeta_{1,\delta}(x)=\delta^{\frac n{q+1}}\zeta_{1,\delta}(e_n-\delta x),\ \ \hat\zeta_{2,\delta}(x)=\delta^{\frac n{p+1}}\zeta_{2,\delta}(e_n-\delta x).$$
Then by \eqref{a14}, for $i=1,2$,
\begin{align}
-\Delta\hat\zeta_{i,\delta}=0\ in\ \Omega_\delta=B_{\frac1\delta}(\frac{e_n}\delta),\ \ \ \ \int_{\Omega_\delta}\hat\zeta_{i,\delta}=0,\ \ \ \
\partial_\nu\hat\zeta_{i,\delta}=O(\delta^2)\ \ on\ \partial\Omega_\delta\subset\R^n_+.
\end{align}
Hence, by the explicit Green's formula in balls \cite{s-t-t,r-w}, we obtain that $\hat\zeta_{i,\delta}=O(\delta^2)$, which gives
$\zeta_{1,\delta}=O(\delta^{2-\frac n{q+1}})$, $\zeta_{2,\delta}=O(\delta^{2-\frac n{p+1}})$.

\medskip
As for the estimates for $\partial_\delta PW_{i,\delta},i=1,2$, the proof can follow the previous steps based on \eqref{a4}.

\end{proof}

\medskip

\medskip
\subsection{The Ansatz}
Denote that $$f_{2,\epsilon}(t)=|t|^{p_\epsilon-1}t,\ \ \ f_{1,\epsilon}(t)=|t|^{q_\epsilon-1}t,$$
where $p_\epsilon=p+\alpha\epsilon,q_\epsilon=q+\beta\epsilon$.
We are aimed to search for a solution of the form
$$(u_1,u_2)=(PW_{1,\delta}+\phi_1,PW_{2,\delta}+\phi_2)$$
with $\delta=d\epsilon$, $d\in(\eta,\frac1\eta)$ for some small $0<\eta<1$ and $(\phi_1,\phi_2)\in H_\epsilon$.

For simplicity, we denote $K(h_1,h_2)=(Kh_1,Kh_2)$ with the same $K$ as in \eqref{1.8}. Hence,
\begin{align}\label{eqK}
(PW_{1,\delta}+\phi_1,PW_{2,\delta}+\phi_2)=K(f_{2,\epsilon}(PW_{2,\delta}+\phi_2),f_{1,\epsilon}(PW_{1,\delta}+\phi_1))
\end{align}
or equivalently,
\begin{align*}
&(\phi_1,\phi_2)-K(f'_{2,\epsilon}(PW_{2,\delta})\phi_2,f'_{1,\epsilon}(PW_{1,\delta})\phi_1)\\&
=K\Big(f_{2,\epsilon}(PW_{2,\delta}+\phi_2)-f'_{2,\epsilon}(PW_{2,\delta})\phi_2-f_{2,\epsilon}(PW_{2,\delta}),\\
&\qquad \quad f_{1,\epsilon}(PW_{1,\delta}+\phi_1)-f'_{1,\epsilon}(PW_{1,\delta})\phi_1-f_{1,\epsilon}(PW_{1,\delta})\Big)\\
& \quad+K(f_{2,\epsilon}(PW_{2,\delta}),f_{1,\epsilon}(PW_{1,\delta}))-(PW_{1,\delta},PW_{2,\delta}).
\end{align*}

Define
\begin{align*}
&K_{d,\epsilon}:=span\{(\delta\partial_\delta PW_{1,\delta},\delta\partial_\delta PW_{2,\delta})\}
\end{align*}
and
\begin{align*}
E_{d,\epsilon}:=\Big\{(\phi_1,\phi_2)\in H_\epsilon:
\int_\Omega\nabla\phi_1\cdot\nabla\partial_\delta PW_{2,\delta}+\nabla\phi_2\cdot\nabla\partial_\delta PW_{1,\delta}=0\Big\}.
\end{align*}

Set the projections: $$\Pi_\epsilon: H_\epsilon\rightarrow K_{d,\epsilon},\ \ \ \Pi^\bot_\epsilon: H_\epsilon\rightarrow E_{d,\epsilon}.$$

\begin{Lem}
Let $0<\eta<1$ and $\delta=d\epsilon$ with $\eta<d<\frac1\eta$. As $\epsilon\rightarrow0$, for any $(\phi_1,\phi_2)\in H_\epsilon$,
 $$ \|\Pi^\bot_\epsilon(\phi_1,\phi_2)\|_\epsilon\leq C\Big(\|(\phi_1,\phi_2)\|+\|(\phi_1,\phi_2)\|_{L^{q_\epsilon+1}\times L^{p_\epsilon+1}(\Omega)}\Big).$$

\end{Lem}

\begin{proof}
Since by the definition of the projection, $ \|\Pi^\bot_\epsilon(\phi_1,\phi_2)\|\leq C\|(\phi_1,\phi_2)\|$, then we are sufficed to consider the
 $L^{q_\epsilon+1}\times L^{p_\epsilon+1}$-norm to show that
\begin{align}\label{212}
\begin{split}
\|\Pi^\bot_\epsilon(\phi_1,\phi_2)\|_{L^{q_\epsilon+1}\times L^{p_\epsilon+1}}&\leq C\Big(\|(\phi_1,\phi_2)\|_{L^{q_\epsilon+1}\times L^{p_\epsilon+1}}+\|\Pi_\epsilon(\phi_1,\phi_2)\|_{L^{q_\epsilon+1}\times L^{p_\epsilon+1}}\Big)\\
&\leq C\Big(\|(\phi_1,\phi_2)\|_{L^{q_\epsilon+1}\times L^{p_\epsilon+1}}+\|(\phi_1,\phi_2)\|\Big).
\end{split}\end{align}

\smallskip
Firstly, by Proposition \ref{lemexpansion} and Lemma \ref{lemb3}, there hold that
\begin{align*}
\|\delta\partial_\delta PW_{1,\delta}\|_{L^{q+1}}\leq C(\|U_{e_n,\delta}\|_{L^{q+1}}+\|U_{e_n,\delta}\|_{L^{q+1}})=O(1)
\end{align*}
and
\begin{align*}
\|\delta\partial_\delta PW_{2,\delta}\|_{L^{p+1}}\leq C(\|V_{e_n,\delta}\|_{L^{p+1}}+\|V_{e_n,\delta}\|_{L^{p+1}})=O(1),
\end{align*}
which imply that
\begin{align}\label{212'}
\|(\delta\partial_\delta PW_{1,\delta},\delta\partial_\delta PW_{2,\delta})\|_{L^{q_\epsilon+1}\times L^{p_\epsilon+1}}=
O(\epsilon^{-\alpha\epsilon}+\epsilon^{-\beta\epsilon})=O(1).\end{align}
\smallskip

Secondly, we show that
writing 
\begin{align}\label{2.27}
\Pi_\epsilon(\phi_1,\phi_2)=c_0(\delta\partial_\delta PW_{1,\delta},\delta\partial_\delta PW_{2,\delta}),
\end{align}
 then
\begin{align}\label{c0}
2c_0\int_\Omega\nabla\delta\partial_\delta PW_{1,\delta}\cdot\nabla\delta\partial_\delta PW_{2,\delta}=\int_\Omega\nabla\phi_1\cdot\nabla\delta\partial_\delta PW_{2,\delta}
+\nabla\phi_2\cdot\nabla\delta\partial_\delta PW_{1,\delta}
\end{align}
 implies $c_0\leq C\|(\phi_1,\phi_2)\|$.

\smallskip
In fact, we write the left hand side of \eqref{c0} as
\begin{align}\label{2.13}
\begin{split}
&\int_\Omega\nabla\delta\partial_\delta PW_{1,\delta}\cdot\nabla\delta\partial_\delta PW_{2,\delta}
=\delta^2\int_\Omega(-\Delta)\partial_\delta PW_{1,\delta}\partial_\delta PW_{2,\delta}\\&
=\delta^2\int_\Omega\partial_\delta(V_{e_n,\delta}^p-V_{-e_n,\delta}^p)\partial_\delta PW_{2,\delta}\\&
=\delta^2\int_\Omega p(V_{e_n,\delta}^{p-1}\partial_\delta V_{e_n,\delta}-V_{-e_n,\delta}^{p-1}\partial_\delta V_{-e_n,\delta})(\partial_\delta V_{e_n,\delta}-\partial_\delta V_{-e_n,\delta})\\
&+\delta^2\int_\Omega p(V_{e_n,\delta}^{p-1}\partial_\delta V_{e_n,\delta}-V_{-e_n,\delta}^{p-1}\partial_\delta V_{-e_n,\delta})(\partial_\delta PV_{e_n,\delta}-\partial_\delta V_{e_n,\delta}+\partial_\delta PV_{-e_n,\delta}-\partial_\delta V_{-e_n,\delta})\\
&=\delta^2\int_\Omega p(V_{e_n,\delta}^{p-1}(\partial_\delta V_{e_n,\delta})^2+V_{-e_n,\delta}^{p-1}(\partial_\delta V_{-e_n,\delta})^2)+\Theta(\delta)
\end{split}\end{align}
where
\begin{align*}
&\Theta(\delta) =-\delta^2\int_\Omega p(V_{e_n,\delta}^{p-1}\partial_\delta V_{e_n,\delta}\partial_\delta V_{-e_n,\delta}+V_{-e_n,\delta}^{p-1}\partial_\delta V_{-e_n,\delta}\partial_\delta V_{e_n,\delta})\\
&+\delta^2\int_\Omega p(V_{e_n,\delta}^{p-1}\partial_\delta V_{e_n,\delta}-V_{-e_n,\delta}^{p-1}\partial_\delta V_{-e_n,\delta})(\partial_\delta PV_{e_n,\delta}-\partial_\delta V_{e_n,\delta}+\partial_\delta PV_{-e_n,\delta}-\partial_\delta V_{-e_n,\delta}).
\end{align*}

By Proposition \ref{lemexpansion}, it holds that
\begin{align*}
&|\partial_\delta PV_{e_n,\delta}-\partial_\delta V_{e_n,\delta}+\partial_\delta PV_{-e_n,\delta}-\partial_\delta V_{-e_n,\delta}|
\leq C(v_{1,\delta,e_n}+v_{1,\delta,-e_n})
\end{align*}
where $v_{i,\delta,\xi},i=1,2$ are as in Lemma \ref{lemb3}.
Hence, in view of \eqref{deltaU},
\begin{align*}
 |\Theta(\delta)| \leq &C\int_\Omega
(V_{e_n,\delta}^{p-1} V_{e_n,\delta}V_{-e_n,\delta}+V_{-e_n,\delta}^{p-1} V_{-e_n,\delta} V_{e_n,\delta})+ (V_{e_n,\delta}^{p}+V_{-e_n,\delta}^{p})(v_{1,\delta,e_n}+v_{1,\delta,-e_n})\\=&o(1).
\end{align*}
Moreover, by Lemma \ref{lemasym'}, we know that
\begin{align*}
\partial_\delta V_{\xi,\delta}\sim-\frac{\delta^{-\frac{n}{p+1}-1}\frac{|x-e_n|^2}{\delta^2}}{(1+\frac{|x-e_n|^2}{\delta^2})^{\frac{n}2}}.
\end{align*}
Then for some $\tilde C>0$,
\begin{align}\label{tildeC}
\begin{split}
&\delta^2\int_\Omega p(V_{e_n,\delta}^{p-1}(\partial_\delta V_{e_n,\delta})^2+V_{-e_n,\delta}^{p-1}(\partial_\delta V_{-e_n,\delta})^2)\\
&\geq
C\delta^2\int_\Omega\frac{\delta^{-\frac{n(p-1)}{p+1}-\frac{2n}{p+1}-2}\frac{|x-e_n|^4}{\delta^4}}{(1+\frac{|x-e_n|^2}{\delta^2})^{\frac{(n-2)(p-1)}2+n}}dx
=C\int_{\frac{\Omega-e_n}\delta}\frac{|y|^4}{(1+|y|^2)^{\frac{(n-2)(p-1)}{2}+n}}dy\geq\tilde C.
\end{split}\end{align}
Hence, from \eqref{c0} and \eqref{2.13}, we get that
$c_0\leq C\|(\phi_1,\phi_2)\|$,
which combined with \eqref{2.27}, gives
$$\|\Pi_\epsilon(\phi_1,\phi_2)\|_{L^{q_\epsilon+1}\times L^{p_\epsilon+1}}\leq C\|(\phi_1,\phi_2)\|.$$
Thus  \eqref{212'} holds true.

\end{proof}
\smallskip
One can refer to Lemma 3.1 in \cite{KP} for a similar argument showing that
 $K_{d,\epsilon}$ and $E_{d,\epsilon}$ are topological complements of each other.
\medskip

Decompose \eqref{eqK} as follows
\begin{align}
\Pi_\epsilon
(PW_{1,\delta}+\phi_1,PW_{2,\delta}+\phi_2)=\Pi_\epsilon\circ K(f_{2,\epsilon}(PW_{2,\delta}+\phi_2),f_{1,\epsilon}(PW_{1,\delta}+\phi_1))\label{Pi},\\
\Pi_\epsilon^\bot
(PW_{1,\delta}+\phi_1,PW_{2,\delta}+\phi_2)=\Pi_\epsilon^\bot\circ K(f_{2,\epsilon}(PW_{2,\delta}+\phi_2),f_{1,\epsilon}(PW_{1,\delta}+\phi_1)).\label{Pibot}
\end{align}
Moreover, \eqref{Pibot} can be written as
\begin{align}\label{eqbot}
L_{d,\epsilon}(\phi_1,\phi_2)=N_{d,\epsilon}(\phi_1,\phi_2)+R_{d,\epsilon},
\end{align}
where
\begin{align*}
&L_{d,\epsilon}(\phi_1,\phi_2)=\Pi_\epsilon^\perp\Big((\phi_1,\phi_2)-K(f'_{2,\epsilon}(PW_{2,\delta})\phi_2,f'_{1,\epsilon}(PW_{1,\delta})\phi_1)\Big),\\
&N_{d,\epsilon}(\phi_1,\phi_2)=\Pi_\epsilon^\perp\circ K\Big(f_{2,\epsilon}(PW_{2,\delta}+\phi_2)-f'_{2,\epsilon}(PW_{2,\delta})\phi_2-f_{2,\epsilon}(PW_{2,\delta}),\\
&\qquad\qquad \qquad\qquad+f_{1,\epsilon}(PW_{1,\delta}+\phi_1)-f'_{1,\epsilon}(PW_{1,\delta})\phi_1-f_{1,\epsilon}(PW_{1,\delta})\Big),\\
&R_{d,\epsilon}=\Pi_\epsilon^\perp\Big(K(f_{2,\epsilon}(PW_{2,\delta}),f_{1,\epsilon}(PW_{1,\delta}))-(PW_{1,\delta},PW_{2,\delta})\Big).
\end{align*}

\medskip

\section{Reduction to a finite dimensional problem}
In this section, we are devoted to solve \eqref{eqbot} and prove the following result.

\begin{Prop}\label{prop3.1}
For every $0<\eta<1$ sufficiently small, there exist some $\epsilon_0>0$ and $C>0$ such that if $\epsilon\in(0,\epsilon_0)$ and $d\in(\eta,\frac1\eta)$, then the
problem \eqref{eqbot} has a unique solution $(\phi_1,\phi_2)=(\phi_{1,d,\epsilon},\phi_{2,d,\epsilon})\in E_{d,\epsilon}$ satisfying that for every $\sigma>0$, as $\epsilon\rightarrow0$
$
\|(\phi_1,\phi_2)\|_\epsilon=o(\epsilon^{1-\sigma})
$ uniformly in $d\in(\eta,\frac1\eta)$. Moreover the map $(\eta,\frac1\eta)\rightarrow E_{d,\epsilon}, d\mapsto(\phi_{1,d,\epsilon},\phi_{2,d,\epsilon})$ is of class $C^1$.

\end{Prop}
\medskip

\subsection{Invertibility of $L_{d,\epsilon}$}

\begin{Prop}\label{propL}
For every $0<\eta<1$ sufficiently small, there exists some $\epsilon_0>0$ and $C>0$ such that if $\epsilon\in(0,\epsilon_0)$ and $d\in(\eta,\frac1\eta)$,
then $$\|L_{d,\epsilon}(\phi_1,\phi_2)\|_\epsilon\geq C\|(\phi_1,\phi_2)\|_\epsilon,\ \ \forall(\phi_{1},\phi_{2})\in E_{d,\epsilon}.$$
Moreover, $L_{d,\epsilon}$ is invertible in $E_{d,\epsilon}.$
\end{Prop}

\begin{proof}
By contradiction, we assume that there exists some $\eta\in(0,1), d_k\in(\eta,\frac1\eta)$ with $d_k\rightarrow d_0,\epsilon_k\rightarrow0$ and $(\phi_{1,k},\phi_{2,k})\in E_{d_k,\epsilon_k}$ such that $\|(\phi_{1,k},\phi_{2,k})\|_\epsilon=1$ and $\|(h_{1,k},h_{2,k})\|_\epsilon\rightarrow0$,
where $(h_{1,k},h_{2,k})=L_{d_k,\epsilon_k}(\phi_{1,k},\phi_{2,k})\in E_{d_k,\epsilon_k}$.

We denote that $\delta_k=d_k\epsilon_k$, $$Z_{i,\delta}=\delta\partial_{\delta} W_{i,\delta}.$$
 Hence
\begin{align}\label{eqzk}
(\phi_{1,k},\phi_{2,k})-K(f'_{2,\epsilon}(PW_{2,\delta})\phi_{2,k},f'_{1,\epsilon}(PW_{1,\delta})\phi_{1,k})=(h_{1,k},h_{2,k})+(z_{1,k},z_{2,k})
\end{align}
with $c_k\in\R$, $$(z_{1,k},z_{2,k})=c_k(PZ_{1,\delta_k},PZ_{2,\delta_k})\in K_{d_k,\epsilon_k}.$$
\medskip

{\bf Step 1.}  We prove that $\|(z_{1,k},z_{2,k})\|_{\epsilon_k}\rightarrow0$ as $k\rightarrow+\infty$.

Since $$\int_\Omega\nabla\phi_{1,k}\cdot\nabla z_{2,k}+\nabla\phi_{2,k}\cdot\nabla z_{1,k}=0=\int_\Omega\nabla h_{1,k}\cdot\nabla z_{2,k}+\nabla h_{2,k}\cdot\nabla z_{1,k},$$
 we test \eqref{eqzk} with $(z_{1,k},z_{2,k})$ to get that
\begin{align}\label{3.7}
\int_\Omega\nabla z_{1,k}\cdot\nabla z_{2,k}+\nabla z_{2,k}\cdot\nabla z_{1,k}=-c_k\int_\Omega f'_{2,\epsilon_k}(PW_{2,\delta_k})\phi_{2,k}PZ_{2,\delta_k}
+f'_{1,\epsilon_k}(PW_{1,\delta_k})\phi_{1,k}PZ_{1,\delta_k}.
\end{align}
By denoting $$\partial_\delta V_{\pm e_n,\delta_k}=\partial_\delta V_{\pm e_n,\delta}|_{\delta=\delta_k}, \ \ \partial_\delta U_{\pm e_n,\delta_k}=\partial_\delta U_{\pm e_n,\delta}|_{\delta=\delta_k},$$
it holds that
\begin{align*}
0=&\int_\Omega\nabla\phi_{1,k}\cdot\nabla z_{2,k}+\nabla\phi_{2,k}\cdot\nabla z_{1,k}=c_k\int_\Omega (-\Delta PZ_{2,\delta}\phi_{1,k}-\Delta PZ_{1,\delta}\phi_{2,k})\\&=
c_k\int_\Omega (f'_{2,\epsilon_k}(V_{e_n,\delta_k})\partial_\delta V_{e_n,\delta_k}-f'_{2,\epsilon_k}(V_{-e_n,\delta_k})\partial_\delta V_{-e_n,\delta_k})\phi_{2,k}\\
&\qquad+ (f'_{1,\epsilon_k}(U_{e_n,\delta_k})\partial_\delta U_{e_n,\delta_k}-f'_{1,\epsilon_k}(U_{-e_n,\delta_k})\partial_\delta U_{-e_n,\delta_k})\phi_{1,k}.
\end{align*}
Hence, the right hand side of \eqref{3.7} becomes
\begin{align}\label{rhs}
\begin{split}
&-c_k\int_\Omega f'_{2,\epsilon_k}(PW_{2,\delta_k})\phi_{2,k}(PZ_{2,\delta_k}-Z_{2,\delta_k})
+( f'_{2,\epsilon_k}(PW_{2,\delta_k})- f'_{2,0}(PW_{2,\delta_k}))\phi_{2,k}Z_{2,\delta_k}\\&
+(f'_{2,0}(PW_{2,\delta_k})Z_{2,\delta_k}-f'_{2,\epsilon_k}(V_{e_n,\delta_k})\delta_k\partial_\delta V_{e_n,\delta_k}+f'_{2,\epsilon_k}(V_{-e_n,\delta_k})\delta_k\partial_\delta V_{-e_n,\delta_k})\phi_{2,k}\\&
-c_k\int_\Omega f'_{1,\epsilon_k}(PW_{1,\delta_k})\phi_{1,k}(PZ_{1,\delta_k}-Z_{1,\delta_k})
+( f'_{1,\epsilon_k}(PW_{1,\delta_k})- f'_{1,0}(PW_{1,\delta_k}))\phi_{1,k}Z_{1,\delta_k}\\&
+(f'_{1,0}(PW_{1,\delta_k})Z_{1,\delta_k}-f'_{1,\epsilon_k}(U_{e_n,\delta_k})\delta_k\partial_\delta U_{e_n,\delta_k}+f'_{1,\epsilon_k}(U_{-e_n,\delta_k})\delta_k\partial_\delta U_{-e_n,\delta_k})\phi_{1,k}\\
&:=I_1+I_2+I_3+J_1+J_2+J_3.
\end{split}\end{align}

From Proposition \ref{lemexpansion}, Lemma \ref{lemb3}, Lemma \ref{lemb13}, we get that
\begin{align}\label{I1}
\begin{split}
&|I_1|=\Big|-c_k\int_\Omega f'_{2,\epsilon_k}(PW_{2,\delta_k})\phi_{2,k}(PZ_{2,\delta_k}-Z_{2,\delta_k})\Big|\\
&\leq C\|f'_{2,\epsilon_k}(PW_{2,\delta_k})\|_{L^{\frac{p+1}{p-1}}}\|\phi_{2,k}\|_{L^{p+1}}\|PZ_{2,\delta_k}-Z_{2,\delta_k}\|_{L^{p+1}}=O(\delta)=o(1).
\end{split}\end{align}
Similarly, using Lemma \ref{lemb3}, Lemma \ref{lemb13} and the fact \eqref{deltaU},
\begin{align}\label{I2}
\begin{split}
&|I_2|=\Big|\int_\Omega( f'_{2,\epsilon_k}(PW_{2,\delta_k})- f'_{2,0}(PW_{2,\delta_k}))\phi_{2,k}Z_{2,\delta_k}\Big|\\
&\leq C\| f'_{2,\epsilon_k}(PW_{2,\delta_k})- f'_{2,0}(PW_{2,\delta_k})\|_{L^{\frac{p+1}{p-1}}}\|\phi_{2,k}\|_{L^{p+1}}\|Z_{2,\delta_k}\|_{L^{p+1}}=O(\epsilon^{1-\gamma})=o(1).
\end{split}\end{align}

Results for $J_1$ and $J_2$ can be followed by similar estimates.

For $I_3$,
\begin{align}\label{I3'}
\begin{split}
&|I_3|=\Big|\int_\Omega(f'_{2,0}(PW_{2,\delta_k})Z_{2,\delta_k}-f'_{2,\epsilon_k}(V_{e_n,\delta_k})\delta_k\partial_\delta V_{e_n,\delta_k}+f'_{2,\epsilon_k}(V_{-e_n,\delta_k})\delta_k\partial_\delta V_{-e_n,\delta_k})\phi_{2,k}\Big|\\
&\leq  \int_\Omega|f'_{2,0}(PW_{2,\delta_k})-f'_{2,\epsilon_k}(V_{e_n,\delta_k})||\delta_k\partial_\delta V_{e_n,\delta_k}||\phi_{2,k}|\\
&\quad+ \int_\Omega|f'_{2,0}(PW_{2,\delta_k})-f'_{2,\epsilon_k}(V_{-e_n,\delta_k})||\delta_k\partial_\delta V_{-e_n,\delta_k}||\phi_{2,k}|.
\end{split}\end{align}
For $p-1\leq1$, for example, it holds that
\begin{align*}
&|f'_{2,0}(PW_{2,\delta_k})- f'_{2,0}(V_{e_n,\delta_k})|=
||PW_{2,\delta_k}|^{p-1}- |V_{e_n,\delta_k}|^{p-1}|\\
&\leq C|PW_{2,\delta_k}-V_{e_n,\delta_k}|^{p-1}
\leq C|PV_{e_n,\delta_k}-V_{e_n,\delta_k}|^{p-1}+C|PV_{-e_n,\delta_k}|^{p-1}.
\end{align*}
Then we have
\begin{align*}
&\int_\Omega|f'_{2,0}(PW_{2,\delta_k})-f'_{2,\epsilon_k}(V_{e_n,\delta_k})||\delta_k\partial_\delta V_{e_n,\delta_k}||\phi_{2,k}|\\
&\quad+ \int_\Omega|f'_{2,0}(PW_{2,\delta_k})-f'_{2,\epsilon_k}(V_{-e_n,\delta_k})||\delta_k\partial_\delta V_{-e_n,\delta_k}||\phi_{2,k}|\\
&
\leq C \int_\Omega|PV_{e_n,\delta_k}-V_{e_n,\delta_k}|^{p-1}|\delta_k\partial_\delta V_{e_n,\delta_k}||\phi_{2,k}|+C \int_\Omega|PV_{-e_n,\delta_k}|^{p-1}|\delta_k\partial_\delta V_{e_n,\delta_k}||\phi_{2,k}|\\
&\leq C\|PV_{e_n,\delta_k}-V_{e_n,\delta_k}\|_{L^{p+1}}^{p-1}\|\phi_{2,k}\|_{L^{p+1}}\|V_{e_n,\delta_k}\|_{L^{p+1}}
+C \int_\Omega|PV_{-e_n,\delta_k}|^{p-1}|\delta_k\partial_\delta V_{e_n,\delta_k}||\phi_{2,k}|\\
&\leq o(1)+C\int_\Omega\frac{\delta_k^{-\frac n{p+1}}}{(1+\frac{|x-e_n|}{\delta_k})^{n-2}}\frac{\delta_k^{-\frac {n(p-1)}{p+1}}}{(1+\frac{|x+e_n|}{\delta_k})^{(p-1)(n-2)}}
|\phi_{2,k}(x)|\\
&\leq o(1)+C\frac{\delta_k^{n-2}}{|e_n+e_n|^{n-2}}\int_\Omega\delta_k^{-\frac{np}{p+1}}\Big(\frac{1}{(1+\frac{|x-e_n|}{\delta_k})^{(p-1)(n-2)}}+\frac{1}{(1+\frac{|x+e_n|}{\delta_k})^{(p-1)(n-2)}}\Big)
|\phi_{2,k}(x)|\\
&\leq o(1)+C\frac{\delta_k^{n-2}}{|e_n+e_n|^{n-2}}\left(\int_\Omega\delta_k^{-n}\Big(\frac{1}{(1+\frac{|x-e_n|}{\delta_k})^{(p-1)(n-2)}}\Big)^{\frac{p+1}p}\right)^{\frac p{p+1}}
\|\phi_{2,k}\|_{L^{p+1}}\\
&\leq o(1)+\frac{C\delta_k^{n-2}}{|e_n+e_n|^{n-2}}\Big(\int_{\frac{\Omega-e_n}{\delta_k}}\frac1{(1+|y|)^{(p-1)(n-2)\frac{p+1}p}}\Big)^{\frac p{p+1}}\|\phi_{2,k}\|_{L^{p+1}}\\
&=o(1)+\begin{cases}\dis O(\delta_k^{n-2-\epsilon_0}) &if\ (p-1)(n-2)\frac{p+1}p\geq n,\\
\dis O(\delta_k^{\frac{pn}{q+1}}) &if\ (p-1)(n-2)\frac{p+1}p<n
\end{cases}
=o(1),
\end{align*}
where $\epsilon_0>0$ is any small constant.

So \eqref{I3'} gives $I_3=o(1)$.
\smallskip

For $J_3$,  since $q-1\leq1$ can be handled similar as $I_3$, we are sufficed to deal with the case when $q-1>1$ as follows
\begin{align*}
&|f'_{1,0}(PW_{1,\delta_k})- f'_{1,0}(U_{e_n,\delta_k})|=
||PW_{1,\delta_k}|^{q-1}- |U_{e_n,\delta_k}|^{q-1}|\\
&\leq C(|PW_{1,\delta_k}-U_{e_n,\delta_k}|^{q-1}+|U_{e_n,\delta_k}|^{q-2}|PW_{1,\delta_k}-U_{e_n,\delta_k}|)\\
&\leq C\Big(|PU_{e_n,\delta_k}-U_{e_n,\delta_k}|^{q-1}+|PU_{-e_n,\delta_k}|^{q-1}+|U_{e_n,\delta_k}|^{q-2}(|PU_{e_n,\delta_k}-U_{e_n,\delta_k}|
+|PU_{-e_n,\delta_k}|)\Big).
\end{align*}
then
\begin{align*}
&\int_\Omega|f'_{1,0}(PW_{1,\delta_k})-f'_{1,\epsilon_k}(U_{e_n,\delta_k})||\delta_k\partial_\delta U_{e_n,\delta_k}||\phi_{1,k}|\\
&\quad+ \int_\Omega|f'_{1,0}(PW_{1,\delta_k})-f'_{1,\epsilon_k}(U_{-e_n,\delta_k})||\delta_k\partial_\delta U_{-e_n,\delta_k}||\phi_{1,k}|\\
&
\leq C \int_\Omega|PU_{e_n,\delta_k}-U_{e_n,\delta_k}|^{q-1}|\delta_k\partial_\delta U_{e_n,\delta_k}||\phi_{1,k}|+C \int_\Omega|PU_{-e_n,\delta_k}|^{q-1}|\delta_k\partial_\delta U_{e_n,\delta_k}||\phi_{1,k}|\\
&\quad+ \int_\Omega |U_{e_n,\delta_k}|^{q-2}|PU_{e_n,\delta_k}-U_{e_n,\delta_k}||\delta_k\partial_\delta U_{e_n,\delta_k}||\phi_{1,k}|
+ \int_\Omega |U_{e_n,\delta_k}|^{q-2}
|PU_{-e_n,\delta_k}||\delta_k\partial_\delta U_{e_n,\delta_k}||\phi_{1,k}|.
\end{align*}
At this time, similar as \eqref{I3'}, we can use Lemma \ref{lemb3} to estimate $J_3=o(1)$ in two cases: $p>\frac n{n-2}$ and $p<\frac n{n-2}$ respectively.
\smallskip

The left hand side of \eqref{3.7} gives that
\begin{align*}
&c_k^2\int\nabla\delta \partial_\delta PW_{1,\delta_k}\cdot\nabla\delta \partial_\delta PW_{2,\delta_k}+\nabla\delta \partial_\delta PW_{2,\delta_k}\cdot\nabla\delta \partial_\delta PW_{1,\delta_k}\\&=
c_k^2\delta^2\int_\Omega (-\Delta \partial_\delta PW_{1,\delta_k} \partial_\delta PW_{2,\delta_k}-\Delta \partial_\delta PW_{2,\delta_k} \partial_\delta PW_{1,\delta_k})\\&=
2c_k^2\delta^2\int_\Omega \partial_\delta(V_{e_n,\delta_k}^p-V_{-e_n,\delta_k}^p) \partial_\delta PW_{2,\delta_k}\\
&=2pc_k^2\int_\Omega\Big(V_{e_n,\delta_k}^{p-1}(\delta\partial_\delta V_{e_n,\delta_k})^2+V_{-e_n,\delta_k}^{p-1}(\delta\partial_\delta V_{-e_n,\delta_k})^2\Big)+\xi(\delta_k)
\end{align*}
with $\xi(\delta_k)=o(1)$.
Moreover, similar as \eqref{tildeC}, by Lemma \ref{lemasym} and Lemma \ref{lemasym'}, there exists some constant $c'>0$ such that
\begin{align}\label{left}
\begin{split}
&\int_\Omega\Big(V_{e_n,\delta_k}^{p-1}(\delta_k\partial_\delta V_{e_n,\delta_k})^2+V_{-e_n,\delta_k}^{p-1}(\delta_k\partial_\delta V_{-e_n,\delta_k})^2\Big)
\geq c'>0.\end{split}
\end{align}
Hence the left hand side of \eqref{3.7} $\geq c>0$, which combined
 with the right hand side estimates \eqref{I1}-\eqref{I3'} and the estimates on $J_1-J_3$, we have that
$c_k\rightarrow0$. Recalling that $(z_{1,k},z_{2,k})=c_k(PZ_{1,\delta_k},PZ_{2,\delta_k})\in K_{d_k,\epsilon_k}$, then we arrive at
$\|(z_{1,k},z_{2,k})\|_\epsilon\rightarrow0$.

\medskip

{\bf Step 2. }  We show \begin{align}\label{>0}
\liminf_{k\rightarrow+\infty}(\|f'_{2,\epsilon_k}(PW_{2,\delta_k})u_{2,k}\|_{L^{\frac{p+1}p}}+\|f'_{1,\epsilon_k}(PW_{1,\delta_k})u_{1,k}\|_{L^{\frac{q+1}q}})>0.
\end{align}

By \eqref{eqzk}, we set for $i=1,2$ that
$$u_{i,k}:=\phi_{i,k}-h_{i,k}-z_{i,k}$$ and then
\begin{align}\label{equk}
\begin{split}
&(u_{1,k},u_{2,k}) =K(f'_{2,\epsilon}(PW_{2,\delta})\phi_{2,k},f'_{1,\epsilon}(PW_{1,\delta})\phi_{1,k})\\& \quad
=K\Big(f'_{2,\epsilon_k}(PW_{2,\delta_k})u_{2,k}+f'_{2,\epsilon_k}(PW_{2,\delta_k})(h_{2,k}+z_{2,k}),\\
&\qquad \qquad  f'_{1,\epsilon_k}(PW_{1,\delta_k})u_{1,k}+f'_{1,\epsilon_k}(PW_{1,\delta_k})(h_{1,k}+z_{1,k})\Big)
\end{split}\end{align}
or
\begin{align}\label{equk0}
\begin{cases}
-\Delta u_{1,k}=f'_{2,\epsilon_k}(PW_{2,\delta_k})u_{2,k}+f'_{2,\epsilon_k}(PW_{2,\delta_k})(h_{2,k}+z_{2,k})\\
-\Delta u_{2,k} =f'_{1,\epsilon_k}(PW_{1,\delta_k})u_{1,k}+f'_{1,\epsilon_k}(PW_{1,\delta_k})(h_{1,k}+z_{1,k})\\
\int_\Omega u_{1,k}=\int_\Omega u_{2,k}=0,\ \ \ \partial_\nu u_{1,k}=\partial_\nu u_{2,k}=0\ on\ \partial\Omega.
\end{cases}\end{align}
Now we claim that
\begin{align}\label{3.13}
\liminf_{k\rightarrow+\infty}\|(u_{1,k},u_{2,k})\|>0.
\end{align}

In fact, in view of the definition of $K$, set $q_k=q_{\epsilon_k},p_k=p_{\epsilon_k}$,
\begin{align}\label{1}
\begin{split}
 \|u_{1,k}\|_{L^{q_k+1}}&\leq C\|K(f'_{2,\epsilon_k}(PW_{2,\delta_k})u_{2,k}+f'_{2,\epsilon_k}(PW_{2,\delta_k})(h_{2,k}+z_{2,k}))\|_{L^{q_k+1}}\\
&\leq C\|K(f'_{2,\epsilon_k}(PW_{2,\delta_k})u_{2,k}+f'_{2,\epsilon_k}(PW_{2,\delta})(h_{2,k}+z_{2,k}))\|_{\dot W^{2,\tilde p_k}}\\
&\leq C\|f'_{2,\epsilon_k}(PW_{2,\delta_k})u_{2,k}+f'_{2,\epsilon_k}(PW_{2,\delta_k})(h_{2,k}+z_{2,k})\|_{L^{\tilde p_k}}\\
&\leq C\|f'_{2,\epsilon_k}(PW_{2,\delta_k})\|_{L^{\bar p_k}}\|u_{2,k}\|_{L^{p+1}}+\|f'_{2,\epsilon_k}(PW_{2,\delta_k})\|_{L^{\hat p_k}}\|h_{2,k}+z_{2,k}\|_{L^{p_k+1}}
\end{split}
\end{align}
where $$\tilde p_k=\frac1{\frac2n+\frac1{q_k+1}},\bar p_k=\frac1{\frac2n+\frac1{q_k+1}-\frac1{p+1}}, \hat p_k=\frac1{\frac2n+\frac1{q_k+1}-\frac1{p_k+1}}$$ and we have used the fact that
$\bar p_k=\frac{p+1}{p-1}+O(\epsilon_k)$ and $\|f'_{2,\epsilon_k}(PW_{2,\delta_k})\|_{L^{\bar p_k}}=O(1)$ as $k\rightarrow+\infty$.

Similarly, we also have
\begin{align}\label{2}
\begin{split}
 \|u_{2,k}\|_{L^{p_k+1}}&\leq C\|u_{2,k}\|_{L^{q+1}}+o(1).
\end{split}
\end{align}

Combining   \eqref{1} and \eqref{2}, we get that if $\|(u_{1,k},u_{2,k}) \|\rightarrow0$ then also
$\|(u_{1,k},u_{2,k}) \|_{L^{q_k+1}\times L^{p_k+1}}\rightarrow0$, which means $\|(u_{1,k},u_{2,k}) \|_{\epsilon}\rightarrow0$,
a contradiction with the assumption $\|(\phi_{1,k},\phi_{2,k})\|_{\epsilon_k}=1$.

Hence, $\liminf_{k\rightarrow+\infty}\|(u_{1,k},u_{2,k}) \|>0$, and,  by the equations \eqref{equk0}, we  have proved  \eqref{>0}.

\medskip

{\bf Step 3.}  In order to carry out a blowing up procedure, we are inspired by
\cite{arxiv} to choose the extension operator $E: W^{1,p^*}(\Omega)\times W^{1,q^*}(\Omega)\rightarrow W^{1,p^*}(\R^n)\times W^{1,q^*}(\R^n)$ such that $E(u_{1,k},u_{2,k})=(u_{1,k},u_{2,k})$ in $\Omega$ and can obtain
\begin{align}\label{ext}
\begin{split}
&\|E(u_{1,k},u_{2,k})\|_{\dot W^{1,p^*}\times\dot W^{1,q^*}(\R^n)}\leq \|E(u_{1,k},u_{2,k})\|_{ W^{1,p^*}\times W^{1,q^*}(\R^n)}\\&\leq C\| (u_{1,k},u_{2,k})\|_{ W^{1,p^*}\times W^{1,q^*}(\Omega)}\leq C\|(u_{1,k},u_{2,k})\|_{\dot W^{1,p^*}\times\dot W^{1,q^*}(\Omega)},
\end{split}\end{align}
where, in the last inequalities,  we use Poincar\'e-Wirtinger inequality since $u_{i,k}(i=1,2)$ has zero average in $\Omega$,
and $p^*$ and $q^*$ are defined in \eqref{emb}.
We identify $u_{i,k}(i=1,2)$ with its extension.

Now we use the blowing up to set
 $$\hat u_{1,k}(y)=\delta_k^{\frac n{q+1}}u_{1,k}(\delta_ky+e_n),\ \ \hat u_{2,k}(y)=\delta_k^{\frac n{p+1}}u_{2,k}(\delta_ky+e_n)$$
and $\Omega_k=\frac{\Omega-e_n}{\delta_k}=B_{\frac1{\delta_k}}(-\frac{e_n}{\delta_k}).$
Then \eqref{ext} turns out
\begin{align}\label{ext'}
\begin{split}
&\|(\hat u_{1,k},\hat u_{2,k})\|_{\dot W^{1,p^*}\times\dot W^{1,q^*}(\R^n)}= \|(u_{1,k},u_{2,k})\|_{\dot W^{1,p^*}\times\dot W^{1,q^*}(\R^n)}\\&\leq \|(u_{1,k},u_{2,k})\|_{\dot W^{1,p^*}\times\dot W^{1,q^*}(\Omega)}\leq C\|(\hat u_{1,k},\hat u_{2,k})\|_{\dot W^{1,p^*}\times\dot W^{1,q^*}(\Omega_k)},
\end{split}\end{align}
which means that $(\hat u_{1,k},\hat u_{2,k})$ is bounded in $\dot W^{1,p^*}(\R^n)\times \dot W^{1,q^*}(\R^n)$.
Up to some subsequence, there is some $(\hat u_1,\hat u_2)\in \dot W^{1,p^*}(\R^n)\times \dot W^{1,q^*}(\R^n)$ such that
$$\hat u_{1,k}\rightharpoonup\hat u_1\ \ weakly\ in\ \dot W^{1,p^*}(\R^n),\ \ \ \ \hat u_{2,k}\rightharpoonup\hat u_2\ \ weakly\ in\ \dot W^{1,q^*}(\R^n)$$
and we show that $(\hat u_1,\hat u_2)=0$.

From the equations of $(u_{1,k},u_{2,k})$ we know that
\begin{align}\label{equkhat}
\begin{cases}
-\Delta \hat u_{1,k}=\delta_k^{\frac{n(p-1)}{p+1}}f'_{2,\epsilon_k}(PW_{2,\delta_k}(\delta_ky+e_n))\hat u_{2,k}+\delta_k^{\frac{n(p-1)}{p+1}}f'_{2,\epsilon_k}(PW_{2,\delta_k}(\delta_ky+e_n))(\hat h_{2,k}+\hat z_{2,k})\\
-\Delta \hat u_{2,k} =\delta_k^{\frac{n(q-1)}{q+1}}f'_{1,\epsilon_k}(PW_{1,\delta_k}(\delta_ky+e_n))\hat u_{1,k}+\delta_k^{\frac{n(q-1)}{q+1}}f'_{1,\epsilon_k}(PW_{1,\delta_k}(\delta_ky+e_n))(\hat h_{1,k}+\hat z_{1,k})\\
\int_{\Omega_k} \hat u_{1,k}=\int_{\Omega_k} \hat u_{2,k}=0,\ \ \ \partial_\nu \hat u_{1,k}=\partial_\nu \hat u_{2,k}=0\ on\ \partial\Omega_k,
\end{cases}\end{align}
where $\hat h_{1,k}(y)=\delta_k^{\frac n{q+1}}h_{1,k}(\delta_ky+e_n),\hat h_{2,k}(y)=\delta_k^{\frac n{p+1}}h_{2,k}(\delta_ky+e_n),\hat z_{1,k}(y)=\delta_k^{\frac n{q+1}}z_{1,k}(\delta_ky+e_n),\hat z_{2,k}(y)=\delta_k^{\frac n{p+1}}z_{2,k}(\delta_ky+e_n)$.

Take $\psi_2,\psi_1\in C_c^\infty(\overline{\R_+^n})$ as the test functions of two equations in \eqref{equkhat} respectively to obtain that
\begin{align*}
&\int_{\R_+^n}\nabla \hat u_1\cdot\nabla\psi_2+\nabla \hat u_2\cdot\nabla\psi_1
=\lim_{k\rightarrow+\infty}\int_{\Omega_k}\nabla \hat u_{1,k}\cdot\nabla\psi_2+\nabla \hat u_{2,k}\cdot\nabla\psi_1\\
&=\lim_{k\rightarrow+\infty}\int_{\Omega_k}\delta_k^{\frac{n(p-1)}{p+1}}f'_{2,\epsilon_k}(PW_{2,\delta_k}(\delta_ky+e_n))\hat u_{2,k}(y)\psi_2(y)dy\\&+
\lim_{k\rightarrow+\infty}\int_{\Omega_k}\delta_k^{\frac{n(q-1)}{q+1}}f'_{1,\epsilon_k}(PW_{1,\delta_k}(\delta_ky+e_n))\hat u_{1,k}(y)\psi_1(y)dy\\
&+\lim_{k\rightarrow+\infty}\int_{\Omega_k}\delta_k^{\frac{n(p-1)}{p+1}}f'_{2,\epsilon_k}(PW_{2,\delta_k}(\delta_ky+e_n))(\hat h_{2,k}+\hat z_{2,k})\psi_2(y)dy\\
&+\lim_{k\rightarrow+\infty}\int_{\Omega_k}\delta_k^{\frac{n(q-1)}{q+1}}f'_{1,\epsilon_k}(PW_{1,\delta_k}(\delta_ky+e_n))(\hat h_{1,k}+\hat z_{1,k})\psi_1(y)dy\\
&=\int_{\R_+^n}f'_{2,0}(V)\hat u_2\psi_2dy+\int_{\R_+^n}f'_{1,0}(U)\hat u_1\psi_1dy,
\end{align*}
which implies that $(\hat u_1,\hat u_2)$ would be a solution of
\begin{align}\label{eqhat}
\begin{cases}
-\Delta \hat u_{1}=f'_{2,0}(V)\hat u_{2},\ &in\ \R^n_+\\
-\Delta \hat u_{2} =f'_{1,0}(U)\hat u_{1},\ &in\ \R^n_+\\
 \partial_\nu \hat u_{1}=\partial_\nu \hat u_{2}=0\ on\ \partial\R^n_+.
\end{cases}\end{align}
If we identify $\hat u_1,\hat u_2$ with their even reflections with respect to $\partial\R^n_+$, then
\begin{align}\label{eqhat0}
\begin{cases}
-\Delta \hat u_{1}=f'_{2,0}(V)\hat u_{2},\ &in\ \R^n\\
-\Delta \hat u_{2} =f'_{1,0}(U)\hat u_{1},\ &in\ \R^n\\
 (\hat u_{1},\hat u_{2})\in \dot W^{2,\frac{p+1}p}(\R^n)\times\dot W^{2,\frac{q+1}q}(\R^n).
\end{cases}\end{align}
Since $(u_{1,k},u_{2,k})\in H_{\epsilon_k}$, we see that $u_{1,k},u_{2,k}$ are both even with respect to $x_1,\ldots,x_{n-1}$. So
$\hat u_{1,k},\hat u_{2,k}$ are   even with respect to $x_1,\ldots,x_{n-1}$ too, which means $\hat u_1,\hat u_2$ are also even  with respect to $x_1,\ldots,x_{n-1}$.
Therefore, we obtain that $(\hat u_1,\hat u_2)=c_0(\partial_\delta U_{0,\delta}|_{\delta=1},\partial_\delta V_{0,\delta}|_{\delta=1})$.

We claim that $c_0=0$.

In fact, since $u_{1,k},u_{2,k}$ are odd with respect to $x_n$ and $\Omega=B_1(0)$, it holds that
\begin{align*}
\int_\Omega u_{1,k}f'_{1,0}(U_{e_n,\delta_k})\delta_k\partial_\delta U_{e_n,\delta_k}=-\int_\Omega u_{1,k}f'_{1,0}(U_{-e_n,\delta_k})\delta_k\partial_\delta U_{-e_n,\delta_k},\\
\int_\Omega u_{2,k}f'_{2,0}(V_{e_n,\delta_k})\delta_k\partial_\delta V_{e_n,\delta_k}=-\int_\Omega u_{2,k}f'_{2,0}(V_{-e_n,\delta_k})\delta_k\partial_\delta V_{-e_n,\delta_k}.
\end{align*}
Thus, setting $(\Psi,\Phi)=(\partial_\delta U_{0,\delta}|_{\delta=1},\partial_\delta V_{0,\delta}|_{\delta=1})$ and by the symmetry of $\Omega$, we get
\begin{align*}
c_0\int_{\R^n}\nabla\Psi\nabla\Phi&=2\int_{\R^n_+}\nabla\Psi\nabla \hat u_{2}=2\int_{\Omega_k}\nabla\Psi\nabla \hat u_{2,k}+o(1)\\
&=2\int_{\Omega_k} \hat u_{2,k}f'_{2,0}(V_{0,1})\Psi+o(1)
=2\int_{\Omega}u_{2,k}f'_{2,0}(V_{e_n,\delta_k})\delta_k\partial_\delta V_{e_n,\delta_k}+o(1)\\
&=\int_{\Omega}u_{2,k}\Big(f'_{2,0}(V_{e_n,\delta_k})\delta_k\partial_\delta V_{e_n,\delta_k}-f'_{2,0}(V_{-e_n,\delta_k})\delta_k\partial_\delta V_{-e_n,\delta_k}\Big)+o(1).
\end{align*}
Similarly, we also have
\begin{align*}
c_0\int_{\R^n}\nabla\Psi\nabla\Phi&=2\int_{\R^n_+}\nabla\Psi\nabla \hat u_{2}=\int_{\Omega}u_{1,k}\Big(f'_{1,0}(U_{e_n,\delta_k})\delta_k\partial_\delta U_{e_n,\delta_k}-f'_{1,0}(U_{-e_n,\delta_k})\delta_k\partial_\delta U_{-e_n,\delta_k}\Big)+o(1).
\end{align*}
Hence, recalling that $u_{i,k}=\phi_{i,k}-h_{i,k}-z_{i,k}$ for $i=1,2$, we have
\begin{align*}
2c_0\int_{\R^n}\nabla\Psi\nabla\Phi=2\int_{\R^n_+}\nabla\Psi\nabla \hat u_{2}&=\int_\Omega\nabla u_{1,k}\cdot\nabla PZ_{2,\delta_k}+\nabla u_{2,k}\cdot\nabla PZ_{1,\delta_k}+o(1)\\&
=\int_\Omega\nabla z_{1,k}\cdot\nabla PZ_{2,\delta_k}+\nabla z_{2,k}\cdot\nabla PZ_{1,\delta_k}+o(1)=o(1),
\end{align*}
which, by a similar argument as in \eqref{left}, gives $c_0=0$ and so $(\hat u_1,\hat u_2)=0$.

\medskip

{\bf Step 4.}  Now we are in position to show a contradiction with \eqref{>0}.

As we did in \eqref{omegasplit}, we split $\Omega=\Omega_+\cup\Omega_-\cup\Omega_c$ with $\Omega_c:=\Omega\setminus(\Omega_+\cup\Omega_-)$.
First on $\Omega_c$,
\begin{align}\label{f1}
\begin{split}
&\|f'_{2,\epsilon_k}(PW_{2,\delta_k})u_{2,k}\|_{L^{\frac{p+1}p}(\Omega_c)}
\leq C\Big(\int_{\Omega_c}V_{e_n,\delta_k}^{\frac{(p_k-1)(p+1)}{p-1}}+V_{-e_n,\delta_k}^{\frac{(p_k-1)(p+1)}{p-1}}\Big)^{\frac{p-1}{p+1}}
\Big(\int_{\Omega_c}u_{2,k}^{p+1}\Big)^{\frac{1}{p+1}}\\
&\leq C\Big(\int_{\Omega_c}\frac{\delta_k^{((n-2)(p+1)-n)\frac{p_k-1}{p-1}}}{|x-e_n|^{\frac{(n-2)(p+1)(p_k-1)}{p-1}}}
+\frac{\delta_k^{((n-2)(p+1)-n)\frac{p_k-1}{p-1}}}{|x+e_n|^{\frac{(n-2)(p+1)(p_k-1)}{p-1}}}\Big)^{\frac{p-1}{p+1}}=o(1).
\end{split}\end{align}
Similarly,
\begin{align}\label{f2}
\begin{split}
&\|f'_{1,\epsilon_k}(PW_{1,\delta_k})u_{1,k}\|_{L^{\frac{q+1}q}(\Omega_c)}
\leq C\Big(\int_{\Omega_c}U_{e_n,\delta_k}^{\frac{(q_k-1)(q+1)}{q-1}}+U_{-e_n,\delta_k}^{\frac{(q_k-1)(q+1)}{q-1}}\Big)^{\frac{q-1}{q+1}}
\Big(\int_{\Omega_c}u_{1,k}^{q+1}\Big)^{\frac{1}{q+1}}\\
&\leq C\begin{cases}\dis\Big(\int_{\Omega_c}\frac{\delta_k^{((n-2)(q+1)-n)\frac{q_k-1}{q-1}}}{|x-e_n|^{\frac{(n-2)(q+1)(q_k-1)}{q-1}}}
+\frac{\delta_k^{((n-2)(q+1)-n)\frac{q_k-1}{q-1}}}{|x+e_n|^{\frac{(n-2)(q+1)(q_k-1)}{q-1}}}\Big)^{\frac{q-1}{q+1}}  &if\ p>\frac n{n-2},\\
\dis \Big(\int_{\Omega_c}\frac{\delta_k^{(((n-2)p-2)(q+1)-n)\frac{q_k-1}{q-1}}}{|x-e_n|^{\frac{((n-2)p-2)(q+1)(q_k-1)}{q-1}}}
+\frac{\delta_k^{(((n-2)p-2)(q+1)-n)\frac{q_k-1}{q-1}}}{|x+e_n|^{\frac{((n-2)p-2)(q+1)(q_k-1)}{q-1}}}\Big)^{\frac{q-1}{q+1}}  &if\ p<\frac n{n-2}
\end{cases}\\
&\leq C\begin{cases}\dis O\Big(\delta_k^{\frac{q_k-1}{p+1}n}\Big)  &if\ p>\frac n{n-2},\\
\dis O\Big(\delta_k^{\frac{q_k-1}{q+1}pn}\Big)  &if\ p<\frac n{n-2}
\end{cases}=o(1),
\end{split}\end{align}
where we use the fact that $(q+1)(n-2)>n$ and $((n-2)p-2)(q+1)=(p+1)n>n$.

Now on $\Omega_+$, by blowing up, we estimate
\begin{align}\label{f3}
\begin{split}
&\|f'_{2,\epsilon_k}(PW_{2,\delta_k})u_{2,k}\|_{L^{\frac{p+1}p}(\Omega_+)}
=\|V_{e_n,\delta_k}^{p_k-1}u_{2,k}\|_{L^{\frac{p+1}p}(\Omega_+)}+o(1)\\
&\leq C\Big(\int_{\Omega_+}\frac{\delta_k^{-\frac{n(p_k-1)}{p}}}{(1+\frac{|x-e_n|}{\delta_k})^{\frac{(n-2)(p+1)(p_k-1)}{p}}}u_{2,k}^{\frac{p+1}p}\Big)^{\frac{p}{p+1}}+o(1)\\
&\leq C\delta_k^{n(1-\frac{p_k}{p})\frac{p}{p+1}}\Big(\int_{\Omega_k}\frac{1}{(1+|y|)^{\frac{(n-2)(p+1)(p_k-1)}{p}}}\hat u_{2,k}^{\frac{p+1}p}\Big)^{\frac{p}{p+1}}+o(1)\\
&\leq C\Big(\int_{\R^n}\frac{1}{(1+|y|)^{\frac{(n-2)(p+1)(p-1)}{p}}}\hat u_{2,k}^{\frac{p+1}p}\Big)^{\frac{p}{p+1}}+o(1)\\
&=o(1)
\end{split}\end{align}
where in the last step, we use the previous result $\hat u_{2,k}\rightharpoonup0$ weakly in $L^{p+1}(\R^n)$, which means
$\hat u_{2,k}^{\frac{p+1}p}\rightharpoonup0$ weakly in $L^{p}(\R^n)$, and $\frac{1}{(1+|y|)^{\frac{(n-2)(p+1)(p-1)}{p}}}\in L^{\frac{p}{p-1}}(\R^n)$.

Next, for $p>\frac n{n-2}$,
\begin{align}\label{f4}
\begin{split}
&\|f'_{1,\epsilon_k}(PW_{1,\delta_k})u_{1,k}\|_{L^{\frac{q+1}q}(\Omega_+)}
=\|U_{e_n,\delta_k}^{q_\epsilon-1}u_{1,k}\|_{L^{\frac{q+1}q}(\Omega_+)}+o(1)\\
&\leq C\Big(\int_{\Omega_+}\frac{\delta_k^{-\frac{n(q_k-1)}{q}}}{(1+\frac{|x-e_n|}{\delta_k})^{\frac{(n-2)(q+1)(q_k-1)}{q}}}u_{1,k}^{\frac{q+1}q}\Big)^{\frac{q}{q+1}}+o(1)\\
&\leq C\delta_k^{n(1-\frac{q_k}{q})\frac{q}{q+1}}\Big(\int_{\Omega_k}\frac{1}{(1+|y|)^{\frac{(n-2)(q+1)(q_k-1)}{q}}}\hat u_{1,k}^{\frac{q+1}q}\Big)^{\frac{q}{q+1}}+o(1)\\
&\leq C\Big(\int_{\R^n}\frac{1}{(1+|y|)^{\frac{(n-2)(q+1)(q-1)}{q}}}\hat u_{1,k}^{\frac{q+1}q}\Big)^{\frac{q}{q+1}}+o(1)
=o(1);
\end{split}\end{align}
While for $p<\frac n{n-2}$,
\begin{align}\label{f4'}
\begin{split}
&\|f'_{1,\epsilon_k}(PW_{1,\delta_k})u_{1,k}\|_{L^{\frac{q+1}q}(\Omega_+)}
=\|U_{e_n,\delta_k}^{q_\epsilon-1}u_{1,k}\|_{L^{\frac{q+1}q}(\Omega_+)}+o(1)\\
&\leq C\Big(\int_{\Omega_+}\frac{\delta_k^{-\frac{n(q_k-1)}{q}}}{(1+\frac{|x-e_n|}{\delta_k})^{\frac{((n-2)p-2)(q+1)(q_k-1)}{q}}}u_{1,k}^{\frac{q+1}q}\Big)^{\frac{q}{q+1}}+o(1)\\
&\leq C\delta_k^{n(1-\frac{q_k}{q})\frac{q}{q+1}}\Big(\int_{\Omega_k}\frac{1}{(1+|y|)^{\frac{((n-2)p-2)(q+1)(q_k-1)}{q}}}\hat u_{1,k}^{\frac{q+1}q}\Big)^{\frac{q}{q+1}}+o(1)\\
&\leq C\Big(\int_{\R^n}\frac{1}{(1+|y|)^{\frac{((n-2)p-2)(q+1)(q-1)}{q}}}\hat u_{1,k}^{\frac{q+1}q}\Big)^{\frac{q}{q+1}}+o(1)
=o(1),
\end{split}\end{align}
where we used the fact that $\hat u_{1,k}\rightharpoonup0$ weakly in $L^{q+1}(\R^n)$, which means
$\hat u_{1,k}^{\frac{q+1}q}\rightharpoonup0$ weakly in $L^{q}(\R^n)$, and $\dis\frac{1}{(1+|y|)^{\frac{(n-2)(q+1)(q-1)}{q}}}
,\frac{1}{(1+|y|)^{\frac{((n-2)p-2)(q+1)(q-1)}{q}}}\in L^{\frac{q}{q-1}}(\R^n)$.

Also on $\Omega_-$,
\begin{align}\label{f5}
&\|f'_{1,\epsilon_k}(PW_{1,\delta_k})u_{1,k}\|_{L^{\frac{q+1}q}(\Omega_-)}+\|f'_{2,\epsilon_k}(PW_{2,\delta_k})u_{2,k}\|_{L^{\frac{p+1}p}(\Omega_-)}
=o(1).
\end{align}

From \eqref{f1}-\eqref{f5}, we finally obtain
\begin{align*}
\|f'_{2,\epsilon_k}(PW_{2,\delta_k})u_{2,k}\|_{L^{\frac{p+1}p}}+\|f'_{1,\epsilon_k}(PW_{1,\delta_k})u_{1,k}\|_{L^{\frac{q+1}q}}=o(1),
\end{align*}
which is a contradiction with \eqref{>0}, concluding the proof.
\end{proof}

\medskip

\subsection{Estimate for $R_{d,\epsilon}$ and $N_{d,\epsilon}$}
In order to prove Proposition \ref{prop3.1} through applying the contraction mapping principle,
we  estimate the zero order term $R_{d,\epsilon}$ and the higher order terms $N_{d,\epsilon}(\phi_1,\phi_2)$.

\begin{Lem}\label{lemR}
Let $\eta\in(0,1)$, $\delta=d\epsilon$. Then $\|R_{d,\epsilon}\|_\epsilon=O(\epsilon^{\frac12+\gamma'})$
as $\epsilon\rightarrow0$ uniformly in $d\in(\eta,\frac1\eta)$ for some $\gamma'\in(0,\frac12)$.

\end{Lem}

\begin{proof}
Rewrite
\begin{align*}
 R_{d,\epsilon}&=\Pi_\epsilon^\perp\Big(K(f_{2,\epsilon}(PW_{2,\delta}),f_{1,\epsilon}(PW_{1,\delta}))-(PW_{1,\delta},PW_{2,\delta})\Big)\\
&=\Pi_\epsilon^\perp\circ K(f_{2,\epsilon}(PW_{2,\delta})-f_{2,0}(PW_{2,\delta}),f_{1,\epsilon}(PW_{1,\delta})-f_{1,0}(PW_{1,\delta}))\\
&\quad+\Pi_\epsilon^\perp\Big(K(f_{2,0}(PW_{2,\delta}),f_{1,0}(PW_{1,\delta}))-(PW_{1,\delta},PW_{2,\delta})\Big).
\end{align*}
Recall the norm $\|\cdot\|_\epsilon=\|\cdot\|+\|\cdot\|_{L^{q_\epsilon+1}\times L^{p_\epsilon+1}}$ in \eqref{norm}. We perform the proof in four steps.
\medskip

{\bf Step 1.}  We first deal with the term
\begin{align}
\Big\|\Pi_\epsilon^\perp\Big(K(f_{2,0}(PW_{2,\delta}),f_{1,0}(PW_{1,\delta}))-(PW_{1,\delta},PW_{2,\delta})\Big)\Big\|.
\end{align}

Write
\begin{align}\label{vw}
\begin{split}
&(v_{1,\delta},v_{2,\delta})=K(f_{2,0}(PW_{2,\delta}),f_{1,0}(PW_{1,\delta})),\\
&(w_{1,\delta},w_{2,\delta})=K(f_{2,0}(V_{e_n,\delta})-f_{2,0}(V_{-e_n,\delta}),f_{1,0}(U_{e_n,\delta})-f_{1,0}(U_{-e_n,\delta}))
\end{split}\end{align}
that is
\begin{align*}
\begin{cases}
-\Delta v_{1,\delta}=f_{2,0}(PV_{e_n,\delta}-PV_{-e_n,\delta})\\
-\Delta v_{2,\delta}=f_{1,0}(PU_{e_n,\delta}-PU_{-e_n,\delta}),\\
 \int_\Omega v_{i,\delta}=0,\ \partial_\nu v_{i,\delta}=0\ on\ \partial\Omega,\ i=1,2
 \end{cases}
\end{align*}
and
\begin{align*}
\begin{cases}
-\Delta w_{1,\delta}=f_{2,0}(V_{e_n,\delta})-f_{2,0}(V_{-e_n,\delta})\\
-\Delta w_{2,\delta}=f_{1,0}(U_{e_n,\delta})-f_{1,0}(U_{-e_n,\delta}),\\
 \int_\Omega w_{i,\delta}=0,\ \partial_\nu w_{i,\delta}=0\ on\ \partial\Omega,\ i=1,2.
 \end{cases}
\end{align*}

In the case of $p>\frac n{n-2}$,
recalling the definition of the norm $\|\cdot\|$, from \eqref{f2'} and Remark \ref{remp}, there exists some small $\gamma'>0$,
\begin{align*}
&\| v_{1,\delta}-w_{1,\delta}\|=\|\Delta v_{1,\delta}-\Delta w_{1,\delta}\|_{L^{\frac{p+1}p}}\\&\leq\|f_{2,0}(PV_{e_n,\delta}-PV_{-e_n,\delta})-f_{2,0}(V_{e_n,\delta})+f_{2,0}(V_{-e_n,\delta})\|_{L^{\frac{p+1}p}}\\
&\leq\|f_{2,0}(PV_{e_n,\delta}-PV_{-e_n,\delta})-f_{2,0}(V_{e_n,\delta}-V_{-e_n,\delta})\|_{L^{\frac{p+1}p}}\\&\quad+
\|f_{2,0}(V_{e_n,\delta}-V_{-e_n,\delta})-f_{2,0}(V_{e_n,\delta})+f_{2,0}(V_{-e_n,\delta})\|_{L^{\frac{p+1}p}}=O(\delta^{\frac12+\gamma'}),
\end{align*}
or
\begin{align*}
\begin{split}
&\|K(f_{2,0}(PW_{2,\delta}),f_{1,0}(PW_{1,\delta}))-(PW_{1,\delta},PW_{2,\delta})\|=O(\delta^{\frac12+\gamma'}).
\end{split}
\end{align*}

\medskip
While for $p<\frac n{n-2}$,
we use Lemma \ref{lemb11'} and Remark \ref{remP'} to get the same estimate that for some $\gamma'\in(0,\frac12)$
\begin{align}
\begin{split}
&\|K(f_{2,0}(PW_{2,\delta}),f_{1,0}(PW_{1,\delta}))-(PW_{1,\delta},PW_{2,\delta})\|
=O(\delta^{\frac12+\gamma'}).
\end{split}
\end{align}

\medskip

{\bf Step 2.}  We show that
\begin{align}\label{lem3.8}
\begin{split}
&\|K(f_{2,0}(PW_{2,\delta}),f_{1,0}(PW_{1,\delta}))-(PW_{1,\delta},PW_{2,\delta})\|_{L^{q_\epsilon+1}\times L^{p_\epsilon+1}}=O(\delta^{\frac12+\gamma'}).
\end{split}
\end{align}

Actually, using the notation \eqref{vw} and setting $$u_{i,\delta}=v_{i,\delta}-w_{i,\delta}(i=1,2),$$
we have
\begin{align*}
\begin{cases}
-\Delta u_{1,\delta}=f_{2,0}(PV_{e_n,\delta}-PV_{-e_n,\delta})-f_{2,0}(V_{e_n,\delta})+f_{2,0}(V_{-e_n,\delta})\\
-\Delta u_{2,\delta}=f_{1,0}(PU_{e_n,\delta}-PU_{-e_n,\delta})-f_{1,0}(U_{e_n,\delta})+f_{1,0}(U_{-e_n,\delta}),\\
 \int_\Omega u_{i,\delta}=0,\ \partial_\nu u_{i,\delta}=0\ on\ \partial\Omega,\ i=1,2
 \end{cases}
\end{align*}
or
\begin{align*}
(u_{1,\delta},u_{2,\delta})=K(f_{2,0}(PW_{2,\delta})-f_{2,0}(V_{e_n,\delta})+f_{2,0}(V_{-e_n,\delta}),
f_{1,0}(PW_{1,\delta})-f_{1,0}(U_{e_n,\delta})+f_{1,0}(U_{-e_n,\delta})).
\end{align*}

(i) The case of $p>\frac n{n-2}$.

By Lemma \ref{lemb11}, for any $\gamma>0$ small, there exist $\sigma>0$ and some $\bar\epsilon>0$ small such that
$\frac1{\frac1{q_{\bar\epsilon}+1}+\frac2n}=\frac{(p+1)(1+\sigma)}{p}$.
Then for $n\geq5$, for any $\epsilon\in(0,\bar\epsilon)$,
\begin{align}\label{barepsilon}
\begin{split}
&\|u_{1,\delta}\|_{L^{q_\epsilon+1}}\leq\|u_{1,\delta}\|_{L^{q_{\bar\epsilon}+1}}|\Omega|^{\frac1{q_\epsilon+1}-\frac1{q_{\bar\epsilon}+1}}\\
&\leq C_\epsilon\|f_{2,0}(PW_{2,\delta})-f_{2,0}(V_{e_n,\delta})+f_{2,0}(V_{-e_n,\delta})\|_{L^{\frac1{\frac1{q_{\bar\epsilon}+1}+\frac2n}}}\\
&\leq C_\epsilon\|f_{2,0}(PW_{2,\delta})-f_{2,0}(W_{2,\delta})\|_{L^{\frac1{\frac1{q_{\bar\epsilon}+1}+\frac2n}}}
+C_\epsilon\|f_{2,0}(W_{2,\delta})-f_{2,0}(V_{e_n,\delta})+f_{2,0}(V_{-e_n,\delta})\|_{L^{\frac1{\frac1{q_{\bar\epsilon}+1}+\frac2n}}}\\
&=C_\epsilon\|f_{2,0}(PW_{2,\delta})-f_{2,0}(W_{2,\delta})\|_{L^{\frac{(p+1)(1+\sigma)}{p}}}
+C_\epsilon\|f_{2,0}(W_{2,\delta})-f_{2,0}(V_{e_n,\delta})+f_{2,0}(V_{-e_n,\delta})\|_{L^{\frac{(p+1)(1+\sigma)}{p}}}\\
&=O(\delta^{\frac12+\gamma'}).
\end{split}
\end{align}
Also, for $n=4$ and for any $\epsilon\in(0,\bar\epsilon)$,
$
 \|u_{1,\delta}\|_{L^{q_\epsilon+1}}
=O(\delta^{\frac12+\gamma'}).
$
\medskip

(ii) The case of $p<\frac n{n-2}$.

In this case we just use Lemma \ref{lemb11'} and Remark \ref{rem'} to find that for $\gamma'>0$ small,
\begin{align}\label{barepsilon'}
\begin{split}
&\|u_{1,\delta}\|_{L^{q_\epsilon+1}}\leq\|u_{1,\delta}\|_{L^{q_{\bar\epsilon}+1}}|\Omega|^{\frac1{q_\epsilon+1}-\frac1{q_{\bar\epsilon}+1}}\\
&\leq C_\epsilon \|f_{2,0}(PW_{2,\delta})-f_{2,0}(W_{2,\delta})\|_{L^{\frac{(p+1)(1+\sigma)}{p}}}
+C_\epsilon\|f_{2,0}(W_{2,\delta})-f_{2,0}(V_{e_n,\delta})+f_{2,0}(V_{-e_n,\delta})\|_{L^{\frac{(p+1)(1+\sigma)}{p}}}\\
&=O(\delta^{\frac12+\gamma'}).
\end{split}
\end{align}
Similarly,  we can prove
$
 \|u_{2,\delta}\|_{L^{p_\epsilon+1}}=
O(\delta^{\frac12+\gamma'}),
$
and so we have also shown \eqref{lem3.8} in this case (ii).

\medskip
\medskip

{\bf Step 3.}  For the term $\Pi_\epsilon^\perp\circ K(f_{2,\epsilon}(PW_{2,\delta})-f_{2,0}(PW_{2,\delta}),f_{1,\epsilon}(PW_{1,\delta})-f_{1,0}(PW_{1,\delta}))$,
we apply Lemma \ref{lemb12} to obtain that for any $\gamma\in(0,1)$,
\begin{align*}
\begin{split}
&\|K(f_{2,\epsilon}(PW_{2,\delta})-f_{2,0}(PW_{2,\delta}),f_{1,\epsilon}(PW_{1,\delta})-f_{1,0}(PW_{1,\delta}))\|\\
&\leq \|(f_{2,\epsilon}(PW_{2,\delta})-f_{2,0}(PW_{2,\delta}),f_{1,\epsilon}(PW_{1,\delta})-f_{1,0}(PW_{1,\delta}))\|_{L^{\frac{p+1}p}\times L^{\frac{q+1}q}}=O(\epsilon^{1-\gamma})
\end{split}\end{align*}
and following the argument in \eqref{barepsilon}, for small $\bar\epsilon,\sigma_1,\sigma_2>0$, there holds that
\begin{align*}
\begin{split}
&\|K(f_{2,\epsilon}(PW_{2,\delta})-f_{2,0}(PW_{2,\delta}),f_{1,\epsilon}(PW_{1,\delta})-f_{1,0}(PW_{1,\delta}))\|_{L^{q_\epsilon+1}\times L^{p_\epsilon+1}}\\
&\leq C_\epsilon\|(f_{2,\epsilon}(PW_{2,\delta})-f_{2,0}(PW_{2,\delta}),f_{1,\epsilon}(PW_{1,\delta})-f_{1,0}(PW_{1,\delta}))\|_{L^{\frac1{\frac1{q_{\bar\epsilon}+1}+\frac2n}}\times L^{\frac1{\frac1{p_{\bar\epsilon}+1}+\frac2n}}}\\
&\leq C_\epsilon\|(f_{2,\epsilon}(PW_{2,\delta})-f_{2,0}(PW_{2,\delta}),f_{1,\epsilon}(PW_{1,\delta})-f_{1,0}(PW_{1,\delta}))\|_{L^{\frac{(p+1)(1+\sigma_1)}p}\times L^{\frac{(q+1)(1+\sigma_2)}q}}\\&
=O(\epsilon^{1-\gamma}).
\end{split}\end{align*}

\medskip

Combining the above three steps, we have the conclusion.

\end{proof}

\medskip
As for the higher order terms, we have
\begin{Lem}
Let $\eta\in(0,1),d>0,\delta=d\epsilon$. Then as $\epsilon\rightarrow0$, the following estimates hold uniformly in $d\in(\eta,\frac1\eta)$:
\begin{align*}
\|N_{d,\epsilon}(\phi_1,\phi_2)\|_\epsilon\leq C
\begin{cases}
\|\phi_1\|_{L^{q_\epsilon+1}}^{q_\epsilon}+\|\phi_2\|_{L^{p_\epsilon+1}}^{p_\epsilon},\ &1<p,q<2\\
\|\phi_1\|_{L^{q_\epsilon+1}}^{2}+\|\phi_2\|_{L^{p_\epsilon+1}}^{p_\epsilon},\ &1<p<2\leq q\\
\|\phi_1\|_{L^{q_\epsilon+1}}^{2}+\|\phi_2\|_{L^{p_\epsilon+1}}^{2},\ &p,q\geq2.
\end{cases}
\end{align*}

\end{Lem}

\begin{proof}
Note that
\begin{align}\label{Nepsilon}
\begin{split}
 \|N_{d,\epsilon}(\phi_1,\phi_2)\|_\epsilon
&=\|\Pi_\epsilon^\perp\circ K\Big(f_{2,\epsilon}(PW_{2,\delta}+\phi_2)-f'_{2,\epsilon}(PW_{2,\delta})\phi_2-f_{2,\epsilon}(PW_{2,\delta}),\\
&\qquad +f_{1,\epsilon}(PW_{1,\delta}+\phi_1)-f'_{1,\epsilon}(PW_{1,\delta})\phi_1-f_{1,\epsilon}(PW_{1,\delta})\Big)\|_\epsilon\\
&\leq C\|f_{2,\epsilon}(PW_{2,\delta}+\phi_2)-f'_{2,\epsilon}(PW_{2,\delta})\phi_2-f_{2,\epsilon}(PW_{2,\delta})\|_{L^{\frac{p+1}p}}\\&\quad
+C\|f_{1,\epsilon}(PW_{1,\delta}+\phi_1)-f'_{1,\epsilon}(PW_{1,\delta})\phi_1-f_{1,\epsilon}(PW_{1,\delta})\|_{L^{\frac{q+1}q}}\\
&\quad+C\|K\Big(f_{2,\epsilon}(PW_{2,\delta}+\phi_2)-f'_{2,\epsilon}(PW_{2,\delta})\phi_2-f_{2,\epsilon}(PW_{2,\delta}),\\
&\qquad +f_{1,\epsilon}(PW_{1,\delta}+\phi_1)-f'_{1,\epsilon}(PW_{1,\delta})\phi_1-f_{1,\epsilon}(PW_{1,\delta})\Big)\|_{L^{q_\epsilon+1}\times L^{p_\epsilon+1}}\\
&\leq C\|f_{2,\epsilon}(PW_{2,\delta}+\phi_2)-f'_{2,\epsilon}(PW_{2,\delta})\phi_2-f_{2,\epsilon}(PW_{2,\delta})\|_{L^{\frac{p+1}p}}\\&\quad
+C\|f_{1,\epsilon}(PW_{1,\delta}+\phi_1)-f'_{1,\epsilon}(PW_{1,\delta})\phi_1-f_{1,\epsilon}(PW_{1,\delta})\|_{L^{\frac{q+1}q}}\\
&\quad+C\|f_{2,\epsilon}(PW_{2,\delta}+\phi_2)-f'_{2,\epsilon}(PW_{2,\delta})\phi_2-f_{2,\epsilon}(PW_{2,\delta})\|_{L^{\frac1{\frac1{q_\epsilon+1}+\frac2n}}}\\
&\quad +C\|f_{1,\epsilon}(PW_{1,\delta}+\phi_1)-f'_{1,\epsilon}(PW_{1,\delta})\phi_1-f_{1,\epsilon}(PW_{1,\delta})\|_{L^{\frac1{\frac1{p_\epsilon+1}+\frac2n}}}.
\end{split}\end{align}

Since
\begin{align*}
\begin{split}
&\Big|f_{2,\epsilon}(PW_{2,\delta}+\phi_2)-f'_{2,\epsilon}(PW_{2,\delta})\phi_2-f_{2,\epsilon}(PW_{2,\delta})\Big|\\&
\leq C\begin{cases}
|PW_{2,\delta}|^{p_\epsilon-2}|\phi_2|^2+|\phi_2|^{p_\epsilon}, \ &p\geq2\\
|\phi_2|^{p_\epsilon}, \ &1<p<2,
\end{cases}
\end{split}\end{align*}
then from Lemma \ref{lemb3}, when $1<p<2$,
\begin{align}
\begin{split}
&\|f_{2,\epsilon}(PW_{2,\delta}+\phi_2)-f'_{2,\epsilon}(PW_{2,\delta})\phi_2-f_{2,\epsilon}(PW_{2,\delta})\|_{L^{\frac{p+1}p}}\\
&\quad+\|f_{2,\epsilon}(PW_{2,\delta}+\phi_2)-f'_{2,\epsilon}(PW_{2,\delta})\phi_2-f_{2,\epsilon}(PW_{2,\delta})\|_{L^{\frac1{\frac1{q_\epsilon+1}+\frac2n}}}\\
&\leq C\|\phi_2\|_{L^{\frac{(p+1)p_\epsilon}{p}}}^{p_\epsilon}+ C\|\phi_2\|_{L^{\frac{p_\epsilon}{\frac1{q_\epsilon+1}+\frac2n}}}^{p_\epsilon}
\leq C\|\phi_2\|_{L^{p_\epsilon+1}}^{p_\epsilon};
\end{split}\end{align}
While for $p\geq2$, it sufficed to further estimate that
\begin{align}
\begin{split}
&\||PW_{2,\delta}|^{p_\epsilon-2}|\phi_2|^2\|_{L^{\frac{p+1}{p}}}+\||PW_{2,\delta}|^{p_\epsilon-2}|\phi_2|^2\|_{L^{\frac1{\frac1{q_\epsilon+1}+\frac2n}}}
\leq C\|\phi_2\|_{L^{p_\epsilon+1}}^{2},
\end{split}\end{align}
where we have used the  H\"older inequalities and the fact that $\frac1{\frac1{q_\epsilon+1}+\frac2n}=\frac{p+1}p+O(\epsilon),\frac{(p+1)p_\epsilon}{p}=\frac{p+1}p+O(\epsilon).$

The estimate for
\begin{align*}
\begin{split}
&\|f_{1,\epsilon}(PW_{1,\delta}+\phi_1)-f'_{1,\epsilon}(PW_{1,\delta})\phi_1-f_{1,\epsilon}(PW_{1,\delta})\|_{L^{\frac{q+1}q}}\\
&\quad +C\|f_{1,\epsilon}(PW_{1,\delta}+\phi_1)-f'_{1,\epsilon}(PW_{1,\delta})\phi_1-f_{1,\epsilon}(PW_{1,\delta})\|_{L^{\frac1{\frac1{p_\epsilon+1}+\frac2n}}}
\end{split}\end{align*}
can be considered similarly.

\end{proof}

\medskip

\subsection{Proof of Proposition \ref{prop3.1}}

Recall that \eqref{Pibot} can be written as
\begin{align*}
L_{d,\epsilon}(\phi_1,\phi_2)=N_{d,\epsilon}(\phi_1,\phi_2)+R_{d,\epsilon},
\end{align*} which, by Proposition \ref{propL},
 is equivalent to the fixed point problem
 \begin{align*}
(\phi_1,\phi_2)=L_{d,\epsilon}^{-1}\Big(N_{d,\epsilon}(\phi_1,\phi_2)+R_{d,\epsilon}\Big):=\L_{d,\epsilon}(\phi_1,\phi_2)
\end{align*}
with
 \begin{align*}
\|\L_{d,\epsilon}(\phi_1,\phi_2)\|_\epsilon\leq C\|N_{d,\epsilon}(\phi_1,\phi_2)\|_\epsilon+C\|R_{d,\epsilon}\|_\epsilon.
\end{align*}

\begin{proof}[
\textbf{Proof of Proposition \ref{prop3.1}}]
Take $C_0>0$ be a constant such that
$\|R_{d,\epsilon}\|_\epsilon\leq C_0\epsilon^{\frac12+\gamma'}$.
 Set
$$\mathcal B:=\{(\phi_1,\phi_2)\in E_{d,\epsilon}:\|(\phi_1,\phi_2)\|_\epsilon\leq 2C_0\epsilon^{\frac12+\gamma'}\}.$$
Then for sufficiently small $\epsilon>0$, $\L_{d,\epsilon}(\mathcal B)\subset\mathcal B$.
Moreover, it is not difficult to prove that there is some $\theta\in(0,1)$ such that
$$\|\L_{d,\epsilon}(\phi_1-\psi_1,\phi_2-\psi_2)\|_\epsilon\leq \theta\|(\phi_1-\psi_1,\phi_2-\psi_2)\|_\epsilon,\ \ for\ every\ (\phi_1,\phi_2),(\psi_1,\psi_2)\in\mathcal B.$$
Hence, for any given $\eta\in(0,1),d\in(\eta,\frac1\eta),\delta=d\epsilon$, applying the Banach Fixed Point Theorem, there exists a unique fixed point $(\phi_{1,d,\epsilon},\phi_{2,d,\epsilon})$ of $\L_{d,\epsilon}$.

Finally, by use of the implicit function theorem, we argue standardly that the map $d\mapsto(\phi_{1,d,\epsilon},\phi_{2,d,\epsilon})$ is of class $C^1$.
One could refer to \cite{mp,pt} for example for more details.

\end{proof}

\medskip

\section{Reduced Problem}

Set the functional $$I_\epsilon(u_1,u_2)=\int_\Omega\nabla u_1\cdot\nabla u_2-\frac1{p_\epsilon+1}\int_\Omega|u_2|^{p_\epsilon+1}-\frac1{q_\epsilon+1}\int_\Omega|u_1|^{q_\epsilon+1}.
$$
For each $\eta$ small, let $\epsilon_0>0$ be as in Proposition \ref{prop3.1}. Then for $\epsilon\in(0,\epsilon_0), \delta=d\epsilon$, consider the reduced functional
$J_\epsilon:(\eta,\frac1\eta)\rightarrow\R$ which is given by
$$J_\epsilon(d):=I_\epsilon(PW_{1,\delta}+\phi_{1,d,\epsilon},PW_{2,\delta}+\phi_{2,d,\epsilon})$$
with $(\phi_{1,d,\epsilon},\phi_{2,d,\epsilon})\in E_{d,\epsilon}$ as in Proposition \ref{prop3.1}.

It is standard to show (see for instance \cite{mp}) that $$J'_\epsilon(d)=0\Longleftrightarrow I'_\epsilon(PW_{1,\delta}+\phi_{1,d,\epsilon},PW_{2,\delta}+\phi_{2,d,\epsilon})=0.$$
We are aimed to check that $J_\epsilon$ has a minimizer in $(\eta,\frac1\eta)$ for sufficiently small $\eta$, which is a critical point.
For this purpose, we expand $J_\epsilon(d)$ as $\epsilon\rightarrow0$.
For convenience, we introduce the following quantities:
\begin{align}\label{notation}
\begin{split}
&\mathcal A_1=\int_{\R^n}U^{q+1},\ \  \mathcal A_2=\int_{\R^n}V^{p+1},\\
&\mathcal B_1=\int_\Xi\int_{\sqrt{1-|x|^2}}^1\delta^{-n-1}U^{q+1}\Big(\frac{|x-e_n|}{\delta}\Big)dx_ndx',\\&
\mathcal B_2=\int_\Xi\int_{\sqrt{1-|x|^2}}^1\delta^{-n-1}V^{p+1}\Big(\frac{|x-e_n|}{\delta}\Big)dx_ndx',\\
&\mathcal C_1=-\int_{\partial\R^{n}_+}|x'|U'(x')V(x')dx'>0,\ \ \mathcal C_2=-\int_{\partial\R^{n}_+}|x'|V'(x')U(x')dx'>0,\\
&\mathcal D_1=\int_{\R^n} U^{q+1}\log U,\ \ \mathcal D_2=\int_{\R^n} V^{p+1}\log V
\end{split}
\end{align}
with $\Xi=\{x'\in\R^{n-1}: |x'|<\frac{\sqrt{15}}8\}$ the projection of $\partial\Omega\cap B_{\frac12}(e_n)$ in the variables $x'=(x_1,\ldots,x_{n-1})$.
\begin{Rem}
By \eqref{eqUV}, $\mathcal A_1=\mathcal A_2$. Moreover, from the proof of Lemma \ref{lemc1}, we can obtain that $\mathcal B_1=O(1),\mathcal B_2=O(1)$.
\end{Rem}

\medskip
\begin{Prop}\label{propJd}
For any $\eta\in(0,1)$ small, as $\epsilon\rightarrow0$, it holds uniformly in $d\in(\eta,\frac1\eta)$ that
\begin{align}\label{Jd}
\begin{split}
 &J_\epsilon(d)=\int_\Omega\nabla PW_{1,\delta}\cdot\nabla PW_{2,\delta}-\frac1{p_\epsilon+1}\int_\Omega|PW_{2,\delta}|^{p_\epsilon+1}-\frac1{q_\epsilon+1}\int_\Omega|PW_{1,\delta}|^{q_\epsilon+1}\\
  =&\frac2n\mathcal A_1+\Big(\frac {n\alpha}{(p+1)^2}+\frac {n\beta}{(q+1)^2}\Big)\mathcal A_1\epsilon\log\epsilon
 +\mathcal G(d)\epsilon+o(\epsilon)
\end{split}\end{align}
with
\begin{align}\label{G}
\begin{split}
\mathcal G(d)&= \frac1{p+1}(\frac{\alpha\mathcal A_1}{p+1}-\alpha\mathcal D_1)+\frac1{q+1}(\frac{\beta\mathcal A_2}{q+1}-\beta\mathcal D_2)
+(\frac {n\alpha}{(p+1)^2}+\frac {n\beta}{(q+1)^2})\mathcal A_1\log d\\&-\Big((1-\frac2{p+1})\mathcal B_2+(1-\frac2{q+1})\mathcal B_1+\frac{\mathcal C_1+\mathcal C_2}2\Big)d.
\end{split}\end{align}
\end{Prop}
\medskip

To show Proposition \ref{propJd}, we use several lemmas as follows.
\begin{Lem}
For $\eta\in(0,1)$ small, as $\epsilon\rightarrow0$, it holds uniformly in $d\in(\eta,\frac1\eta)$ that
\begin{align*}
J_\epsilon(d)=I_\epsilon(PW_{1,\delta},PW_{2,\delta})+o(\epsilon).
\end{align*}
\end{Lem}

\begin{proof}
\begin{align*}
&I_\epsilon(PW_{1,\delta}+\phi_{1,d,\epsilon},PW_{2,\delta}+\phi_{2,d,\epsilon})-I_\epsilon(PW_{1,\delta},PW_{2,\delta})\\
&=\int_\Omega\nabla\phi_{1,d,\epsilon}\cdot\nabla\phi_{2,d,\epsilon}+\int_\Omega(U_{e_n,\delta}^q-U_{-e_n,\delta}^q)\phi_{1,d,\epsilon}
+(V_{e_n,\delta}^p-V_{-e_n,\delta}^p)\phi_{2,d,\epsilon}\\
&\quad-\frac1{q_\epsilon+1}\int_\Omega\Big(|PW_{1,\delta}+\phi_{1,d,\epsilon}|^{q_\epsilon+1}-|PW_{1,\delta}|^{q_\epsilon+1}\Big)\\&\quad
-\frac1{p_\epsilon+1}\int_\Omega\Big(|PW_{2,\delta}+\phi_{2,d,\epsilon}|^{p_\epsilon+1}-|PW_{2,\delta}|^{p_\epsilon+1}\Big)\\
&=o(\epsilon)+\int_\Omega(U_{e_n,\delta}^q-U_{-e_n,\delta}^q-|PW_{1,\delta}|^{q_\epsilon-1}PW_{1,\delta})\phi_{1,d,\epsilon}\\&\quad
+\int_\Omega(V_{e_n,\delta}^p-V_{-e_n,\delta}^p-|PW_{2,\delta}|^{p_\epsilon-1}PW_{2,\delta})\phi_{2,d,\epsilon}\\
&\quad-\frac1{q_\epsilon+1}\int_\Omega\Big(|PW_{1,\delta}+\phi_{1,d,\epsilon}|^{q_\epsilon+1}-|PW_{1,\delta}|^{q_\epsilon+1}-|PW_{1,\delta}|^{q_\epsilon-1}PW_{1,\delta}\phi_{1,d,\epsilon}\Big)\\&\quad
-\frac1{p_\epsilon+1}\int_\Omega\Big(|PW_{2,\delta}+\phi_{2,d,\epsilon}|^{p_\epsilon+1}-|PW_{2,\delta}|^{p_\epsilon+1}-|PW_{2,\delta}|^{p_\epsilon-1}PW_{2,\delta}\phi_{2,d,\epsilon}\Big)\\
&:=o(\epsilon)+I_1+I_2+J_1+J_2.
\end{align*}
By Proposition \ref{prop3.1}, Lemma \ref{lemb11}, Lemma \ref{lemb11'}  and Lemma \ref{lemb12} with $\sigma=0$, we get that
\begin{align*}
 |I_1|+|I_2|&\leq\|U_{e_n,\delta}^q-U_{-e_n,\delta}^q-|PW_{1,\delta}|^{q_\epsilon-1}PW_{1,\delta}\|_{L^{\frac{q+1}q}}\|\phi_{1,d,\epsilon}\|_{L^{q+1}}\\&\quad
+\|V_{e_n,\delta}^p-V_{-e_n,\delta}^p-|PW_{2,\delta}|^{p_\epsilon-1}PW_{2,\delta}\|\|_{L^{\frac{p+1}p}}\|\phi_{2,d,\epsilon}\|_{L^{p+1}}=o(\epsilon)
\end{align*}
and there exists $\theta=\theta(d,\epsilon)\in(0,1)$,
\begin{align*}
 |J_1|+|J_2|&\leq C\int_\Omega\Big(|PW_{1,\delta}+\theta\phi_{1,d,\epsilon}|^{q_\epsilon-1}\phi_{1,d,\epsilon}^2+|PW_{2,\delta}+\theta\phi_{2,d,\epsilon}|^{p_\epsilon-1}\phi_{2,d,\epsilon}^2\Big)\\
&\leq C\Big(\||PW_{1,\delta}|^{q_\epsilon-1}\|_{L^{\frac{q+1}{q-1}}}\|\phi_{1,d,\epsilon}\|_{L^{q+1}}^2+\|\phi_{1,d,\epsilon}\|_{L^{q_\epsilon+1}}^{q_\epsilon+1}\\&\qquad+
\||PW_{2,\delta}|^{p_\epsilon-1}\|_{L^{\frac{p+1}{p-1}}}\|\phi_{2,d,\epsilon}\|_{L^{p+1}}^2+\|\phi_{2,d,\epsilon}\|_{L^{p_\epsilon+1}}^{p_\epsilon+1}
\Big)=o(\epsilon).
\end{align*}

\end{proof}

\medskip

\subsection{Expansion of the leading term}
We are sufficed to expand the leading term $I_\epsilon(PW_{1,\delta},PW_{2,\delta})$.

Recall from Proposition \ref{lemexpansion} that
\begin{align*}
&PW_{1,\delta}(x)=W_{1,\delta}(x)-\delta^{-\frac n{q+1}+1}\Big(\varphi_{1,0}(\frac{e_n-x}{\delta})-\varphi_{1,0}(\frac{e_n+x}{\delta})\Big)+\zeta_{1,\delta}(x),\\
&PW_{2,\delta}(x)=W_{2,\delta}(x)-\delta^{-\frac n{p+1}+1}\Big(\varphi_{2,0}(\frac{e_n-x}{\delta})-\varphi_{2,0}(\frac{e_n+x}{\delta})\Big)+\zeta_{2,\delta}(x),
\end{align*}
where $\zeta_{1,\delta}=O(\delta^{2-\frac n{q+1}})$, $\zeta_{2,\delta}=O(\delta^{2-\frac n{p+1}})$,
$\partial_\delta\zeta_{1,\delta}=O(\delta^{1-\frac n{q+1}})$, $\partial_\delta\zeta_{2,\delta}=O(\delta^{1-\frac n{p+1}})$ as $\delta\rightarrow0$ uniformly in $\Omega$.
We have
\begin{align*}
&\int_\Omega\nabla  PW_{1,\delta}\cdot \nabla PW_{2,\delta}
=\int_\Omega PW_{1,\delta}(-\Delta)PW_{2,\delta}
=\int_\Omega PW_{1,\delta}\Big(U_{e_n,\delta}^q-U_{-e_n,\delta}^q\Big)\\
&=\int_\Omega\Big(W_{1,\delta}-\delta^{-\frac n{q+1}+1}\Big(\varphi_{1,0}(\frac{e_n-x}{\delta})-\varphi_{1,0}(\frac{e_n+x}{\delta})\Big)+\zeta_{1,\delta}(x)\Big)\Big(U_{e_n,\delta}^q-U_{-e_n,\delta}^q\Big)\\
&:=I_1-I_2+I_3,
\end{align*}
with
\begin{align*}
&I_1=\int_\Omega W_{1,\delta}\Big(U_{e_n,\delta}^q-U_{-e_n,\delta}^q\Big),\ \  I_3=\int_\Omega\zeta_{1,\delta}(x)\Big(U_{e_n,\delta}^q-U_{-e_n,\delta}^q\Big)\\
&I_2=\int_\Omega\delta^{-\frac n{q+1}+1}\Big(\varphi_{1,0}(\frac{e_n-x}{\delta})-\varphi_{1,0}(\frac{e_n+x}{\delta})\Big)\Big(U_{e_n,\delta}^q-U_{-e_n,\delta}^q\Big).
\end{align*}

\begin{Prop}\label{propnabla}
It holds that
\begin{align*}
&\int_\Omega\nabla  PW_{1,\delta}\cdot \nabla PW_{2,\delta}=\mathcal A_1-2\mathcal B_1\delta+\mathcal C_1\delta+o(\delta)\\
&=\mathcal A_2-2\mathcal B_2\delta+\mathcal C_2\delta+o(\delta)=\mathcal A_1-(\mathcal B_1+\mathcal B_2)\delta+\frac{\mathcal C_1+\mathcal C_2}{2}\delta+o(\delta).
\end{align*}
\end{Prop}

\begin{proof}
By Lemma \ref{lemc1} and Lemma \ref{lemb6}, we first estimate $I_1$:
\begin{align*}
 I_1&=\int_\Omega W_{1,\delta}\Big(U_{e_n,\delta}^q-U_{-e_n,\delta}^q\Big)\\
&=\int_\Omega U_{e_n,\delta}^{q+1}- U_{e_n,\delta}U_{-e_n,\delta}^q-U_{-e_n,\delta}U_{e_n,\delta}^q+ U_{-e_n,\delta}^{q+1}\\
&=\mathcal A_1-2\mathcal B_1\delta+o(\delta).
\end{align*}

We apply Lemma \ref{lemb8} and Lemma \ref{lemb7} to obtain that
\begin{align*}
 I_2&=-\mathcal C_1\delta+o(\delta).
\end{align*}

Finally for $I_3$, we have for $p>\frac n{n-2}$, 
\begin{align*}
&\int_{\Omega_+}|\zeta_{1,\delta}(x)|U_{e_n,\delta}^q\leq C\delta^{2-\frac n{q+1}}\int_{\Omega_+}\frac{\delta^{-\frac{qn}{q+1}}}{(1+\frac{|x-e_n|}\delta)^{q(n-2)}}\\
&\leq  C\delta^{2-\frac n{q+1}}\int_{\frac{\Omega_+-e_n}{\delta}}\frac{\delta^{n-\frac{qn}{q+1}}}{(1+|x|)^{q(n-2)}}=O(\delta^2)=o(\delta).
\end{align*}
On the other hand,
\begin{align*}
&\int_{\Omega\setminus\Omega_+}|\zeta_{1,\delta}(x)|U_{e_n,\delta}^q\leq C\delta^{2-\frac n{q+1}}\int_{\Omega\setminus\Omega_+}\frac{\delta^{-\frac{qn}{q+1}}}{(1+\frac{|x-e_n|}\delta)^{q(n-2)}}=O(\delta^{(q-1)(n-2)})=o(\delta).
\end{align*}

While for $p<\frac n{n-2}$, since $q((n-2)p-2)>n$,
\begin{align*}
&\int_{\Omega_+}|\zeta_{1,\delta}(x)|U_{e_n,\delta}^q\leq C\delta^{2-\frac n{q+1}}\int_{\Omega_+}\frac{\delta^{-\frac{qn}{q+1}}}{(1+\frac{|x-e_n|}\delta)^{q((n-2)p-2)}}\\
&\leq  C\delta^{2-\frac n{q+1}}\int_{\frac{\Omega_+-e_n}{\delta}}\frac{\delta^{n-\frac{qn}{q+1}}}{(1+|x|)^{q((n-2)p-2)}}=O(\delta^2)=o(\delta).
\end{align*}
On the other hand,
\begin{align*}
&\int_{\Omega\setminus\Omega_+}|\zeta_{1,\delta}(x)|U_{e_n,\delta}^q\leq C\delta^{2-\frac n{q+1}}\int_{\Omega\setminus\Omega_+}\frac{\delta^{-\frac{qn}{q+1}}}{(1+\frac{|x-e_n|}\delta)^{q((n-2)p-2)}}=o(\delta).
\end{align*}

By symmetry, we finally find
\begin{align*}
&I_3=\int_\Omega\zeta_{1,\delta}(x)\Big(U_{e_n,\delta}^q-U_{-e_n,\delta}^q\Big)=o(\delta).
\end{align*}

\end{proof}

\medskip

\begin{Prop}\label{propnonlinear}
It holds that
\begin{align}\label{1p}
\begin{split}
 &\frac1{p_\epsilon+1}\int_\Omega|PW_{2,\delta}|^{p_\epsilon+1}\\=&\frac1{p+1}\Big(\mathcal A_2-2\mathcal B_2\delta+(p+1)\mathcal C_2\delta\Big)+\Big(-\frac {n\alpha\epsilon}{(p+1)^2}\log\delta\mathcal A_2+\frac{\alpha\epsilon}{p+1}\mathcal D_2-\frac{\alpha\epsilon}{(p+1)^2}\mathcal A_2\Big)+o(\epsilon)\end{split}\end{align}
 and\begin{align}\label{1q}
\begin{split}
 &\frac1{q_\epsilon+1}\int_\Omega|PW_{1,\delta}|^{q_\epsilon+1}\\
 =&\frac1{q+1}\Big(\mathcal A_1-2\mathcal B_1\delta+(q+1)\mathcal C_1\delta\Big)+\Big(-\frac {n\alpha\epsilon}{(q+1)^2}\log\delta\mathcal A_1+\frac{\beta\epsilon}{q+1}\mathcal D_1-\frac{\beta\epsilon}{(q+1)^2}\mathcal A_1\Big)+o(\epsilon).
\end{split}\end{align}
\end{Prop}

\begin{proof}
We write
\begin{align*}
\int_\Omega|PW_{2,\delta}|^{p_\epsilon+1}=\int_\Omega|PW_{2,\delta}|^{p+1}+|PW_{2,\delta}|^{p+1}(|PW_{2,\delta}|^{\alpha\epsilon}-1):=J_1+J_2.
\end{align*}

First of all,  in view of the decay of $\varphi_{2,\delta}$ and $\zeta_{2,\delta}$ in Proposition \ref{lemexpansion}, similar arguments as in the proof of Proposition \ref{propnabla}, we have
\begin{align*}
&\int_{\Omega_+}|PW_{2,\delta}|^{p+1}=\int_{\Omega_+}|W_{2,\delta}-(W_{2,\delta}-PW_{2,\delta})|^{p+1}\\
&=\int_{\frac{\Omega_+-e_n}\delta}\Big|V(y)-\delta\Big(\varphi_{2,0}(-y)-\varphi_{2,0}(\frac{2e_n}{\delta}+y)\Big)-\zeta_{2,\delta}(\delta y+e_n)\delta^{\frac n{p+1}}\Big|^{p+1}
+o(\delta)\\
&=\int_{\frac{\Omega_+-e_n}\delta}\Big(|V(y)|^{p+1}-\delta(p+1)V(y)^p\varphi_{2,\delta}(-y)\Big)
+o(\delta)\\
&=\frac{\mathcal A_2}2-\mathcal B_2\delta+\frac{p+1}2\mathcal C_2\delta+o(\delta).
\end{align*}
Also we can obtain
\begin{align*}
&\int_{\Omega_-}|PW_{2,\delta}|^{p+1}=\frac{\mathcal A_2}2-\mathcal B_2\delta+\frac{p+1}2\mathcal C_2\delta+o(\delta).
\end{align*}
Hence it holds that
\begin{align}\label{J1}
&J_1=\mathcal A_2-2\mathcal B_2\delta+(p+1)\mathcal C_2\delta+o(\delta).
\end{align}

Secondly,  note that
\begin{align*}
J_2=\int_\Omega|PW_{2,\delta}|^{p+1}(|PW_{2,\delta}|^{\alpha\epsilon}-1)=\alpha\epsilon\int_\Omega|PW_{2,\delta}|^{p+1}\log|PW_{2,\delta}|+o(\delta).
\end{align*}
Following Lemma \ref{lemc1}, we can obtain that
\begin{align*}
&\int_{\frac{\Omega_+-e_n}{\delta}}V^{p+1}=\frac {\mathcal A_2}{2}-\mathcal B_2\delta+o(\delta),\\&
\int_{\frac{\Omega_+-e_n}{\delta}}V^{p+1}\log V=\frac12\int_{\R^n}V^{p+1}\log V+o(1)=\frac {\mathcal D_2}{2}+o(1),
\end{align*}
and then
\begin{align*}
&\alpha\epsilon\int_{\Omega_+}|PW_{2,\delta}|^{p+1}\log|PW_{2,\delta}|\\
&=\int_{\Omega_+}|W_{2,\delta}-(W_{2,\delta}-PW_{2,\delta})|^{p+1}\log|W_{2,\delta}-(W_{2,\delta}-PW_{2,\delta})|\\
&=\alpha\epsilon\int_{\Omega_+}V_{e_n,\delta}^{p+1}\log V_{e_n,\delta}+o(\epsilon)
=\alpha\epsilon\int_{\frac{\Omega_+-e_n}{\delta}}V^{p+1}\log(\delta^{-\frac n{p+1}} V)+o(\epsilon)\\
&=\alpha\epsilon\int_{\frac{\Omega_+-e_n}{\delta}}V^{p+1}\big(-\frac n{p+1}\log\delta +\log V)\big)+o(\epsilon)\\
&=-\alpha\epsilon\frac n{p+1}\log\delta\int_{\frac{\Omega_+-e_n}{\delta}}V^{p+1}+\alpha\epsilon\int_{\frac{\Omega_+-e_n}{\delta}}V^{p+1}\log V+o(\epsilon)\\
&=-\frac {n\alpha\epsilon}{2(p+1)}\log\delta\mathcal A_2+\frac{\alpha\epsilon}2\mathcal D_2+o(\delta).
\end{align*}
Similarly,
\begin{align*}
&\alpha\epsilon\int_{\Omega_-}|PW_{2,\delta}|^{p+1}\log|PW_{2,\delta}|=-\frac {n\alpha\epsilon}{2(p+1)}\log\delta\mathcal A_2+\frac{\alpha\epsilon}2\mathcal D_2+o(\delta).
\end{align*}
Hence
\begin{align}\label{J2}
J_2=-\frac {n\alpha\epsilon}{(p+1)}\log\delta\mathcal A_2+\alpha\epsilon\mathcal D_2+o(\epsilon).
\end{align}

Moreover, since $\frac1{p_\epsilon+1}=\frac1{p+1}-\frac{\alpha\epsilon}{(p+1)^2}+o(\epsilon)$, \eqref{J1} and \eqref{J2} imply that
\begin{align*}
 &\frac1{p_\epsilon+1}\int_\Omega|PW_{2,\delta}|^{p_\epsilon+1}\\
 =&\Big(\frac1{p+1}-\frac{\alpha\epsilon}{(p+1)^2}+o(\epsilon)\Big)\Big(\mathcal A_2-2\mathcal B_2\delta+(p+1)\mathcal C_2\delta-\frac {n\alpha\epsilon}{(p+1)}\log\delta\mathcal A_2+\alpha\epsilon\mathcal D_2+o(\epsilon)\Big)\\
 =&\frac1{p+1}\Big(\mathcal A_2-2\mathcal B_2\delta+(p+1)\mathcal C_2\delta\Big)+\Big(-\frac {n\alpha\epsilon}{(p+1)^2}\log\delta\mathcal A_2+\frac{\alpha\epsilon}{p+1}\mathcal D_2-\frac{\alpha\epsilon}{(p+1)^2}\mathcal A_2\Big)+o(\epsilon).
\end{align*}

Similar argument gives \eqref{1q}.

\end{proof}

\medskip

From Proposition \ref{propnabla}, Proposition \ref{propnonlinear}, \eqref{1q} and $\mathcal A_1=\mathcal A_2$, we get for the total energy that
\begin{align}\label{Jd}
\begin{split}
 &J_\epsilon(d)=\int_\Omega\nabla PW_{1,\delta}\cdot\nabla PW_{2,\delta}-\frac1{p_\epsilon+1}\int_\Omega|PW_{2,\delta}|^{p_\epsilon+1}-\frac1{q_\epsilon+1}\int_\Omega|PW_{1,\delta}|^{q_\epsilon+1}\\
 =&(1-\frac1{q+1}-\frac1{p+1})\mathcal A_1+\Big(\frac {n\alpha}{(p+1)^2}\epsilon\log\epsilon+\frac {n\beta}{(q+1)^2}\epsilon\log\epsilon\Big)\mathcal A_1
 +\mathcal G(d)\epsilon+o(\epsilon)\\
  =&\frac2n\mathcal A_1+\Big(\frac {n\alpha}{(p+1)^2}+\frac {n\beta}{(q+1)^2}\Big)\mathcal A_1\epsilon\log\epsilon
 +\mathcal G(d)\epsilon+o(\epsilon)
\end{split}\end{align}
with $\mathcal G$ defined as \eqref{G}.

\medskip

\subsection{Existence of the critical point}

\begin{Lem}
The function $\mathcal G(d)$ has a unique critical point $d^*$, which ia a global maximum.

\end{Lem}
\begin{proof}
From the definition of $\mathcal G(d)$, it is obvious that $\mathcal G(d)\rightarrow-\infty$ as $d\rightarrow0^+$ and as $d\rightarrow+\infty$.
Hence $\mathcal G(d)$ achieves a global maximum  $$d^*:=\frac{(\frac {n\alpha}{(p+1)^2}+\frac {n\beta}{(q+1)^2})\mathcal A_1}{\Big((1-\frac2{p+1})\mathcal B_2+(1-\frac2{q+1})\mathcal B_1+\frac{\mathcal C_1+\mathcal C_2}2}$$ solving
\begin{align*}
\mathcal G'(d)=
\frac nd(\frac {\alpha}{(p+1)^2}+\frac {\beta}{(q+1)^2})\mathcal A_1-\Big((1-\frac2{p+1})\mathcal B_2+(1-\frac2{q+1})\mathcal B_1+\frac{\mathcal C_1+\mathcal C_2}2\Big)=0,
\end{align*}
which
is the unique critical point of $\mathcal G(d)$.
\end{proof}

\medskip

\begin{proof}[
\textbf{Proof of Theorem \ref{th1}}]
Let $\eta\in(0,1)$ such that $d^*\in(\eta,\frac1\eta)$. Then $\mathcal G(\eta),\mathcal G(\frac1\eta)<\mathcal G(d^*)$ since $d^*$ is the global maximum of $\mathcal G$.
Moreover, given such $\eta$, from \eqref{Jd}, we find that as $\epsilon\rightarrow0$
\begin{align*}
\begin{split}
 &J_\epsilon(d)=\frac2n\mathcal A_1+\Big(\frac {n\alpha}{(p+1)^2}+\frac {n\beta}{(q+1)^2}\Big)\mathcal A_1\epsilon\log\epsilon
 +\mathcal G(d)\epsilon+o(\epsilon)
\end{split}\end{align*}
uniformly in $(\eta,\frac1\eta)$. Therefore, there exists $\epsilon_0>0$ small such that for any $\epsilon\in(0,\epsilon_0)$
$J_\epsilon(\eta),J_\epsilon(\frac1\eta)<J_\epsilon(d^*)$, which means that $J_\epsilon$ has an interior maximum in $(\eta,\frac1\eta)$, a critical point of $J_\epsilon$.

Hence we have found a critical point $(PW_{1,\delta}+\phi_{1,d,\epsilon},PW_{2,\delta}+\phi_{2,d,\epsilon})$ of $I_\epsilon$ such that $I'_\epsilon(PW_{1,\delta}+\phi_{1,d,\epsilon},PW_{2,\delta}+\phi_{2,d,\epsilon})=0$, concluding Theorem \ref{th1}.
\end{proof}

\medskip
\section*{Appendix}

\appendix

\section{Basic Estimates for Bubbles}
\renewcommand{\theequation}{A.\arabic{equation}}

Firstly, after making some minor modifications to the proof of Lemma B.3 in \cite{arxiv}, we can prove by direct calculation the following Lemma \ref{lemb3} and Lemma \ref{lemb2}.
\begin{Lem}\label{lemb3}
For $\xi\in\R^n,\delta>0$, if we set
\begin{align*}
u_{1,\delta,\xi}=\frac{\delta^{-\frac n{q+1}}}{(1+\frac{|x-\xi|^2}{\delta^2})^{\frac{n-2}2}},\ \ \ \tilde u_{1,\delta,\xi}=\frac{\delta^{-\frac n{q+1}}}{(1+\frac{|x-\xi|^2}{\delta^2})^{\frac{(n-2)p-2}2}},\ \ \
u_{2,\delta,\xi}=\frac{\delta^{-\frac n{p+1}}}{(1+\frac{|x-\xi|^2}{\delta^2})^{\frac{n-2}2}},\\
v_{1,\delta,\xi}=\frac{\delta^{1-\frac n{q+1}}}{(1+\frac{|x-\xi|^2}{\delta^2})^{\frac{n-3}2}},\ \ \
\tilde v_{1,\delta,\xi}=\frac{\delta^{1-\frac n{q+1}}}{(1+\frac{|x-\xi|^2}{\delta^2})^{\frac{(n-2)p-3}2}},\ \ \
v_{2,\delta,\xi}=\frac{\delta^{1-\frac n{p+1}}}{(1+\frac{|x-\xi|^2}{\delta^2})^{\frac{n-3}2}},
\end{align*}
then for $R>0$, as $\delta\rightarrow0$, there hold that
\begin{align*}
&\int_{B_R(\xi)}u_{1,\delta,\xi}^t=
\begin{cases}O(\delta^{\frac{tn}{p+1}}),\ &0<t<\frac n{n-2}\\
O(\delta^{n(1-\frac n{(n-2)(q+1)})}|ln\delta|),\ &t=\frac n{n-2}\\
O(\delta^{n-\frac{nt}{q+1}}),\ &t>\frac n{n-2}
\end{cases}\\
&\int_{B_R(\xi)}\tilde u_{1,\delta,\xi}^t=
\begin{cases}O(\delta^{\frac{tpn}{q+1}}),\ &0<t<\frac n{(n-2)p-2}\\
O(\delta^{n(1-\frac n{((n-2)p-2)(q+1)})}|ln\delta|),\ &t=\frac n{(n-2)p-2}\\
O(\delta^{n-\frac{nt}{q+1}}),\ &t>\frac n{(n-2)p-2}
\end{cases}\\
&\int_{B_R(\xi)}u_{2,\delta,\xi}^t=
\begin{cases}O(\delta^{\frac{tn}{q+1}}),\ &0<t<\frac n{n-2}\\
O(\delta^{n(1-\frac n{(n-2)(p+1)})}|ln\delta|),\ &t=\frac n{n-2}\\
O(\delta^{n-\frac{nt}{p+1}}),\ &t>\frac n{n-2}
\end{cases}\\
&\int_{B_R(\xi)}v_{1,\delta,\xi}^t=
\begin{cases}O(\delta^{\frac{tn}{p+1}}),\ &0<t<\frac n{n-3}\\
O(\delta^{\frac {n^2}{(n-3)(p+1)}}|ln\delta|),\ &t=\frac n{n-3}\\
O(\delta^{n-t(\frac{n}{q+1}-1)}),\ &t>\frac n{n-3}
\end{cases}\\
&\int_{B_R(\xi)}\tilde v_{1,\delta,\xi}^t=
\begin{cases}O(\delta^{\frac{tpn}{q+1}}),\ &0<t<\frac n{(n-2)p-3}\\
O(\delta^{\frac {n^2}{((n-2)p-3)(p+1)}}|ln\delta|),\ &t=\frac n{(n-2)p-3}\\
O(\delta^{n-t(\frac{n}{q+1}-1)}),\ &t>\frac n{(n-2)p-3}
\end{cases}\\
&\int_{B_R(\xi)}v_{2,\delta,\xi}^t=\begin{cases}O(\delta^{\frac{tn}{q+1}}),\ &0<t<\frac n{n-3}\\
O(\delta^{\frac {n^2}{(n-3)(q+1)}}|ln\delta|),\ &t=\frac n{n-3}\\
O(\delta^{n-t(\frac{n}{p+1}-1)}),\ &t>\frac n{n-3}.
\end{cases}
\end{align*}

\end{Lem}

Recall that $$f_{2,\epsilon}(t)=|t|^{p+\alpha\epsilon-1}t,\ \ \ f_{1,\epsilon}(t)=|t|^{q+\beta\epsilon-1}t.$$
\begin{Lem}\label{lemb2}
For $t\in\R$, it holds that
\begin{align*}
f_{1,\epsilon}(t)&=|t|^{q-1}t+\beta\epsilon|t|^{q-1}t\log|t|+\epsilon^2\xi_{1,\epsilon}(t),\\
f'_{1,\epsilon}(t)&=q|t|^{q-1}+\beta\epsilon(|t|^{q-1}+q|t|^{q-1}\log|t|)+\epsilon^2\eta_{1,\epsilon}(t)
\end{align*}
and
\begin{align*}
f_{2,\epsilon}(t)&=|t|^{p-1}t+\alpha\epsilon|t|^{p-1}t\log|t|+\epsilon^2\xi_{2,\epsilon}(t),\\
f'_{1,\epsilon}(t)&=p|t|^{p-1}+\beta\epsilon(|t|^{p-1}+p|t|^{p-1}\log|t|)+\epsilon^2\eta_{2,\epsilon}(t),
\end{align*}
where
\begin{align*}
&|\xi_{1,\epsilon}(t)|\leq\frac12(|t|^q+|t|^{q+\beta\epsilon})(\log|t|)^2,\\
&|\eta_{1,\epsilon}(t)|\leq2(q+1)(|t|^{q-1}+|t|^{q-1+\beta\epsilon})(\log|t|+(\log|t|)^2)\\
&|\xi_{2,\epsilon}(t)|\leq\frac12(|t|^p+|t|^{p+\beta\epsilon})(\log|t|)^2,\\
&|\eta_{2,\epsilon}(t)|\leq2(p+1)(|t|^{p-1}+|t|^{p-1+\beta\epsilon})(\log|t|+(\log|t|)^2).
\end{align*}

\end{Lem}

\medskip
Next, we need other tools when carrying out the reduction method.
\begin{Lem}\label{lemb11}
For $\sigma\geq0$ small, let $\eta\in(0,1)$, $d\in(\eta,\frac1\eta)$ and $\delta=d\epsilon$. Then for $p>\frac n{n-2}$, for every $\gamma\in(0,1)$, it holds that
\begin{align}\label{b111}
\begin{split}
&\|f_{1,0}(PW_{1,\delta})-f_{1,0}(W_{1,\delta})\|_{L^{\frac{(q+1)(1+\sigma)}{q}}}=O(\delta^{1-\gamma}),\\
&\|f_{2,0}(PW_{2,\delta})-f_{2,0}(W_{2,\delta})\|_{L^{\frac{(p+1)(1+\sigma)}{p}}}=\begin{cases}O(\delta^{1-\gamma}), &n\geq5\ or\ n=4\ and\ p+1\geq4\\
O(\delta^{\frac4{q+1}-\gamma}),&n=4, p+1<4.
\end{cases}
\end{split}
\end{align}
Moreover,
\begin{align}\label{b11f}
\begin{split}
&\|f_{2,0}(V_{e_n,\delta}-V_{-e_n,\delta})-f_{2,0}(V_{e_n,\delta})+f_{2,0}(V_{-e_n,\delta})\|_{L^{\frac{(p+1)(1+\sigma)}p}}\\
&=\begin{cases}
O(\delta^{1-\gamma}),\ &p-\frac1p\geq\frac n{n-2},\sigma\geq0\\
O\Big(\delta^{\frac{(p-1)}{q+1}n}\Big),\ &p-\frac1p<\frac n{n-2},\sigma\geq0,
\end{cases}\\
&\|f_{1,0}(U_{e_n,\delta}-U_{-e_n,\delta})-f_{1,0}(U_{e_n,\delta})+f_{1,0}(U_{-e_n,\delta})\|_{L^{\frac{(q+1)(1+\sigma)}q}}=
O(\delta^{1-\gamma}).
\end{split}
\end{align}

\end{Lem}
\begin{Rem}\label{remp}
Firstly, we see that $p\geq\frac{n+\sqrt2}{n-2}$ ensures $\frac{(p-1)}{q+1}n>1$.
Next, if $p\in(\frac n{n-2},\frac{n+\sqrt2}{n-2})$, it may hold that $\frac{(p-1)}{q+1}n<1$. But $p>\frac n{n-2}$
implies $\frac{(p-1)}{q+1}n>\frac12$ and $\frac4{q+1}>\frac12$.
Hence, for some small $\gamma'\in(0,\frac12)$, we actually have
\begin{align}\label{f2'}
\begin{split}
&\|f_{2,0}(PW_{2,\delta})-f_{2,0}(W_{2,\delta})\|_{L^{\frac{(p+1)(1+\sigma)}{p}}}+\|f_{2,0}(V_{e_n,\delta}-V_{-e_n,\delta})-f_{2,0}(V_{e_n,\delta})+f_{2,0}(V_{-e_n,\delta})\|_{L^{\frac{(p+1)(1+\sigma)}p}}\\
&=O(\delta^{\frac12+\gamma'}).
\end{split}
\end{align}
\end{Rem}

\begin{proof}[\bf Proof of Lemma \ref{lemb11}]
For $\sigma\geq0$, by H\"older inequalities, we have
\begin{align*}
\begin{split}
&\int_\Omega|f_{1,0}(PW_{1,\delta})-f_{1,0}(W_{1,\delta})|^{\frac{(q+1)(1+\sigma)}q}
=\int_\Omega\Big||PW_{1,\delta}|^{q-1} PW_{1,\delta}-|W_{1,\delta}|^{q-1}W_{1,\delta}\Big|^{\frac{(q+1)(1+\sigma)}q}\\
&\leq C\int_\Omega \Big( |W_{1,\delta}|^{q-1}|PW_{1,\delta}-W_{1,\delta}|+|PW_{1,\delta}-W_{1,\delta}|^q\Big)^{\frac{(q+1)(1+\sigma)}q}\\
&\leq C\int_\Omega |W_{1,\delta}|^{\frac{(q-1)(q+1)(1+\sigma)}q}|PW_{1,\delta}-W_{1,\delta}|^{\frac{(q+1)(1+\sigma)}q}
+|PW_{1,\delta}-W_{1,\delta}|^{(q+1)(1+\sigma)}\\
&\leq C\Big(\int_\Omega |W_{1,\delta}|^{(q+1)(1+\sigma)}\Big)^{\frac{q-1}q}
\Big(\int_\Omega|PW_{1,\delta}-W_{1,\delta}|^{(q+1)(1+\sigma)}\Big)^{\frac1q}
+\int_\Omega|PW_{1,\delta}-W_{1,\delta}|^{(q+1)(1+\sigma)}
\end{split}
\end{align*}
and similarly,
\begin{align}\label{b111}
\begin{split}
&\int_\Omega|f_{2,0}(PW_{2,\delta})-f_{2,0}(W_{2,\delta})|^{\frac{(p+1)(1+\sigma)}p}\\
&\leq C\Big(\int_\Omega |W_{2,\delta}|^{(p+1)(1+\sigma)}\Big)^{\frac{p-1}p}
\Big(\int_\Omega|PW_{2,\delta}-W_{2,\delta}|^{(p+1)(1+\sigma)}\Big)^{\frac1p}
+\int_\Omega|PW_{2,\delta}-W_{2,\delta}|^{(p+1)(1+\sigma)}.
\end{split}
\end{align}

Since $p>\frac n{n-2}$, when $n\geq5$, $p+1>\frac{2n-2}{n-2}>\frac n{n-3}$ and so for $\sigma\geq0$
\begin{align}\label{b112}
\begin{split}
&\Big(\int_\Omega|PW_{2,\delta}-W_{2,\delta}|^{(p+1)(1+\sigma)}\Big)^{\frac1p}
=
O(\delta^{\frac{(p+1)(1+\sigma)-n\sigma}p}).
\end{split}
\end{align}
For $n=4$ ($\frac n{n-3}=4$) and $\sigma=0$, we have
\begin{align}\label{b113}
\begin{split}
&\int_\Omega|PW_{2,\delta}-W_{2,\delta}|^{p+1}
=\begin{cases}
O(\delta^{\frac{4(p+1)}{q+1}}), &p+1<4\\
O(\delta^{4}|ln\delta|), &p+1=4\\
O(\delta^{4}), &p+1>4.
\end{cases}
\end{split}
\end{align}
While for $n=4$ and $\sigma>0$ small,
\begin{align}\label{b114}
\begin{split}
&\int_\Omega|PW_{2,\delta}-W_{2,\delta}|^{(p+1)(1+\sigma)}
=\begin{cases}
O(\delta^{\frac{4(p+1)(1+\sigma)}{q+1}}), &p+1<4\\
O(\delta^{(p+1)(1+\sigma)-\sigma n}), &p+1\geq4.
\end{cases}
\end{split}
\end{align}

On the other hand, since $p>\frac 2{n-2}$ implies $p+1>\frac n{n-2}$, then we have
\begin{align}\label{b115}
\begin{split}
&\int_\Omega |W_{2,\delta}|^{(p+1)(1+\sigma)}=
O(\delta^{-n\sigma}).
\end{split}
\end{align}

Combining \eqref{b112}-\eqref{b115}, we find that for $n\geq5$
\begin{align*}
\begin{split}
&\|f_{2,0}(PW_{2,\delta})-f_{2,0}(W_{2,\delta})\|_{L^{\frac{(p+1)(1+\sigma)}{p}}}=O(\delta^{1-\frac{pn\sigma}{(p+1)(1+\sigma)}});
\end{split}
\end{align*}
While for $n=4$
\begin{align*}
\begin{split}
&\|f_{2,0}(PW_{2,\delta})-f_{2,0}(W_{2,\delta})\|_{L^{\frac{(p+1)(1+\sigma)}{p}}}=\begin{cases}
O(\delta^{\frac4{q+1}-\frac{4\sigma}{(p+1)(1+\sigma)}}), &p+1<4\\
O(\delta^{1-\frac{4p\sigma}{(p+1)(1+\sigma)}}), &p+1\geq4,\ \sigma>0,\ or\ p+1>4\\
O(\delta^{1-\gamma}), &p+1=4,\ \sigma=0,\ n=4.
\end{cases}
\end{split}
\end{align*}

The case of
\begin{align*}
&\|f_{1,0}(PW_{1,\delta})-f_{1,0}(W_{1,\delta})\|_{L^{\frac{(q+1)(1+\sigma)}{q}}}=O(\delta^{1-\gamma})
\end{align*}
can be obtained by the same (more simpler) method.

\medskip

Now using \eqref{omegasplit}, we estimate that
\begin{align}\label{+}
\begin{split}
&\|f_{2,0}(V_{e_n,\delta}-V_{-e_n,\delta})-f_{2,0}(V_{e_n,\delta})+f_{2,0}(V_{-e_n,\delta})\|^{\frac{(p+1)(1+\sigma)}p}_{L^{\frac{(p+1)(1+\sigma)}p}(\Omega_+)}\\
&\leq\int_{\Omega_+}\Big||W_{2,\delta}|^{p-1}W_{2,\delta}-V_{e_n,\delta}^p+V_{-e_n,\delta}^p\Big|^{\frac{(p+1)(1+\sigma)}p}\\
&\leq C\int_{\Omega_+}\Big||W_{2,\delta}|^{p-1}W_{2,\delta}-V_{e_n,\delta}^p\Big|^{\frac{(p+1)(1+\sigma)}p}
+C\int_{\Omega_+}\Big|V_{-e_n,\delta}^p\Big|^{\frac{(p+1)(1+\sigma)}p}\\
&\leq C\int_{\Omega_+}\Big||V_{e_n,\delta}|^{p-1}V_{-e_n,\delta}+V_{-e_n,\delta}^p\Big|^{\frac{(p+1)(1+\sigma)}p}
+O\Big(\delta^{(pn-2(p+1))(1+\sigma)}\Big)\\
&\leq C\delta^{\frac{n}{q+1}\frac{p+1}{p}(1+\sigma)}\int_{\Omega_+}|V_{e_n,\delta}|^{\frac{(p-1)(p+1)(1+\sigma)}p}
+O\Big(\delta^{(pn-2(p+1))(1+\sigma)}\Big)\\
&=\begin{cases}
O\Big(\delta^{n-\frac{(p-1)(1+\sigma)}{p}n}\Big),\ &p-\frac1p>\frac n{n-2},\sigma\geq0\ \ or\ \ p-\frac1p\geq\frac n{n-2},\sigma>0\\
O\Big(\delta^{\frac{n}{p}}|\ln\delta|\Big),\ &p-\frac1p=\frac n{n-2},\sigma=0\\
O\Big(\delta^{\frac{(p-1)(p+1)(1+\sigma)}{p(q+1)}n}\Big),\ &p-\frac1p<\frac n{n-2},\sigma\geq0\ small.
\end{cases}
\end{split}
\end{align}
By exchanging the roles of $e_n$ and $-e_n$, we obtain a similar estimate for $$\|f_{2,0}(V_{e_n,\delta}-V_{-e_n,\delta})-f_{2,0}(V_{e_n,\delta})+f_{2,0}(V_{-e_n,\delta})\|_{L^{\frac{(p+1)(1+\sigma)}p}(\Omega_-)}.$$
On the other hand, note that on $\Omega\setminus(\Omega_+\cup\Omega_-)$, $|x-e_n|,|x+e_n|\geq\frac12$, so
\begin{align}\label{+-}
\begin{split}
& \int_{\Omega\setminus(\Omega_+\cup\Omega_-)}\Big||W_{2,\delta}|^{p-1}W_{2,\delta}-V_{e_n,\delta}^p-V_{-e_n,\delta}^p\Big|^{\frac{(p+1)(1+\sigma)}p}\\
&\leq C\int_{\Omega\setminus(\Omega_+\cup\Omega_-)}V_{e_n,\delta}^{(p+1)(1+\sigma)}
+C\int_{\Omega\setminus(\Omega_+\cup\Omega_-)} V_{-e_n,\delta}^{(p+1)(1+\sigma)}\\
&=O\Big(\delta^{(pn-2(p+1))(1+\sigma)}\Big).
\end{split}
\end{align}
Finally, from \eqref{+}-\eqref{+-}, we obtain \eqref{b112}.
\begin{align}\label{b11f2}
\begin{split}
&\|f_{2,0}(V_{e_n,\delta}-V_{-e_n,\delta})-f_{2,0}(V_{e_n,\delta})+f_{2,0}(V_{-e_n,\delta})\|_{L^{\frac{(p+1)(1+\sigma)}p}}\\
&=\begin{cases}
O\Big(\delta^{\frac{p}{(p-1)(1+\sigma)}n-n}\Big),\ &p-\frac1p>\frac n{n-2},\sigma\geq0\ \ or\ \ p-\frac1p\geq\frac n{n-2},\sigma>0\\
O\Big(\delta^{\frac{n}{p+1}}|\ln\delta|\Big),\ &p-\frac1p=\frac n{n-2},\sigma=0\\
O\Big(\delta^{\frac{(p-1)}{q+1}n}\Big),\ &p-\frac1p<\frac n{n-2},\sigma\geq0.
\end{cases}
\end{split}
\end{align}

Similarly,
\begin{align}\label{b11f1}
\begin{split}
&\|f_{1,0}(U_{e_n,\delta}-U_{-e_n,\delta})-f_{1,0}(U_{e_n,\delta})+f_{1,0}(U_{-e_n,\delta})\|_{L^{\frac{(q+1)(1+\sigma)}q}}\\
&=\begin{cases}
O\Big(\delta^{\frac{q}{(q-1)(1+\sigma)}n-n}\Big),\ &q-\frac1q>\frac n{n-2},\sigma\geq0\ \ or\ \ q-\frac1q\geq\frac n{n-2},\sigma>0\\
O\Big(\delta^{\frac{n}{q+1}}|\ln\delta|\Big),\ &q-\frac1q=\frac n{n-2},\sigma=0\\
O\Big(\delta^{\frac{(q-1)}{p+1}n}\Big),\ &q-\frac1q<\frac n{n-2},\sigma\geq0.
\end{cases}
\end{split}
\end{align}

Note that  $p\in(\frac n{n-2},\frac{n+2}{n-2})$ and $\frac1{p+1}+\frac1{q+1}=\frac{n-2}n$ imply that
\begin{align}\label{q}
 q\in\Big(\frac{n+2}{n-2},\frac{n^2+2n-4}{(n-2)^2}\Big)
 \end{align}
which implies that 
 $\frac n{q+1}>\frac{(n-2)^2}{2(n-1)}\geq1$  and $\frac{(q-1)}{p+1}n>1$ for $n\geq5$;
While in the case $n=4$ and $q-\frac1q=\frac n{n-2}=2$, we have $\frac n{q+1}=\frac4{2+\sqrt2}>1$.
Moreover, note that for $n\geq4$, $\frac n{p-1}>\frac n{p+1}>\frac{n-2}{2}\geq1$.

To sum up, we in fact have proved that, if $\sigma>0$ is sufficiently small, for any $\gamma>0$,
\begin{align*}
\begin{split}
&\|f_{2,0}(V_{e_n,\delta}-V_{-e_n,\delta})-f_{2,0}(V_{e_n,\delta})+f_{2,0}(V_{-e_n,\delta})\|_{L^{\frac{(p+1)(1+\sigma)}p}}\\
&=\begin{cases}
O\Big(\delta^{\frac{p}{(p-1)(1+\sigma)}n-n}\Big),\ &p-\frac1p>\frac n{n-2},\sigma\geq0\ \ or\ \ p-\frac1p\geq\frac n{n-2},\sigma>0\\
O\Big(\delta^{\frac{n}{p+1}}|\ln\delta|\Big),\ &p-\frac1p=\frac n{n-2},\sigma=0\\
O\Big(\delta^{\frac{(p-1)}{q+1}n}\Big),\ &p-\frac1p<\frac n{n-2},\sigma\geq0.
\end{cases}\\
&=\begin{cases}
O(\delta^{1-\gamma}),\ &p-\frac1p\geq\frac n{n-2},\sigma\geq0\\
O\Big(\delta^{\frac{(p-1)}{q+1}n}\Big),\ &p-\frac1p<\frac n{n-2},\sigma\geq0.
\end{cases}\\
&\|f_{1,0}(U_{e_n,\delta}-U_{-e_n,\delta})-f_{1,0}(U_{e_n,\delta})+f_{1,0}(U_{-e_n,\delta})\|_{L^{\frac{(q+1)(1+\sigma)}q}}=O(\delta^{1-\gamma}),
\end{split}
\end{align*}
concluding the proof of \eqref{b11f}.

\end{proof}

\medskip

\begin{Lem}\label{lemb11'}
For $\sigma\geq0$ small, let $\eta\in(0,1)$, $d\in(\eta,\frac1\eta)$ and $\delta=d\epsilon$. Then for $p<\frac n{n-2}$, it holds that
\begin{align}\label{b111'}
\begin{split}
&\|f_{1,0}(PW_{1,\delta})-f_{1,0}(W_{1,\delta})\|_{L^{\frac{(q+1)(1+\sigma)}{q}}}=\begin{cases}O(\delta^{1-\frac{qn\sigma}{(q+1)(1+\sigma)}})&if\ q+1\geq \frac n{p(n-2)-3},\\
O(\delta^{\frac{np}{q+1}}) &if\ q+1<\frac n{p(n-2)-3},
\end{cases}\\
&\|f_{2,0}(PW_{2,\delta})-f_{2,0}(W_{2,\delta})\|_{L^{\frac{(p+1)(1+\sigma)}{p}}}=\begin{cases}O(\delta^{1-\frac{pn\sigma}{(p+1)(1+\sigma)}})&if\ p+1\geq \frac n{n-3},\\
O(\delta^{\frac n{q+1}-\frac{n(p-1)\sigma}{(p+1)(1+\sigma)}}) &if\ p+1<\frac n{n-3}.
\end{cases}
\end{split}
\end{align}
Moreover,
\begin{align}\label{b11f'}
\begin{split}
&\|f_{2,0}(V_{e_n,\delta}-V_{-e_n,\delta})-f_{2,0}(V_{e_n,\delta})+f_{2,0}(V_{-e_n,\delta})\|_{L^{\frac{(p+1)(1+\sigma)}p}}
=O\Big(\delta^{\frac{(p-1)n}{q+1}}\Big),\\
&\|f_{1,0}(U_{e_n,\delta}-U_{-e_n,\delta})-f_{1,0}(U_{e_n,\delta})+f_{1,0}(U_{-e_n,\delta})\|_{L^{\frac{(q+1)(1+\sigma)}q}}\\&=
\begin{cases}
O\Big(\delta^{\frac{(p+1)n}{q+1}}\Big) \ &\frac{(p+1)(q-1)}q\geq1,\\
O\Big(\delta^{pn}\Big) \ &\frac{(p+1)(q-1)}q<1.
\end{cases}
\end{split}
\end{align}

\end{Lem}

\begin{Rem}\label{remP'}
The condition (ii) of {\bf (P)} ensures that
$\frac{p-1}{q+1}n>\frac12$ and is sufficient to show $p>\frac3{n-2}$ implying $\frac{p+1}{q+1}n>1$, which is needed in our proof.
\end{Rem}
\begin{proof}[\bf Proof of Lemma \ref{lemb11'}]

For $\sigma\geq0$,
\begin{align}\label{b112'}
\begin{split}
&\Big(\int_\Omega|PW_{2,\delta}-W_{2,\delta}|^{(p+1)(1+\sigma)}\Big)^{\frac1p}
=\begin{cases}
O(\delta^{\frac{(p+1)(1+\sigma)-n\sigma}p}) &if\ p+1\geq \frac n{n-3},\\
O(\delta^{\frac{n(p+1)(1+\sigma)}{p(q+1)}}) &if\ p+1<\frac n{n-3}.
\end{cases}
\end{split}
\end{align}

On the other hand, since $p>\frac 2{n-2}$ implies $p+1>\frac n{n-2}$, then we have
\begin{align}\label{b115'}
\begin{split}
&\int_\Omega |W_{2,\delta}|^{(p+1)(1+\sigma)}=
O(\delta^{-n\sigma}).
\end{split}
\end{align}

Combining \eqref{b112'}-\eqref{b115'}, we find that
\begin{align*}
\begin{split}
&\|f_{2,0}(PW_{2,\delta})-f_{2,0}(W_{2,\delta})\|_{L^{\frac{(p+1)(1+\sigma)}{p}}}=
\begin{cases}O(\delta^{1-\frac{pn\sigma}{(p+1)(1+\sigma)}})&if\ p+1\geq \frac n{n-3},\\
O(\delta^{\frac n{q+1}-\frac{n(p-1)\sigma}{(p+1)(1+\sigma)}}) &if\ p+1<\frac n{n-3}.
\end{cases}
\end{split}
\end{align*}

The case of
\begin{align*}
&\|f_{1,0}(PW_{1,\delta})-f_{1,0}(W_{1,\delta})\|_{L^{\frac{(q+1)(1+\sigma)}{q}}}
\end{align*}
can be obtained as follows.
\begin{align}\label{b112''}
\begin{split}
&\Big(\int_\Omega|PW_{1,\delta}-W_{1,\delta}|^{(q+1)(1+\sigma)}\Big)^{\frac1q}
=\begin{cases}
O(\delta^{\frac{(q+1)(1+\sigma)-n\sigma}q}) &if\ q+1\geq \frac n{p(n-2)-3},\\
O(\delta^{\frac{np(1+\sigma)}{q}}) &if\ q+1<\frac n{p(n-2)-3},
\end{cases}
\end{split}
\end{align}
\begin{align}\label{b115''}
\begin{split}
&\int_\Omega |W_{1,\delta}|^{(q+1)(1+\sigma)}=
O(\delta^{-n\sigma}),
\end{split}
\end{align}
\begin{align*}
\begin{split}
&\|f_{1,0}(PW_{1,\delta})-f_{1,0}(W_{1,\delta})\|_{L^{\frac{(q+1)(1+\sigma)}{q}}}=
\begin{cases}O(\delta^{1-\frac{qn\sigma}{(q+1)(1+\sigma)}})&if\ q+1\geq \frac n{p(n-2)-3},\\
O(\delta^{\frac{np}{q+1}}) &if\ q+1<\frac n{p(n-2)-3}.
\end{cases}
\end{split}
\end{align*}

\medskip

Now similar as \eqref{+}, since $p<\frac n{n-2}$,
\begin{align}\label{+'}
\begin{split}
&\|f_{2,0}(V_{e_n,\delta}-V_{-e_n,\delta})-f_{2,0}(V_{e_n,\delta})+f_{2,0}(V_{-e_n,\delta})\|_{L^{\frac{(p+1)(1+\sigma)}p}(\Omega_+)}=
O\Big(\delta^{\frac{(p-1)n}{q+1}}\Big).
\end{split}
\end{align}
By exchanging the roles of $e_n$ and $-e_n$, we obtain $$\|f_{2,0}(V_{e_n,\delta}-V_{-e_n,\delta})-f_{2,0}(V_{e_n,\delta})+f_{2,0}(V_{-e_n,\delta})\|_{L^{\frac{(p+1)(1+\sigma)}p}(\Omega_-)}
=O\Big(\delta^{\frac{(p-1)n}{q+1}}\Big).$$
On the other hand, note that on $\Omega\setminus(\Omega_+\cup\Omega_-)$, $|x-e_n|,|x+e_n|\geq\frac12$, so
\begin{align}\label{+-}
\begin{split}
&\int_{\Omega\setminus(\Omega_+\cup\Omega_-)}\Big||W_{2,\delta}|^{p-1}W_{2,\delta}-V_{e_n,\delta}^p-V_{-e_n,\delta}^p\Big|^{\frac{(p+1)(1+\sigma)}p}
=O\Big(\delta^{(pn-2(p+1))(1+\sigma)}\Big).
\end{split}
\end{align}
Finally, from \eqref{+}-\eqref{+-}, we obtain \eqref{b112'}.
\begin{align}\label{b11f2'}
\begin{split}
&\|f_{2,0}(V_{e_n,\delta}-V_{-e_n,\delta})-f_{2,0}(V_{e_n,\delta})+f_{2,0}(V_{-e_n,\delta})\|_{L^{\frac{(p+1)(1+\sigma)}p}}=O\Big(\delta^{\frac{(p-1)n}{q+1}}\Big).
\end{split}
\end{align}

Next,
\begin{align}\label{b11f1'}
\begin{split}
&\|f_{1,0}(U_{e_n,\delta}-U_{-e_n,\delta})-f_{1,0}(U_{e_n,\delta})+f_{1,0}(U_{-e_n,\delta})\|_{L^{\frac{(q+1)(1+\sigma)}q}(\Omega_\pm)}\\
&=\begin{cases}
O\Big(\delta^{\frac{(p+1)n}{q+1}}\Big) \ &\frac{(p+1)(q-1)}q\geq1,\\
O\Big(\delta^{pn}\Big) \ &\frac{(p+1)(q-1)}q<1.
\end{cases}
\end{split}
\end{align}
On the other hand,
\begin{align}\label{+-'}
\begin{split}
&\int_{\Omega\setminus(\Omega_+\cup\Omega_-)}\Big||W_{1,\delta}|^{q-1}W_{1,\delta}-U_{e_n,\delta}^q-U_{-e_n,\delta}^q\Big|^{\frac{(q+1)(1+\sigma)}q}
=O\Big(\delta^{pn(1+\sigma)}\Big).
\end{split}
\end{align}
To sum up, we in fact have proved that, if $\sigma>0$ is sufficiently small, for any $\gamma>0$,
\begin{align*}
\begin{split}
&\|f_{1,0}(U_{e_n,\delta}-U_{-e_n,\delta})-f_{1,0}(U_{e_n,\delta})+f_{1,0}(U_{-e_n,\delta})\|_{L^{\frac{(q+1)(1+\sigma)}q}}\\&=
\begin{cases}
O\Big(\delta^{\frac{(p+1)n}{q+1}}\Big) \ &\frac{(p+1)(q-1)}q\geq1,\\
O\Big(\delta^{pn}\Big) \ &\frac{(p+1)(q-1)}q<1.
\end{cases}
\end{split}
\end{align*}
concluding the proof of \eqref{b11f'}.

\end{proof}

\begin{Lem}\label{lemb12}
For $\sigma>0$ small, let $\eta\in(0,1)$, $d>0$ and $\delta=d\epsilon$. Then for every $\gamma\in(0,1)$, it holds that
\begin{align*}
&\|f_{1,\epsilon}(PW_{1,\delta})-f_{1,0}(PW_{1,\delta})\|_{L^{\frac{(q+1)(1+\sigma)}{q}}}=O(\epsilon^{1-\gamma})\\
&\|f_{2,\epsilon}(PW_{2,\delta})-f_{2,0}(PW_{2,\delta})\|_{L^{\frac{(p+1)(1+\sigma)}{p}}}=O(\epsilon^{1-\gamma}).
\end{align*}

\end{Lem}

\begin{proof}
By Lemma \ref{lemb2}, for $\sigma=0$, there exists some $\theta>0$ small such that
\begin{align}\label{b2}
\begin{split}
&\int_\Omega|f_{1,\epsilon}(PW_{1,\delta})-f_{1,0}(PW_{1,\delta})|^{\frac{q+1}q}
=\int_\Omega\Big|\alpha\epsilon |PW_{1,\delta}|^{q-1} PW_{1,\delta}\log|PW_{1,\delta}|+\epsilon^2\xi_{1,\epsilon}(PW_{1,\delta})\Big|^{\frac{q+1}q}\\
&\leq C\epsilon^{\frac{q+1}q}\int_\Omega  |PW_{1,\delta}|^{q+1}(\log|PW_{1,\delta}|)^{\frac{q+1}q}
+C\epsilon^{\frac{2(q+1)}q}\int_\Omega  |\xi_{1,\epsilon}(PW_{1,\delta})|^{\frac{q+1}q}\\
&\leq C\epsilon^{\frac{q+1}q}\Big(1+\int_\Omega  |PW_{1,\delta}|^{(q+1)(1+\frac{\theta}{n})}+\epsilon^{\frac{q+1}q}\int_\Omega  |PW_{1,\delta}|^{(q+1)(1+\frac{\theta}{n})}\Big).
\end{split}\end{align}
Note that by Lemma \ref{lemb3},
\begin{align*}
&\int_\Omega  |PW_{1,\delta}|^{(q+1)(1+\frac{\theta}{n})}\leq C\int_\Omega\Big(U_{e_n,\delta}^{(q+1)(1+\frac{\theta}{n})}+U_{-e_n,\delta}^{(q+1)(1+\frac{\theta}{n})}\Big)=O(\delta^{-\theta})
\end{align*}
which, combined with \eqref{b2}, gives
\begin{align*}
&\|f_{1,\epsilon}(PW_{1,\delta})-f_{1,0}(PW_{1,\delta})\|_{L^{\frac{(q+1)(1+\sigma)}{q}}}=O(\epsilon^{1-\gamma}).
\end{align*}
Similar estimate gives
\begin{align*}
&\|f_{2,\epsilon}(PW_{2,\delta})-f_{2,0}(PW_{2,\delta})\|_{L^{\frac{(p+1)(1+\sigma)}{p}}}=O(\epsilon^{1-\gamma}).
\end{align*}

The case of $\sigma>0$ can be handled by analogous steps.
\end{proof}

\begin{Lem}\label{lemb13}
For $d>0$ and $\delta=d\epsilon$. Then for every $\gamma\in(0,1)$, it holds that
\begin{align*}
&\|f'_{1,\epsilon}(PW_{1,\delta})-f'_{1,0}(PW_{1,\delta})\|_{L^{\frac{q+1}{q-1}}}=O(\epsilon^{1-\gamma})\\
&\|f'_{2,\epsilon}(PW_{2,\delta})-f'_{2,0}(PW_{2,\delta})\|_{L^{\frac{p+1}{p-1}}}=O(\epsilon^{1-\gamma}).
\end{align*}

\end{Lem}

\begin{proof}
By Lemma \ref{lemb2},
\begin{align*}
\begin{split}
&\Big(\int_\Omega|f'_{1,\epsilon}(PW_{1,\delta})-f'_{1,0}(PW_{1,\delta})|^{\frac{q+1}{q-1}}\Big)^{\frac{q-1}{q+1}}\\
&\leq C\epsilon\Big(\int_\Omega  |PW_{1,\delta}|^{q+1}+ |PW_{1,\delta}|^{q+1}(\log|PW_{1,\delta}|)^{\frac{q+1}{q-1}}\Big)^{\frac{q-1}{q+1}}\\&
\quad+C\epsilon^{2}\Big(\int_\Omega \Big( |PW_{1,\delta}|^{q+1}+ |PW_{1,\delta}|^{q+1+\beta\epsilon\frac{q+1}{q-1}}\Big)
\Big(\log|PW_{1,\delta}|+(\log|PW_{1,\delta}|)^2\Big)^{\frac{q+1}{q-1}}\Big)^{\frac{q-1}{q+1}}\\
&\leq C\epsilon^{1-\gamma}.
\end{split}\end{align*}

Similar estimate gives
\begin{align*}
\begin{split}
&\Big(\int_\Omega|f'_{2,\epsilon}(PW_{2,\delta})-f'_{2,0}(PW_{2,\delta})|^{\frac{p+1}{p-1}}\Big)^{\frac{p-1}{p+1}}\leq C\epsilon^{1-\gamma}.
\end{split}\end{align*}

\end{proof}

\medskip
\section{Integrals in the Leading Term}
\renewcommand{\theequation}{B.\arabic{equation}}

For the expansion of the energy in section 4, we need calculate some integrals.
\begin{Lem}\label{lemc1}
There hold that
\begin{align*}
\begin{split}
&\int_\Omega U^{q+1}_{e_n,\delta}=\int_\Omega U^{q+1}_{-e_n,\delta}=\frac{\mathcal A_1}2-\mathcal B_1\delta+o(\delta)\\
&\int_\Omega V^{p+1}_{e_n,\delta}=\int_\Omega V^{p+1}_{-e_n,\delta}=\frac{\mathcal A_2}2-\mathcal B_2\delta+o(\delta).
\end{split}\end{align*}

\end{Lem}

\begin{proof}
Recall that
$\Omega=B_1(0)$ and
\begin{align*}
\Omega_+=\Omega\cap B_{\frac12}(e_n),\ \ \ \Omega_-=\Omega\cap B_{\frac12}(-e_n).
\end{align*}
We write
\begin{align*}
\begin{split}
&\int_{\Omega_+} U^{q+1}_{e_n,\delta}=\frac12\int_{B_{\frac12}(e_n)} U^{q+1}_{e_n,\delta}-\int_\Theta  U^{q+1}_{e_n,\delta}
\end{split}\end{align*}
where $$\Theta=B^-_{\frac12}(e_n)\setminus\Omega_+,\ \ B^-_{\frac12}(e_n)=\{y\in B_{\frac12}(e_n):y_n<1\}.$$
Let  $\Xi$ as the projection of $\partial \Omega\cap B_{\frac12}(e_n)$ in the variables $x'=(x_1,\ldots,x_{n-1})$.
Then by calculation,  $\Theta=\{x'\in\Xi:\sqrt{1-|x'|^2}<x_n<1\}$ with $\Xi=\{x'\in\R^{n-1}:|x'|<\frac{\sqrt{15}}8\}$.
It is easy to see from Lemma \ref{lemasym} that
\begin{align*}
\begin{split}
 \frac12\int_{B_{\frac12}(e_n)} U^{q+1}_{e_n,\delta}&=\frac12\int_{\R^n} U^{q+1}_{e_n,\delta}-\frac12\int_{\R^n\setminus B_{\frac12}(e_n)} U^{q+1}_{e_n,\delta}
\\&=\frac12\int_{\R^n} U^{q+1}+\begin{cases}O(\delta^{(n-2)(q+1)-n}) &if\ p>\frac n{n-2},\\
O(\delta^{pn}) &if\ p<\frac n{n-2}.
\end{cases}
\end{split}\end{align*}
Moreover, by definition of $\mathcal B_1$,
\begin{align*}
\begin{split}
&\int_\Theta  U^{q+1}_{e_n,\delta}=\int_\Xi\int_{\sqrt{1-|x'|^2}}^1\delta^{-n}U^{q+1}\Big(\frac{|x-e_n|}\delta\Big)dx_ndx'
=\delta\mathcal B_1
\end{split}\end{align*}
In view of Lemma \ref{lemasym}, there exists some positive constants $a_1,b_1,a_1',b_1',a_2,b_2>0$ such that
\begin{align}\label{ab}
&\begin{cases}\dis\frac{a_1}{(1+|x|^2)^{\frac{n-2}2}}\leq U(x)\leq\frac{b_1}{(1+|x|^2)^{\frac{n-2}2}} &p>\frac n{n-2}\\
\dis\frac{a_1'}{(1+|x|^2)^{\frac{(n-2)p-2}2}}\leq U(x)\leq\frac{b_1'}{(1+|x|^2)^{\frac{(n-2)p-2}2}} &p<\frac n{n-2},
 \end{cases}\\
&\quad\frac{a_2}{(1+|x|^2)^{\frac{n-2}2}}\leq V(x)\leq \frac{b_2}{(1+|x|^2)^{\frac{n-2}2}}.
\end{align}

Inspired by the argument in \cite{arxiv}, we estimate that for $p>\frac n{n-2}$,
\begin{align*}
\begin{split}
&\mathcal B_1=\int_\Xi\int_{\sqrt{1-|x'|^2}}^1\delta^{-n-1}U^{q+1}\Big(\frac{|x-e_n|}\delta\Big)dx_ndx'\\
&\geq a_1^{q+1}\int_\Xi\int_{\sqrt{1-|x'|^2}}^1\delta^{-n-1}\frac1{\Big(1+\frac{|x-e_n|^2}{\delta^2}\Big)^{\frac{(n-2)(q+1)}2}}dx_ndx'\\
&= a_1^{q+1}\int_\Xi\int_{\sqrt{1-|x'|^2}}^1\frac{\delta^{(n-2)(q+1)-n-1}}{(\delta^2+|x-e_n|^2)^{\frac{(n-2)(q+1)}2}}dx_ndx'\\
&=\int_\Xi\frac{ a_1^{q+1}\delta^{(n-2)(q+1)-n-1}}{(\delta^2+|x'|^2)^{\frac{(n-2)(q+1)}2}}\int_0^{\frac{1-\sqrt{1-|x'|^2}}{\sqrt{\delta^2+|x'|^2}}}\frac1{
(1+y_n^2)^{\frac{(n-2)(q+1)}2}}dy_ndx'\\
&=\int_\Xi\frac{ a_1^{q+1}\delta^{(n-2)(q+1)-n-1}}{(\delta^2+|x'|^2)^{\frac{(n-2)(q+1)}2}}(\frac{|x'|^2}2+O(|x'|^4))dx'\\
&=\frac{ a_1^{q+1}}2\int_{\frac\Xi\delta}\frac{|y'|^2}{(1+|y'|^2)^{\frac{(n-2)(q+1)}2}}dy'+O(\delta)\int_{\frac\Xi\delta}\frac{|y'|^4}{(1+|y'|^2)^{\frac{(n-2)(q+1)}2}}dy'\\
&=\frac{ a_1^{q+1}}2\int_{\R^{n-1}}\frac{|y'|^2}{(1+|y'|^2)^{\frac{(n-2)(q+1)}2}}dy'+o(1)
\end{split}\end{align*}
where we  use the change of variables $y_n=\frac{1-x_n}{\sqrt{\delta^2+|x'|^2}}$ and so $\delta^2+|x-e_n|^2=(\delta^2+|x'|^2)(1+y_n^2)$.
By the same way, we have for the other direction that
\begin{align*}
\begin{split}
&\mathcal B_1\leq\frac{b_1^{q+1}}2\int_{\R^{n-1}}\frac{|y'|^2}{(1+|y'|^2)^{\frac{(n-2)(q+1)}2}}dy'+o(1).
\end{split}\end{align*}

While for $p<\frac n{n-2}$, by use of the same transformation,
\begin{align*}
\begin{split}
&\mathcal B_1=\int_\Xi\int_{\sqrt{1-|x'|^2}}^1\delta^{-n-1}U^{q+1}\Big(\frac{|x-e_n|}\delta\Big)dx_ndx'\\
&\geq { a_1'}^{q+1}\int_\Xi\int_{\sqrt{1-|x'|^2}}^1\delta^{-n-1}\frac1{\Big(1+\frac{|x-e_n|^2}{\delta^2}\Big)^{\frac{((n-2)p-2)(q+1)}2}}dx_ndx'\\
&= { a_1'}^{q+1}\int_\Xi\int_{\sqrt{1-|x'|^2}}^1\frac{\delta^{(p+1)n-n-1}}{(\delta^2+|x-e_n|^2)^{\frac{(p+1)n}2}}dx_ndx'\\
&=\int_\Xi\frac{{ a_1'}^{q+1}\delta^{(p+1)n-n-1}}{(\delta^2+|x'|^2)^{\frac{(p+1)n}2}}\int_0^{\frac{1-\sqrt{1-|x'|^2}}{\sqrt{\delta^2+|x'|^2}}}\frac1{
(1+y_n^2)^{\frac{(p+1)n}2}}dy_ndx'\\
&=\int_\Xi\frac{ {a_1'}^{q+1}\delta^{pn-1}}{(\delta^2+|x'|^2)^{\frac{(p+1)n}2}}(\frac{|x'|^2}2+O(|x'|^4))dx'\\
&=\frac{ { a_1'}^{q+1}}2\int_{\frac\Xi\delta}\frac{|y'|^2}{(1+|y'|^2)^{\frac{(p+1)n}2}}dy'+O(\delta)\int_{\frac\Xi\delta}\frac{|y'|^4}{(1+|y'|^2)^{\frac{(p+1)n}2}}dy'\\
&=\frac{{ a_1'}^{q+1}}2\int_{\R^{n-1}}\frac{|y'|^2}{(1+|y'|^2)^{\frac{(p+1)n}2}}dy'+o(1)
\end{split}\end{align*}
and also
\begin{align*}
\begin{split}
&\mathcal B_1\leq\frac{{b'_1}^{q+1}}2\int_{\R^{n-1}}\frac{|y'|^2}{(1+|y'|^2)^{\frac{(p+1)n}2}}dy'+o(1).
\end{split}\end{align*}

\medskip
To sum up, we have
\begin{align*}
\begin{split}
&\int_{\Omega} U^{q+1}_{e_n,\delta}=\int_{\Omega_+} U^{q+1}_{e_n,\delta}+o(\delta)=\frac{\mathcal A_1}2-\mathcal B_1\delta+o(\delta).
\end{split}\end{align*}
Similarly,\begin{align*}
\begin{split}
&\int_{\Omega} U^{q+1}_{-e_n,\delta}=\int_{\Omega_-} U^{q+1}_{-e_n,\delta}+o(\delta)=\frac{\mathcal A_1}2-\mathcal B_1\delta+o(\delta).
\end{split}\end{align*}

The same argument gives that\begin{align*}
\begin{split}
&\int_\Omega V^{p+1}_{e_n,\delta}=\int_\Omega V^{p+1}_{-e_n,\delta}=\frac{\mathcal A_2}2-\mathcal B_2\delta+o(\delta).
\end{split}\end{align*}
\end{proof}

\medskip

\begin{Lem}\label{lemb6}
As $\delta\rightarrow0$,
\begin{align*}
&\int_\Omega U_{e_n,\delta}U_{-e_n,\delta}^q=\int_\Omega U_{-e_n,\delta}U_{e_n,\delta}^q=o(\delta),\\
&\int_\Omega V_{e_n,\delta}V_{-e_n,\delta}^p=\int_\Omega V_{-e_n,\delta}V_{e_n,\delta}^p=o(\delta).
\end{align*}

\end{Lem}

\begin{proof}

Firstly,
\begin{align*}
&\int_\Omega U_{e_n,\delta}U_{-e_n,\delta}^q\leq C\int_\Omega\delta^{-n}U\Big(\frac{|x-e_n|}\delta\Big)U^q\Big(\frac{|x+e_n|}\delta\Big)\\
&\leq C\delta^{-n}\Big(\int_{\Omega_+}+\int_{\Omega_-}\Big)U\Big(\frac{|x-e_n|}\delta\Big)U^q\Big(\frac{|x+e_n|}\delta\Big)
+\begin{cases}O(\delta^{(n-2)(q+1)-n}) &if\ p>\frac n{n-2},\\
O(\delta^{pn}) &if\ p<\frac n{n-2}
\end{cases}\\
&=\begin{cases}O(\delta^{\frac{pq-1}{q+1}n}+\delta^{pn}+\delta^{(n-2)(q+1)-n}) &if\ p>\frac n{n-2},\\
O(\delta^{\frac{pq-1}{q+1}n}+\delta^{pn}+\delta^{pn}) &if\ p<\frac n{n-2},
\end{cases}
\end{align*}
where we use
\begin{align*}
&\delta^{-n}\int_{\Omega_+}U\Big(\frac{|x-e_n|}\delta\Big)U^q\Big(\frac{|x+e_n|}\delta\Big)
\leq C\int_{\Omega_+}\frac{\delta^{q((n-2)p-2)-n}}{(1+\frac{|x-e_n|}\delta)^{(n-2)p-2}}\\
&\leq C\int_{\frac{\Omega_+-e_n}{\delta}}\frac{\delta^{q((n-2)p-2)}}{(1+|y|)^{(n-2)p-2}}=C\delta^{q((n-2)p-2)}
\Big(1+\int_1^{\frac1{2\delta}}r^{n-1-((n-2)p-2)}dr\Big)\\&
=O(\delta^{pn})=o(\delta),
\end{align*}
and since $q((n-2)p-2)>n$ and  $\frac{pq-1}{q+1}n=\frac{q(p+1)}{q+1}n-n=(\frac{qn}{(q+1)(n-2)-n}-1)n=2(p+1)$,
\begin{align*}
&\delta^{-n}\int_{\Omega_-}U\Big(\frac{|x-e_n|}\delta\Big)U^q\Big(\frac{|x+e_n|}\delta\Big)
\leq C\int_{\Omega_-}\frac{\delta^{\frac{p+1}{q+1}n-n}}{(1+\frac{|x+e_n|}\delta)^{q((n-2)p-2)}}\\
&\leq C\int_{\frac{\Omega_-+e_n}{\delta}}\frac{\delta^{\frac{p+1}{q+1}n}}{(1+|y|)^{q((n-2)p-2)}}=O(\delta^{\frac{q(p+1)}{q+1}n-n})=O(\delta^{\frac{pq-1}{q+1}n})=o(\delta).
\end{align*}

\medskip

On the other hand,
\begin{align*}
&\int_\Omega V_{e_n,\delta}V_{-e_n,\delta}^p\leq C\int_\Omega\delta^{-n}V\Big(\frac{|x-e_n|}\delta\Big)V^p\Big(\frac{|x+e_n|}\delta\Big)\\
&\leq C\delta^{-n}\Big(\int_{\Omega_+}+\int_{\Omega_-}\Big)V\Big(\frac{|x-e_n|}\delta\Big)V^p\Big(\frac{|x+e_n|}\delta\Big)
+O(\delta^{(n-2)(p+1)-n})=o(\delta),
\end{align*}
where, in view of Remark \ref{remP'},  we use
\begin{align*}
&\delta^{-n}\int_{\Omega_+}V\Big(\frac{|x-e_n|}\delta\Big)V^p\Big(\frac{|x+e_n|}\delta\Big)
\leq C\int_{\Omega_+}\frac{\delta^{p(n-2)-n}}{(1+\frac{|x-e_n|}\delta)^{n-2}}\\
&\leq C\int_{\frac{\Omega_+-e_n}{\delta}}\frac{\delta^{p(n-2)}}{(1+|y|)^{n-2}}=C\delta^{p(n-2)}
\Big(1+\int_1^{\frac1{2\delta}}r^{n-1-(n-2)}dr\Big)=O(\delta^{(n-2)p-2 })=O(\delta^{\frac{p+1}{q+1}n})=o(\delta)
\end{align*}
and 
\begin{align*}
&\delta^{-n}\int_{\Omega_-}V\Big(\frac{|x-e_n|}\delta\Big)V^p\Big(\frac{|x+e_n|}\delta\Big)
\leq C\int_{\Omega_-}\frac{\delta^{-2}}{(1+\frac{|x+e_n|}\delta)^{p(n-2)}}\\
&\leq C\int_{\frac{\Omega_-+e_n}{\delta}}\frac{\delta^{n-2}}{(1+|y|)^{p(n-2)}}=O(\delta^{(n-2)p-2 })=o(\delta).
\end{align*}

Therefore, we deduce
\begin{align*}
&\int_\Omega V_{e_n,\delta}V_{-e_n,\delta}^p=o(\delta).
\end{align*}

\end{proof}
\medskip

\begin{Lem}\label{lemb8}
As $\delta\rightarrow0$,
\begin{align*}
&\delta^{-\frac n{q+1}+1}\int_\Omega\varphi_{1,0}(\frac{e_n-x}\delta)U_{e_n,\delta}^q
=\delta^{-\frac n{q+1}+1}\int_\Omega\varphi_{1,0}(\frac{e_n+x}\delta)U_{-e_n,\delta}^q=-\frac{\mathcal C_1}{2}\delta+o(\delta),\\
&\delta^{-\frac n{q+1}+1}\int_\Omega\varphi_{2,0}(\frac{e_n-x}\delta)V_{e_n,\delta}^p
=\delta^{-\frac n{q+1}+1}\int_\Omega\varphi_{2,0}(\frac{e_n+x}\delta)V_{-e_n,\delta}^p=-\frac{\mathcal C_2}{2}\delta+o(\delta).
\end{align*}

\end{Lem}
\begin{proof}
By a change of variables and from \eqref{phi0}
\begin{align*}
&\delta^{-\frac n{q+1}+1}\int_{\Omega_+}\varphi_{1,0}(\frac{e_n-x}\delta)U_{e_n,\delta}^q
=\delta\int_{\frac{\Omega_+-e_n}{\delta}}\varphi_{1,0}(-x)U^q(x)\\
&=\delta\int_{\R^{n}_+}\varphi_{1,0}(x)(-\Delta V)dx+o(\delta)
=\delta\int_{\partial\R^{n}_+}\partial_\nu\varphi_{1,0}(x)V(x')dx'+o(\delta)\\
&=\frac\delta2\int_{\partial\R^{n}_+}|x'|U'(x)\Big|_{x_n=0}V(x')dx'+o(\delta)=-\frac{\mathcal C_1}{2}\delta+o(\delta)
\end{align*}
where $\nu$ is the exterior unitary normal on $\partial\R^n_+$.

Similarly, \begin{align*}
&\delta^{-\frac n{p+1}+1}\int_{\Omega_+}\varphi_{2,0}(\frac{e_n-x}\delta)V_{e_n,\delta}^p
=\frac\delta2\int_{\partial\R^{n}_+}|x'|V'(x)\Big|_{x_n=0}U(x')dx'+o(\delta)=-\frac{\mathcal C_2}{2}\delta+o(\delta).
\end{align*}

\end{proof}

\medskip
\begin{Lem}\label{lemb7}As $\delta\rightarrow0$,
\begin{align*}
&\delta^{-\frac n{q+1}+1}\int_\Omega\varphi_{1,0}(\frac{e_n-x}\delta)U_{-e_n,\delta}^q
=\delta^{-\frac n{q+1}+1}\int_\Omega\varphi_{1,0}(\frac{e_n+x}\delta)U_{e_n,\delta}^q=o(\delta);\\
&\delta^{-\frac n{p+1}+1}\int_\Omega\varphi_{2,0}(\frac{e_n-x}\delta)V_{-e_n,\delta}^p
=\delta^{-\frac n{p+1}+1}\int_\Omega\varphi_{2,0}(\frac{e_n+x}\delta)V_{e_n,\delta}^p=o(\delta).
\end{align*}
\end{Lem}

\begin{proof}
First, for $p>\frac n{n-2}$,
\begin{align*}
&\delta^{-\frac n{q+1}+1}\int_{\Omega_+}\varphi_{1,0}(\frac{e_n-x}\delta)U_{-e_n,\delta}^q
\leq C\delta^{-\frac n{q+1}+1}\int_{\Omega_+}\varphi_{1,0}(\frac{e_n-x}\delta)\frac{\delta^{-\frac{qn}{q+1}}}{(1+\frac{|x+e_n|}{\delta})^{q(n-2)}}\\
&\leq C\delta^{-\frac n{q+1}+1+q(n-2)-\frac{qn}{q+1}}\int_{\Omega_+-e_n}\varphi_{1,0}(\frac{-x}\delta)
=C\delta^{q(n-2)+1}\int_{\frac{\Omega_+-e_n}\delta}\varphi_{1,0}(y)dy\\
&\leq C\delta^{q(n-2)+1}\int_{\frac{\Omega_+-e_n}\delta}\frac1{(1+|x|)^{n-3}}\leq C\delta^{q(n-2)+1}(1+\int_1^{\frac1{2\delta}}r^2dr)=O(\delta^{q(n-2)-2})=o(\delta).
\end{align*}
On the other hand,
\begin{align*}
&\delta^{-\frac n{q+1}+1}\int_{\Omega_-}\varphi_{1,0}(\frac{e_n-x}\delta)U_{-e_n,\delta}^q
\leq C\delta^{-\frac n{q+1}+1}\int_{\Omega_-+e_n}\varphi_{1,0}(\frac{2e_n-x}\delta)\frac{\delta^{-\frac{qn}{q+1}}}{(1+\frac{|x|}{\delta})^{q(n-2)}}\\
&\leq C\delta\int_{\frac{\Omega_-+e_n}\delta}\varphi_{1,0}(\frac{2e_n}\delta-x)\frac{1}{(1+|x|)^{q(n-2)}}
=O(\delta^{n-2})=o(\delta).
\end{align*}
So we can get
\begin{align*}
&\delta^{-\frac n{q+1}+1}\int_\Omega\varphi_{1,0}(\frac{e_n-x}\delta)U_{-e_n,\delta}^q
=\delta^{-\frac n{q+1}+1}\int_\Omega\varphi_{1,0}(\frac{e_n+x}\delta)U_{e_n,\delta}^q=o(\delta).
\end{align*}
Similarly, in this case,
\begin{align*}
&\delta^{-\frac n{p+1}+1}\int_\Omega\varphi_{2,0}(\frac{e_n-x}\delta)V_{-e_n,\delta}^p
=\delta^{-\frac n{p+1}+1}\int_\Omega\varphi_{2,0}(\frac{e_n+x}\delta)V_{e_n,\delta}^p=o(\delta).
\end{align*}

Second, for $p<\frac n{n-2}$,
\begin{align*}
&\delta^{-\frac n{q+1}+1}\int_{\Omega_+}\varphi_{1,0}(\frac{e_n-x}\delta)U_{-e_n,\delta}^q
\leq C\delta^{-\frac n{q+1}+1}\int_{\Omega_+}\varphi_{1,0}(\frac{e_n-x}\delta)
\frac{\delta^{-\frac{qn}{q+1}}}{(1+\frac{|x+e_n|}{\delta})^{q((n-2)p-2)}}\\
&\leq C\delta^{-\frac n{q+1}+1+q((n-2)p-2)-\frac{qn}{q+1}}\int_{\Omega_+-e_n}\varphi_{1,0}(\frac{-x}\delta)
=C\delta^{q((n-2)p-2)+1}\int_{\frac{\Omega_+-e_n}\delta}\varphi_{1,0}(y)dy\\
&\leq C\delta^{q((n-2)p-2)+1}\int_{\frac{\Omega_+-e_n}\delta}\frac1{(1+|x|)^{(n-2)p-3}}\leq C\delta^{q((n-2)p-2)+1}(1+\int_1^{\frac1{2\delta}}r^2dr)\\&=O(\delta^{pn}).
\end{align*}
On the other hand, since $q((n-2)p-2)>n$, we have
\begin{align*}
&\delta^{-\frac n{q+1}+1}\int_{\Omega_-}\varphi_{1,0}(\frac{e_n-x}\delta)U_{-e_n,\delta}^q
\leq C\delta^{-\frac n{q+1}+1}\int_{\Omega_-+e_n}\varphi_{1,0}(\frac{2e_n-x}\delta)\frac{\delta^{-\frac{qn}{q+1}}}{(1+\frac{|x|}{\delta})^{q(n-2)}}\\
&\leq C\delta\int_{\frac{\Omega_-+e_n}\delta}\varphi_{1,0}(\frac{2e_n}\delta-x)\frac{1}{(1+|x|)^{q((n-2)p-2)}}
=O(\delta^{(n-2)p-2}).
\end{align*}
In view of Remark \ref{remP'}, it holds that
\begin{align*}
&\delta^{-\frac n{q+1}+1}\int_\Omega\varphi_{1,0}(\frac{e_n-x}\delta)U_{-e_n,\delta}^q
=\delta^{-\frac n{q+1}+1}\int_\Omega\varphi_{1,0}(\frac{e_n+x}\delta)U_{e_n,\delta}^q=o(\delta).
\end{align*}

Similarly, in this case,
\begin{align*}
&\delta^{-\frac n{p+1}+1}\int_\Omega\varphi_{2,0}(\frac{e_n-x}\delta)V_{-e_n,\delta}^p
=\delta^{-\frac n{p+1}+1}\int_\Omega\varphi_{2,0}(\frac{e_n+x}\delta)V_{e_n,\delta}^p=o(\delta).
\end{align*}

\end{proof}
\medskip

 \noindent\textbf{Acknowledgments}
Guo was supported by NSFC grants (No.12271539,11771469).
Peng was supported by NSFC grants (No.11831009).

\end{document}